\documentclass[notitlepage, 12pt]{article}
\usepackage[nottoc]{tocbibind}
\usepackage{amsmath}
\usepackage{amssymb}
\usepackage{amsthm}
\usepackage{setspace}
\usepackage{comment}
\usepackage[all]{xy}
\usepackage{enumitem}
\usepackage{tabularx}
\usepackage{hyperref}

\DeclareMathOperator{\Ker}{Ker}
\DeclareMathOperator{\Hom}{Hom}
\DeclareMathOperator{\End}{End}
\DeclareMathOperator{\Gal}{Gal}
\DeclareMathOperator{\Aut}{Aut}

\DeclareMathOperator{\Span}{Span}
\DeclareMathOperator{\Img}{Im}

\DeclareMathOperator{\Inf}{Inf}
\DeclareMathOperator{\mult}{mult}

\DeclareMathOperator{\supp}{supp}

\theoremstyle{plain}
\newtheorem{Theorem}{Theorem}[section]
\newtheorem{Proposition}[Theorem]{Proposition}
\newtheorem{Corollary}[Theorem]{Corollary}
\newtheorem{Lemma}[Theorem]{Lemma}

\theoremstyle{definition}
\newtheorem{Definition}[Theorem]{Definition}
\newtheorem{Remark}[Theorem]{Remark}
\newtheorem{Example}[Theorem]{Example}
\newtheorem{Construction}[Theorem]{Construction}
\newtheorem{Notation}[Theorem]{Notation}

\newcommand{\arr}{\ar@{>>}}
\newcommand{\dotar}{\ar@{.>}}
\newcommand{\dotarr}{\ar@{.>>}}

\newcommand{\id}{\mathrm{id}}
\newcommand{\pr}{\mathrm{pr}}

\newcommand{\N}{\mathbb{N}}
\newcommand{\Z}{\mathbb{Z}}

\newcommand{\E}{\mathcal{E}}

\newcommand{\PP}{\mathcal{P}}
\newcommand{\T}{\mathcal{T}}

\newcommand{\st}{|\,}

\newcommand{\fhat}{\hat{f}}

\newcommand{\zetahat}{\hat{\zeta}}

\newcommand{\Lambdahat}{{\hat{\Lambda}}}

\newcommand{\Ahat}{{\hat{A}}}
\newcommand{\Bhat}{{\hat{B}}}
\newcommand{\Chat}{{\hat{C}}}

\newcommand{\Ghat}{{\hat{G}}}
\newcommand{\Hhat}{{\hat{H}}}

\newcommand{\Bbar}{{\bar{B}}}
\newcommand{\Cbar}{{\bar{C}}}

\newcommand{\Gbar}{{\bar{G}}}
\newcommand{\Hbar}{{\bar{H}}}
\newcommand{\Ibar}{{\bar{I}}}
\newcommand{\Kbar}{{\bar{K}}}
\newcommand{\Lbar}{{\bar{L}}}
\newcommand{\Mbar}{{\bar{M}}}
\newcommand{\Nbar}{{\bar{N}}}

\newcommand{\Ubar}{{\bar{U}}}
\newcommand{\Zbar}{{\bar{Z}}}

\newcommand{\pbar}{{\bar{p}}}
\newcommand{\prbar}{{\overline{\mathrm{pr}}}}
\newcommand{\pibar}{{\bar{\pi}}}
\newcommand{\etabar}{{\bar{\eta}}}

\newcommand{\rhobar}{{\bar{\rho}}}

\newcommand{\epsbar}{{\bar{\eps}}}

\def\phi{\varphi}
\newcommand{\calA}{\mathcal{A}}
\newcommand{\calB}{\mathcal{B}}
\newcommand{\calC}{\mathcal{C}}
\newcommand{\calE}{\mathcal{E}}
\newcommand{\calF}{\mathcal{F}}

\newcommand{\calH}{\mathcal{H}}

\newcommand{\calM}{\mathcal{M}}
\newcommand{\calN}{\mathcal{N}}

\newcommand{\eps}{\varepsilon}

\newcommand{\na}{\mathrm{na}}
\newcommand{\ab}{\mathrm{ab}}

\newcommand{\directsum}{\oplus}
\renewcommand{\iff}{\Leftrightarrow}
\renewcommand{\implies}{\Rightarrow}
\newcommand{\isom}{\cong}
\newcommand{\normal}{\triangleleft}
\newcommand{\shortarrow}{\raisebox{1pt}{$\rightarrow$}}
\newcommand{\pwr}[2]{\,{}^{#1}\kern-1pt {#2}\,}

\newcommand{\fprod}[2][G]{\underset{#2\phantom{\, {#1}}}{{\prod}_{#1}}}

\newcommand*{\defeq}{\mathrel{\vcenter{\baselineskip0.5ex \lineskiplimit0pt
                     \hbox{\scriptsize.}\hbox{\scriptsize.}}}%
                     =}
\newcommand{\PPP}{\PP/\kern-5pt\sim}
\newcommand{\deriv}{\bar\partial}
\newcommand{\onto}{\twoheadrightarrow}

\newcommand{\xonto}[2][]{%
  \xrightarrow[#1]{#2}\mathrel{\mkern-14mu}\rightarrow
}
\newcommand{\mle}{\preccurlyeq}

\makeatletter
\def\moverlay{\mathpalette\mov@rlay}
\def\mov@rlay#1#2{\leavevmode\vtop{%
   \baselineskip\z@skip \lineskiplimit-\maxdimen
   \ialign{\hfil$\m@th#1##$\hfil\cr#2\crcr}}}
\newcommand{\charfusion}[3][\mathord]{
    #1{\ifx#1\mathop\vphantom{#2}\fi
        \mathpalette\mov@rlay{#2\cr#3}
      }
    \ifx#1\mathop\expandafter\displaylimits\fi}
\makeatother

\newcommand{\dotcup}{\charfusion[\mathbin]{\cup}{\cdot}}
\newcommand{\bigdotcup}{\charfusion[\mathop]{\bigcup}{\cdot}}
\newcommand{\notetous}[1]{\textbf{[#1]}}

\newcommand{\subdemoinfo}[2]{\smallskip\noindent\textbf{#1:}
\emph{#2}\quad}

\numberwithin{equation}{section}

\begin{document}

\author{Dan Haran}
\title{Fundaments of epimorphisms of profinite groups
}

\maketitle
\thispagestyle{empty}

\begin{abstract}
We propose and develop a theory
that allows to characterize epimorphisms of profinite groups
in terms of indecomposable epimorphisms.
\end{abstract}

\clearpage
\pagenumbering{arabic}

\newpage
\section*{Introduction}

A central theme in Galois theory is
to understand a profinite group through
a series of its distinguished closed normal subgroups
(\cite{Efrat1},
\cite{Efrat2},
\cite{Efrat3},
\cite{Efrat-Minac},
\cite{Minac et al});
one could add to this topic
also the efforts to deduce properties of profinite groups
from their distinguished quotients
(\cite{Minac-Spira}, \cite{Minac-Tan}).
Although the research in this area is far from its peak,
sometimes even a more general and a more ambitious goal is required -
a characterization of an epimorphism
of profinite groups.
(This is more general in the sense that
a profinite group $G$ can be identified with the epimorphism
$G \onto 1$.)
In fact, 
the original motivation for this work
has been to describe 
a smallest embedding cover of a profinite group
and give a group-theoretic proof of its uniqueness
(\cite[Problem~36.2.25]{FJ}).

In this paper we provide a framework for a characterization
of epimorphisms,
the theory of \emph{fundaments}.

Given an epimorphism
$\pi \colon H \onto G$,
its \emph{fundament kernel} $M(\pi)$ 
is the intersection of the maximal normal subgroups of $H$ 
properly contained in $\Ker(\pi)$; 
quotienting by $M(\pi)$ yields the \emph{fundament} $\bar\pi$,
and $\pi$ is called \emph{fundamental} when $M(\pi)=1$.
Iterating this construction produces 
the \emph{fundament kernel series} 
$M_0(\pi)\ge M_1(\pi)\ge M_2(\pi)\ge\cdots$ 
with associated \emph{fundament series} of epimorphisms
$$
\cdots \xonto{\pi_3} G_2 \xonto{\pi_2} G_1 \xonto{\pi_1} G_0 = G,
$$
where $G_i = H/M_i(\pi)$ and 
each $\pi_{k}$ is the fundament of the preceding quotient $\pi_{k-1}$.
The intersection of all $M_i(\pi)$ is trivial,
so $\pi$ is recovered as the inverse limit of its successive fundaments.
Moreover,
two epimorphisms onto $G$
are isomorphic
if and only if
their fundament series are isomorphic.

Thus we are left with the task 
to classify fundamental epimorphisms,
up to isomorphism,
and to decide
when a series of fundamental epimorphisms
as above
is a fundament series of its inverse limit.
This is, essentially, what this paper is about.

It turns out that
fundamental epimorphisms
are fiber products of
(possibly infinitely many, even uncountably many)
indecomposable epimorphisms.
An epimorphism
$\phi \colon H \onto G$
is \emph{indecomposable}
if it is not an isomorphism
and every factorization
$H \xonto{\psi} E \xonto{\rho} G$
forces one of $\psi$ or $\rho$ to be an isomorphism.
Equivalently,
$\Ker \phi$ is 
a minimal normal subgroup of $H$
(and in particular finite).

Obviously,
two fiber products of indecomposable epimorphisms
onto the same group $G$
are isomorphic,
if
the cardinality of factors isomorphic to $\eta$
is the same in both of them,
for every indecomposable epimorphism $\eta$ onto $G$.
The converse is not true:
Let $C_n$ be the cyclic group of order $n$.
Then
$\eta_0 \colon C_2 \times C_2 \onto C_2$
and
$\eta_1 \colon C_4 \onto C_2$
are indecomposable.
The fiber product
of two copies of $\eta_1$
and
the fiber product
of $\eta_0$ and $\eta_1$
are both isomorphic
to the same epimorphism
$C_4 \times C_2 \onto C_2$.

Fortunately,
this problem arises only with 
indecomposable epimorphisms onto $G$
with the same abelian kernel,
which is then a simple $G$-module $A$.
The kernel of a fiber product of such epimorphisms
is the direct product of copies of $A$.
Sections~\ref{duality} develops a duality theory between
isotypic profinite semisimple modules of type $A$
and
vector spaces over the finite field $\End_G(A)$.
Then Section~\ref{second cohomology}
develops a duality theory between
fiber products of
indecomposable epimorphisms onto $G$
with kernel $A$
and
linear maps from these vector spaces into $H^2(G,A)$.
This allows us to classify the former by the latter.
Then,
in Section~\ref{fundaments},
we succeed to characterize fundamental epimorphisms
by invariants introduced there.

Throughout the whole paper
we use extensively
commutative squares
\begin{equation}\label{ex square}
\xymatrix@=20pt{
H 
\arr[rr]_{\theta} \arr[d]^{\pi}
&& H' \arr[d]^{\pi'}
\\
G \arr[rr]^{\theta_0} && G' \\
}
\end{equation}
of epimorphisms.
The key notions,
introduced in Section~\ref{section cartesian}
are \emph{cartesian squares}
---\eqref{ex square} is cartesian
if $H$ is the fiber product of $\theta_0$ and $\pi'$,
or, equivalently,
if $\theta$ maps $\Ker \pi$ bijectively onto $\Ker \pi'$---,
and the \emph{compactness} of cartesian squares,
---cartesian \eqref{ex square} is compact
if there is no proper subgroup $E$ of $H$
such that
$\pi(E) = G$ and $\theta(E) = H'$---, 
which later, in Section~\ref{mfp I},
extends to general fiber products.
We have also found very useful the notion of
a \emph{semi-cartesian} square
---\eqref{ex square} is semi-cartesian
if $\theta$ maps $\Ker \pi$ onto $\Ker \pi'$.

The theory of fundaments
characterizes epimorphisms.
But it
actually achieves more.
Given epimorphisms
$\pi \colon H \onto G$,
$\pi' \colon H' \onto G'$,
and
$\theta_0 \colon G \onto G'$,
we find the exact conditions
on the fundament series of $\pi$ and $\pi'$,
such that there exists an epimorphism
$\theta \colon H \onto H'$
with \eqref{ex square} semi-cartesian.

This generalization is essential
for an application of 
the theory developed here.
In the forthcoming work \cite{embed}
we give
a group-theoretic proof 
of the uniqueness of the smallest embedding cover.
Further applications will include
characterizations of projective and free profinite groups
via fundaments.

\bigskip

We fix the following notation and conventions
throughout the whole paper:

We use the word `family' for indexed family.

We use $\onto$ to denote an epimorphism of profinite groups.

An epimorphism $\eta \colon H \onto G$
is sometimes called
a \textbf{cover} (of $G$)
and sometimes 
an \textbf{extension} (of $G$).
Another epimorphism
$\pi \colon E \onto G$
is \textbf{dominated by $\eta$}
(or \textbf{$\eta$ dominates $\pi$}),
denoted $\pi \mle \eta$,
if there is
$\psi \colon H \onto E$
such that
$\eta = \pi \circ \psi$.
We say that
$\pi, \psi$ are \textbf{isomorphic},
denoted $\pi \isom_G \psi$,
if there is an isomorphism
$\psi \colon H \to E$
such that
$\eta = \pi \circ \psi$.

All groups in this article
are profinite,
and homomorphisms between groups are continuous;
a subgroup of a group is a closed subgroup.

We denote by $\prod_{i \in I} K_i$
the direct product of a family
of groups
$(K_i)_{i \in I}$.

\newpage 

\section{Indecomposable epimorphisms}

\begin{Definition}\label{indecomposable}
An epimorphism $\phi\colon H \onto G$ is
\textbf{indecomposable}
if it is not an isomorphism and
whenever $\psi \colon H \onto E$ and $\rho \colon E \onto G$
satisfy
$\phi = \rho\circ\psi$,
then either $\psi$ or $\rho$ is an isomorphism.
\end{Definition}

Let $H$ be a profinite group.
For $M \normal H$ let 
$\calN_M(H)$ denote the set of
the maximal ones among
all normal subgroups of $H$
strictly contained in $M$.

\begin{Lemma}\label{trivial}
Let $N \le M$ be closed normal subgroups of a profinite group $H$.
The following are equivalent:
\begin{itemize}
\item[(a)]
$N \in \calN_M(H)$;
\item[(b)]
$M/N$ is a minimal normal subgroup of $H/N$
(in particular, $M/N \ne 1$);
\item[(c)]
the canonical map
$H/N \onto H/M$
is indecomposable.
\end{itemize}
If (a), (b), or (c) holds,
then $N$ is open in $M$,
i.e., $M/N$ is a finite group.
\end{Lemma}

\begin{proof}
Indeed,
\begin{align*}
(a) \quad
&\iff \quad
N < M \quad \quad \& \quad
(\not\exists L \normal H)\ N < L < M 
\\
&\iff \quad
M/N \ne 1 \quad \& \quad
(\not\exists \Lbar \normal H/N)\ 1 < \Lbar < M/N 
\quad \iff  \quad (b).
\end{align*}

(a) $\implies$ (c):
If $H/N \onto H/M$ is a composition of two epimorphisms,
then, without loss of generality,
they are the canonical maps
$H/N \onto H/L$ and $H/L \onto H/M$,
where $L \normal H$ and $N \le L \le M$.
By (a), either $N = L$, in which case $H/N \onto H/L$ is an isomorphism,
or $L = M$, in which case $H/L \onto H/M$ is an isomorphism.

(c) $\implies$ (a):
Let $L \normal H$ with $N \le L \le M$.
Then $H/N \onto H/M$ is the composition of 
$H/N \onto H/L$ and $H/L \onto H/M$.
If the former is an isomorphism, then $N = L$;
if the latter is an isomorphism, then $L = M$.
This gives (a).

Assume (b) and put $\Hbar = H/N$ and $\Mbar = M/N$.
As $\Mbar \ne 1$,
there is an open $\Ubar \normal \Hbar$ such that
$\Ubar \cap \Mbar < \Mbar$.
As $\Ubar \cap \Mbar \normal \Hbar$,
by (b) we have $\Ubar \cap \Mbar = 1$.
Thus $1$ is open in $\Mbar$,
that is, $\Mbar$ is finite.
\end{proof}

\begin{Remark}\label{rem: minimal normal subgroup}
Let
$\phi\colon H \onto G$
be an epimorphism.
It follows from Lemma~\ref{trivial}
(with $N = 1$) that:

(a)
$\phi$ is indecomposable
if and only if
$\Ker \phi$ is a minimal normal subgroup of $H$;

(b)
If $\phi$ is indecomposable,
then $\Ker \phi$ is finite.

(c)
A minimal normal subgroup of $H$ is finite.

\noindent
Furthermore

(d)
If $1 \ne M \normal H$, then
$\calN_M(H) \ne \emptyset$.

Indeed,
as $M \ne 1$,
by \cite[Remark 1.2.1(c)]{FJ}
there is an open $U \normal H$ such that
$N_0 \defeq M \cap U < M$.
Thus,
$N_0 \normal H$ is a proper subgroup of $M$.
Since $N_0$ is open in $M$,
there are only finitely many $N \normal H$
with $N_0 \le N < M$.
A maximal one among them is in $\calN_M(H)$.

(e)
However,
by (c),
non-trivial torsion free profinite groups
do not have minimal normal subgroups.
\end{Remark}

\begin{Lemma}\label{image and preimage}
Let $\pi \colon H \onto \Hbar$ be an epimorphism.
Let $M \normal H$ and $\Mbar = \pi(M)$.
\begin{enumerate}
\item[(a)]
If $M$ is a minimal normal subgroup of $H$
and $\Mbar \neq 1$,
then $\Mbar$ is a minimal normal subgroup of $\Hbar$
and
$\pi|_{M} \colon M \onto \Mbar$ is an isomorphism.
\item[(b)]
Suppose that $M \cap \Ker \pi = 1$.
If $\Mbar$ is a minimal normal subgroup of $\Hbar$,
then $M$ is a minimal normal subgroup of $H$.
\item[(c)]
Let $\Nbar \in \calN_\Mbar(\Hbar)$.
Then $N \defeq \pi^{-1}(\Nbar) \cap M$
is in $\calN_M(H)$ and $\pi(N) = \Nbar$.
\end{enumerate}
\end{Lemma}

\begin{proof}
Let $\rho \colon M \onto \Mbar$ be the restriction of $\pi$.

(a)
Let $1 \le \Lbar < \Mbar$ be a normal subgroup of $\Hbar$.
Then $\pi^{-1}(\Lbar) \normal H$,
hence $L \defeq \pi^{-1}(\Lbar) \cap M \normal H$.
But $L = \rho^{-1}(\Lbar)$.
Hence $L < M$.
As $M$ is minimal and $1 \le L < M$, we have $L = 1$,
hence $\rho$ is injective and
$\Lbar = \rho(L) = 1$.
Thus $\Mbar$ is minimal.

(b)
Let $1 < L \leq M$ be a normal subgroup of $H$.
Then $\Lbar =\pi(L)$ is normal in $\Hbar$ and $\Lbar \leq \Mbar$.
By assumption, $\rho \colon M \onto \Mbar$ is an isomorphism,
hence $\Lbar = \rho(L)$ is not trivial.
As $\Mbar$ is minimal, we have $\Lbar = \Mbar$.
Apply $\rho^{-1}$ to get $L = M$.

(c)
Let $N_0 = \pi^{-1}(\Nbar)$.
Then $N_0 \normal H$,
hence $N = N_0 \cap M \normal H$.
But $N = \rho^{-1}(\Nbar)$,
hence $\pi(N) = \rho(N) = \Nbar < \Mbar$,
whence $N < M$,
and $N$ is open in $M$,
because $\Nbar$ is open in $\Mbar$.
Thus,
$N \normal H$ is a proper subgroup of $M$.
If there were $L \normal H$ such that $N < L < M$,
then $\Lbar \defeq \rho(L) = \pi(L) \normal \Hbar$
and $\Nbar < \Lbar < \Mbar$, 
because $\Ker \rho \le N$,
a contradiction.
Thus $N \in \calN_M(H)$.
\end{proof}

\begin{Remark}\label{product of simples}
Let $M \ne 1$ be a minimal normal subgroup of a profinite group $H$.
By Remark~\ref{rem: minimal normal subgroup}(c)
there is an open $U \normal H$ such that $U \cap M = 1$.
Then $\Hbar = H/U$ is a finite group.
By Lemma~\ref{image and preimage}(a),
the image $\Mbar$ of $M$ in $\Hbar$
is a minimal normal subgroup of $\Hbar$.
Therefore (\cite[Theorem 4.3A]{DM})
$\Mbar$ is the direct product of finite simple groups
conjugate to each other in $\Hbar$.
As the quotient map $H \onto \Hbar$ restricts to an isomorphism
$M \to \Mbar$,\ 
$M$ is the direct product of finite simple groups
conjugate to each other in $H$.
%
\end{Remark}

Let $G$ be a profinite group.
A finite $G$-module $A \ne 0$
is \textbf{simple}
if $A$ has no submodule $B$
with $0 \subsetneqq B \subsetneqq A$.

Let
$\phi\colon H \onto G$
be an epimorphism of profinite groups
and let $K$ be its kernel.
If $K$ is abelian,
it is customary to view $K$ as a $G$-module,
with $G$-action induced from the conjugation in $H$;
if $L \le K$,
then $L$ is a $G$-submodule of $K$
if and only if
$L$ is a normal subgroup of $H$ contained in $K$.
One can generalize this to a non-abelian $K$
(but abelian $L$):

\begin{Lemma}\label{action}
Let
$\phi\colon H \onto G$
be an epimorphism of profinite groups
and let 
$L$ be a normal subgroup of $H$ contained in $Z(\Ker \phi)$.
Then
\begin{itemize}
\item[(a)]
The restriction of the conjugation in $H$ to $L$
defines a homomorphism
$\hat\rho \colon H \to \Aut(L)$
such that
$\hat\rho(\Ker \phi) = 1$.
Hence,
$\hat\rho$ induces a homomorphism
$\rho \colon G \to \Aut(L)$,
which turns $L$ into a $G$-module.
The action is given by
${}^g x = h x h^{-1}$,
where $g \in G$, $x \in L$, and
$h$ is an arbitrary element of $H$ such that $\phi(h) = g$.
\item[(b)]
$A \le Z(\Ker \phi)$
is a minimal normal subgroup of $H$
if and only if
$A$ is a simple $G$-module.
If this is the case,
then $A$ is finite.
\item[(c)]
If $A = \Ker \phi$ is abelian,
then
$\phi$ is indecomposable
if and only if
$A$ is a simple $G$-module.
\item[(d)]
Let
$\phi'\colon H' \onto G$
and
$\eps \colon H \onto H'$ 
be epimorphisms
such that
$\phi' \circ \eps = \phi$.
Then
$G$ acts on $\eps(L)$ as a subgroup of $Z(\Ker \phi')$
and
$\eps|_{L} \colon L \onto \eps(L)$
is $G$-equivariant.
\item[(e)]
Suppose there is a commutative diagram
\begin{equation*}
\xymatrix{
H \arr[r]_{\eps} \arr[d]^{\phi} & H' \arr[d]^{\phi'} \\
G \arr[r]^{\pi} & G' \rlap{.}
}
\end{equation*}
such that
$\eps(\Ker \phi) = \Ker \phi'$.
Then
$G'$ acts on $\eps(L)$ as a subgroup of $Z(\Ker \phi')$
via $\phi'$, as in (a),
so that
${}^{\pi(g)}(\eps(x)) = \eps({}^g x)$,
for all $x \in L$ and $g \in G$.
\end{itemize}
\end{Lemma}

\begin{proof}
Let $K = \Ker \phi$.

(a)
As $L \le Z(K)$, we have $[L,K] = 1$, and hence
$\hat\rho(K) = 1$.

(b)
Obviously,
$A \le Z(K)$
is a $G$-module
if and only if
$A$ is a normal subgroup of $H$.
Applying the attribute "minimal"
to both notions 
we get the equivalence.
The  finiteness follows from
Remark~\ref{rem: minimal normal subgroup}(c).

(c)
This follows from (b) by
Remark~\ref{rem: minimal normal subgroup}(a).

(e)
Let $K' = \Ker \phi'$.
We have
$\eps(K) = K'$,
hence
$\eps(L)$ is a normal subgroup of $H'$
contained in $Z(K')$.
So $G'$ acts on it, by (a).
Let 
$g \in G$ and $x \in L$, and choose
$h \in H$ such that $\phi(h) = g$.
Then $\phi'(\eps(h)) = \pi(g)$,
hence
${}^{\pi(g)}(\eps(x)) = \eps(h) \eps(x) (\eps(h))^{-1} = \eps(hxh^{-1})
= \eps({}^g x)$.

(d)
This follows from (e) with $G' = G$ and $\pi = \id_G$.
Notice that
$\eps(K) = \Ker \phi'$.
\end{proof}

\section{Cartesian squares}\label{section cartesian}

\begin{Definition}\label{cartesian}
Consider a commutative diagram
\begin{equation}\label{cartesian square diagram}
\xymatrix{
H \arr[r]_\eta \arr[d]^\beta & G \arr[d]^\phi \\
B \arr[r]^\alpha & A \\
}
\end{equation}
Let $p = \alpha \circ \beta = \phi \circ \eta$
and put
$K_1 = \Ker \beta$,
$K_2 = \Ker \eta$,
and
$K = \Ker p$.
We call \eqref{cartesian square diagram}
a \textbf{cartesian square}
(\cite[Definition 25.2.2]{FJ})
if it satisfies the following equivalent conditions:
\begin{itemize}
\item[(a)]
Up to an isomorphism,
$H$ is the fiber product
$B \times_{A} G$ of $\alpha$ and $\phi$,
with coordinate projections $\beta$ and $\eta$.
\item[(b)]
Given homomorphisms
$\beta' \colon E \to B$
and
$\eta' \colon E \to G$
such that
$\alpha \circ \beta' = \phi \circ \eta'$,
there is a unique homomorphism
$\eps \colon E \to H$
such that
$\beta' = \beta \circ \eps$
and 
$\eta' = \eta \circ \eps$.
\item[(c)]
$K = K_1 \times K_2$,
that is,
$K_1 \cap K_2 = 1$
and 
$K = K_1 K_2$.
\item[(d)]
$\beta$ maps $\Ker \eta$ isomorphically onto $\Ker \alpha$.
\item[(e)]
For all $b \in B$, $g \in G$ 
such that
$\alpha(b) = \phi(g)$
there is a unique $h \in H$
such that
$\beta(h) = b$ and $\eta(h) = g$.
\end{itemize}
The equivalence of the conditions is
\cite[Proposition 25.2.1,
Lemma 25.2.4,
and
Lemma 25.2.5]{FJ},
except for
(e) $\implies$ (c),
which is trivial,
and
\newline\noindent
(d) $\implies$ (c):
As $\beta|_{K_2}$ is injective,
$K_1 \cap K_2 = 1$.
As $\beta(K_2) = \Ker \alpha$,
we have
$K = \Ker(\alpha \circ \beta) = \beta^{-1}(\Ker \alpha) = 
(\Ker \beta) K_2 = K_1 K_2$.
\end{Definition}

 \medskip

\begin{Corollary}\label{cartesian indecomposable}
In a cartesian square \eqref{cartesian square diagram}
$\alpha$ is indecomposable
if and only if
$\eta$ is indecomposable.
\end{Corollary}

\begin{proof}
This follows from
(a) and (b) of
Lemma~\ref{image and preimage}
by Remark~\ref{rem: minimal normal subgroup}(a).
\end{proof}

\begin{Lemma}\label{lattice}
Let \eqref{cartesian square diagram} be a cartesian square.
Then $\beta$ maps
the lattice of normal subgroups of $H$ contained in $\Ker \eta$
isomorphically onto
the lattice of normal subgroups of $B$ contained in $\Ker \alpha$.
\end{Lemma}

\begin{proof}
The restriction $\delta \colon \Ker \eta \to \Ker \alpha$
of $\beta$
is an isomorphism, so it maps
the lattice of subgroups of $\Ker \eta$
isomorphically onto
the lattice of subgroups of $\Ker \alpha$.
If $L \le \Ker \eta$ is normal in $H$,
then $\delta(L) = \beta(L)$ is normal in $B$.
Conversely, if $L \le \Ker \alpha$ is normal in $B$,
then $\delta^{-1}(L) = \beta^{-1}(L) \cap \Ker \eta$ is normal in $H$.
\end{proof}

We will need also a weakening of the cartesian property:

\begin{Definition}\label{semi-cartesian}
A commutative diagram
\eqref{cartesian square diagram}
is a \textbf{semi-cartesian square}
if it satisfies the following equivalent properties:
\begin{itemize}
\item[(a)]
$\Ker(\phi \circ \eta) = (\Ker \beta) (\Ker(\phi)$.
\item[(b)]
$\beta(\Ker \eta) = \Ker \alpha$.
\item[(c)]
$\eta(\Ker \beta) = \Ker \phi$.
\item[(d)]
(Push-out)
Given homomorphisms
$\alpha' \colon B \to A'$
and
$\phi' \colon G \to A'$
such that
$\alpha' \circ \beta = \phi' \circ \eta$,
there is a unique homomorphism
$\lambda \colon A \to A'$
such that
$\alpha' = \lambda \circ \alpha$
and 
$\phi' = \lambda \circ \phi$.
\begin{equation}\label{push-out}
\xymatrix@=9pt{
H \arr[rr]_\eta \arr[dd]^{\beta} && G \arr[dd]_{\phi} \ar@/^1pc/[rddd]^{\phi'}
\\
\\
B \arr[rr]^\alpha \ar@/_1pc/[rrrd]_{\alpha'}
&& A \dotar[rd]^(.4){\lambda} \\
&&& A'
}
\end{equation}
\item[(e)]
If
epimorphisms
$\alpha'$,
$\phi'$,
and
$\pi$
make the following diagram commute
\begin{equation}\label{push-in}
\xymatrix@=10pt{
H \arr[dd]_{\beta} \arr[rr]^{\eta} 
&& G \arr[dd]^{\phi}
\labelmargin-{3pt}
\arr[ld]_{\phi'} 
\labelmargin+{3pt}
\\
& A' \arr[rd]^(.40){\pi}
\\
\labelmargin-{3pt}
B  \arr[ru]^{\alpha'}
\labelmargin+{3pt}
\arr[rr]_{\alpha} && A
}
\end{equation}
then $\pi$ is an isomorphism.
\item[(f)]
For all $b \in B$, $g \in G$ 
such that
$\alpha(b) = \phi(g)$
there is $h \in H$
such that
$\beta(h) = b$ and $\eta(h) = g$.
\end{itemize}
\end{Definition}

\begin{proof}[Proof of the equivalence of the properties]
Put $p = \phi \circ \eta = \alpha \circ \beta$.
We may assume that
$B = H/K_1$,
$G = H/K_2$,
$A = H/K$,
where $K_1, K_2 \le K$; \
$\beta$, $\eta$, and $p \colon H \onto H/K$
are the quotient maps;
and
$\alpha, \phi$
are the maps induced from
$\eta, \beta$,
respectively.

(a) $\iff$ (b):
$\beta(\Ker \eta) = \beta(K_2) = K_1K_2/K_1$
and
$\Ker \alpha = K/K_1$.
This gives the equivalence.

(a) $\iff$ (c):
$\eta(\Ker \beta) = \eta(K_1) = K_1K_2/K_2$
and
$\Ker \phi = K/K_1$.
This gives the equivalence.

(a) $\implies$ (d):
The uniqueness of $\lambda$ follows already from 
$\lambda \circ \phi = \phi'$.
Let 
$p' = \alpha' \circ \beta = \phi' \circ \eta$.
We have $K_1, K_2 \le \Ker p'$,
hence $\Ker p = K_1 K_2 \le \Ker p'$.
So
there is $\lambda \colon A \to A'$
such that
$\lambda \circ p = p'$.
Thus
$\lambda \circ \alpha \circ \beta = \alpha' \circ \beta$,
whence
$\lambda \circ \alpha = \alpha'$.
Similarly
$\lambda \circ \phi = \phi'$.

(d) $\implies$ (a):
We have shown in `(a) $\implies$ (d)'
that if we replace 
$\alpha$, $\phi$
by
$H/K_1 \to H/K_1K_2$, $H/K_2 \to H/K_1K_2$, 
then the square has the universal property (d).
This universal property implies 
the uniqueness of
$\alpha$, $\phi$ in \eqref{cartesian square diagram},
up to an isomorphism.
Hence $K_1K_2 = K$.

(a) $\implies$ (e):
We may assume that
$A' = H/L$,
where $K_1,K_2 \le L \le K$,
and
$\alpha', \phi',\pi$
are the quotient maps
$H/K_1 \onto H/L$,
$H/K_2 \onto H/L$,
and
$H/L \onto H/K$,
respectively.
Then $K_1 K_2 \le L \le K$.
By (a),
$L = K$,
hence $\pi$ is an isomorphism.

(e) $\implies$ (a):
Let $A' = H/K_1 K_2$
and let
$\alpha' \colon H/K_1 \onto H/K_1 K_2$,
$\phi' \colon H/K_2 \onto H/K_1 K_2$,
$\pi \colon H/K_1 K_2 \onto H/K$
be the quotient maps.
As $\pi$ is an isomorphism,
$K_1 K_2 = K$.

(c) $\implies$ (f):
Choose $h' \in H$ such that $\beta(h') = b$
and let $g' = \eta(h')$.
Then $\phi(g') = \phi(\eta(h')) = \alpha(b) = \phi(g)$,
hence there is $g'' \in \Ker \phi$
such that $g = g' g''$.
By (c) there is $h'' \in \Ker \beta$ such that
$\eta(h'') = g''$.
Then $\eta(h' h'') = g$
and $\beta(h' h'') = b$.

(f) $\implies$ (c):
Clearly $\eta(\Ker \beta) \le \Ker \phi$.
Conversely,
let $g \in \Ker \phi$.
Then $\beta(1) = \phi(g)$.
By (f) there is $h \in H$ 
such that
$h \in \Ker \beta$ and $\eta(h) = g$.
\end{proof}

\begin{Remark}\label{iso is semi}
Clearly,
if $\alpha$ is an isomorphism,
then \eqref{cartesian square diagram}
is semi-cartesian.
\end{Remark}

\begin{Corollary}\label{indecomposable semi}
Let \eqref{cartesian square diagram} be a commutative diagram
with $\alpha$ indecomposable.
The following are equivalent:
\begin{itemize}
\item[(a)]
\eqref{cartesian square diagram}
is a semi-cartesian square.
\item[(b)]
There is no $\gamma \colon G \onto B$
such that
$\gamma \circ \eta = \beta$.
\item[(c)]
$\Ker \eta \not\subseteq \Ker \beta$.
\end{itemize}
\end{Corollary}

\begin{proof}
We prove the equivalence of the negations of these assertions:

($\neg$a) $\implies$ ($\neg$b):
There is a commutative diagram \eqref{push-in} 
in which $\pi$ is not an isomorphism.
As $\alpha$ is indecomposable,
$\alpha'$ is an isomorphism.
Then $\gamma := (\alpha')^{-1} \circ \phi' \colon G \onto B$
satisfies
$\gamma \circ \eta = \beta$.

($\neg$b) $\implies$ ($\neg$a):
If there is
$\gamma \colon G \onto B$
such that
$\gamma \circ \eta = \beta$,
put
$A' = B$,
$\alpha' = \id_B$,
$\phi' = \gamma$,
and
$\pi = \alpha$.
Then \eqref{push-in} commutes,
but $\pi$ is not an isomorphism.
So
\eqref{cartesian square diagram}
is not semi-cartesian,
by Definition~\ref{semi-cartesian}(e).

($\neg$b) $\iff$ ($\neg$a):
By the first isomorphism theorem,
there is 
$\gamma \colon G \onto B$
such that
$\gamma \circ \eta = \beta$,
if and only if
$\Ker \eta \subseteq \Ker \beta$.
\end{proof}

\begin{Corollary}\label{twist semi}
Suppose that $\phi$
in a semi-cartesian square
\eqref{cartesian square diagram}
is the composition
$\phi = \phi_2 \circ \phi_1$
of two epimorphisms.
Then 
($\phi_1 \circ \eta, \phi_2, \beta, \alpha)$
is a semi-cartesian square.
\end{Corollary}

\begin{proof}
By assumption,
$(\Ker \beta)(\Ker \eta) = \Ker (\alpha \circ \beta)$.
But
$$
\Ker \eta \le \Ker(\phi_1 \circ \eta)
\le \Ker(\phi \circ \eta) = \Ker (\alpha \circ \beta),
$$
hence
$(\Ker \beta)(\Ker(\phi_1 \circ \eta)) = \Ker (\alpha \circ \beta)$.
\end{proof}

\begin{Lemma}\label{expand}
Consider two commutative diagrams

\noindent
\begin{minipage}{.495\linewidth}
\begin{equation}\label{semi1}
\xymatrix@=24pt{
H' \arr[r]_{\eta'} \arr[d]^{\beta'} & G \arr[d]^\phi \\
B \arr[r]^\alpha & A \\
}
\end{equation}
\end{minipage}
\begin{minipage}{.495\linewidth}
\begin{equation}\label{semi2}
\xymatrix@=24pt{
H \arr[r]_\eta \arr[d]^\beta & G \arr[d]^\phi \\
B \arr[r]^\alpha & A \\
}
\end{equation}
\end{minipage}
\newline
and let
$\psi \colon H' \to H$
be a homomorphism
such that
$\beta' = \beta \circ \psi$
and
$\eta' = \eta \circ \psi$.
\begin{itemize}
\item[(a)]
If $\psi$ is surjective, then
\eqref{semi1} is semi-cartesian
if and only if
\eqref{semi2} is semi-cartesian.
\item[(b)]
If
\eqref{semi1} is semi-cartesian
and
\eqref{semi2} is cartesian,
then $\psi$ is surjective.
\end{itemize}
\end{Lemma}

\begin{proof}
(a)
The universal property (d) of Definition~\ref{semi-cartesian}
holds for
\eqref{semi1}
if and only if
it holds for
\eqref{semi2}.

(b)
Let $h \in H$.
Then
$\alpha(\beta(h)) = \phi(\eta(h))$.
By Definition~\ref{semi-cartesian}(f)
there is
$h' \in H'$
such that
$\beta'(h') = \beta(h)$
and
$\eta'(h') = \eta(h)$,
that is,
$\beta(\psi(h')) = \beta(h)$
and
$\eta(\psi(h')) = \eta(h)$.
By Definition~\ref{cartesian}(e),
$\psi(h') = h$.
\end{proof}

\begin{Definition}\label{compact cartesian square}
A cartesian square
of epimorphisms of profinite groups
\eqref{cartesian square diagram}
is \textbf{compact} if
it satisfies the following equivalent properties:
\begin{itemize}
\item[(a)]
No proper closed subgroup of $H$ is mapped
by $\beta$ onto $B$
and by $\eta$ onto $G$.
\item[(b)]
If $E$ is a profinite group
and $\beta' \colon E \onto B$ and $\eta' \colon E \onto G$
satisfy
$\alpha \circ \beta' = \phi \circ \eta'$,
then the unique homomorphism
$\psi \colon E \to H$ such that
$\beta \circ \psi = \beta'$ and $\eta \circ \psi = \eta'$
is surjective.
\item[(c)]
If
$\alpha' \colon B \onto A'$,
$\phi' \colon G \onto A'$,
and
$\pi \colon A' \onto A$
satisfy
$\pi \circ \alpha' = \alpha$
and
$\pi \circ \phi' = \phi$,
then $\pi$ is an isomorphism.
\item[(d)]
There are no
$\alpha' \colon B \onto A'$,
$\phi' \colon G \onto A'$,
and an indecomposable
$\pi \colon A' \onto A$
such that
$\pi \circ \alpha' = \alpha$
and
$\pi \circ \phi' = \phi$.
\end{itemize}
\end{Definition}

\begin{proof}[Proof of the equivalence of the properties]

~\newline\indent
(a) $\implies$ (b):
Let $H_0 = \psi(E)$.
Then $\beta(H_0) = \beta'(E) = B$
and $\eta(H_0) = \eta'(E) = G$,
hence, by (a), $H_0 = H$.

(b) $\implies$ (a):
Let $E \le H$ such that
$\beta(E) = B$
and $\eta(E) = G$.
Let $\beta' = \beta|_{B}$,
$\eta' = \eta|_{E}$,
and
let $\psi \colon E \to H$ be the inclusion.
Then
$\beta \circ \psi = \beta'$ and $\eta \circ \psi = \eta'$,
hence, by (b), $\psi$ is surjective.
Thus $E = H$.

(b) $\implies$ (c):
Lemma~\ref{expand}(a).
(c) $\implies$ (b):
Lemma~\ref{expand}(b).

(c) $\implies$ (d):
Clear, because indecomposable epimorphism is not an isomorphism.

(d) $\implies$ (c):
Let $\alpha',\phi', \pi$
such that
$\pi \circ \alpha' = \alpha$
and
$\pi \circ \phi' = \phi$.
If $\pi$ is not an isomorphism,
then it is the composition of two epimorphisms
$\pi = \pi'' \circ \pi'$,
with $\pi''$ indecomposable,
hence
$\pi'' \circ (\pi' \circ \alpha') = \alpha$
and
$\pi'' \circ (\pi' \circ \phi') = \phi$,
a contradiction to (d).
So, $\pi$ is an isomorphism.
\end{proof}

\begin{Remark}\label{square trivialities}
(\cite[Lemma 2.3]{FH})
In the following commutative diagram
of epimorphisms with three squares
(the left one, the right one, and \eqref{cartesian square diagram})
\begin{equation}\label{two squares}
\xymatrix@=24pt{
H
\arr[r]_{\zeta} \arr[d]^{\beta}
\arr@/^1pc/[rr]^{\eta}
& H' \arr[r]_{\eta'} \arr[d]^{\beta'}
& G \arr[d]^{\phi} 
\\
B
\arr@/_1pc/[rr]_{\alpha}
\arr[r]^{\gamma} & B' \arr[r]^{\alpha'} & A \\
}
\end{equation}
\begin{itemize}
\item[(a)]
If two of the squares are cartesian,
then so is the third one.
\item[(b)]
If all the three squares are cartesian,
then 
\eqref{cartesian square diagram} is compact
if and only if
both the left and right square are compact.
\end{itemize}
Consider a cartesian square \eqref{cartesian square diagram}.
\begin{itemize}[resume]
\item[(c)]
If there exist
either
$\gamma \colon B \onto B'$
and
$\alpha' \colon B' \onto A$
such that
$\alpha' \circ \gamma = \alpha$,
or
$\zeta \colon H \onto H'$
and
$\eta' \colon H' \onto G$
such that
$\eta' \circ \zeta = \eta$,
then there exists a commutative diagram \eqref{two squares}
with three cartesian squares.
\end{itemize}
\end{Remark}

\begin{Lemma}\label{two semi squares}
\begin{itemize}
\item[(a)]
If \eqref{cartesian square diagram} is semi-cartesian,
then the right-handed square in \eqref{two squares}
is semi-cartesian.
\item[(b)]
If both the right-handed
and the left-handed square in \eqref{two squares}
are semi-cartesian
then so is \eqref{cartesian square diagram}.
\end{itemize}
\end{Lemma}

\begin{proof}
(a)
This follows from criterion (e) of Definition~\ref{semi-cartesian}.

(b)
This follows from criterion (c) of Definition~\ref{semi-cartesian}:
If
$\zeta(\Ker \beta) = \Ker \beta'$
and
$\eta'(\Ker \beta') = \Ker \phi$,
then
$(\eta' \circ \zeta)(\Ker \beta) = \Ker \phi$.
\end{proof}

Here is a fundamental example of a compact cartesian square:

\begin{Proposition}\label{indecomposable cartesian square}
A cartesian square \eqref{cartesian square diagram}
with $\alpha$ indecomposable
is compact
if and only if
there is no $\gamma \colon G \onto B$
such that
$\alpha \circ \gamma = \phi$.
\end{Proposition}

\begin{proof}
By Definition~\ref{compact cartesian square}(c),
square \eqref{cartesian square diagram}
is not compact
if and only if
\begin{itemize}
\item[(c)]
there are
$\alpha' \colon B \onto A'$,
$\phi' \colon G \onto A'$,
and
$\pi \colon A' \onto A$
such that
$\pi \circ \alpha' = \alpha$,
$\pi \circ \phi' = \phi$,
and
$\pi$ is not an isomorphism.
\end{itemize}
If (c) holds,
then, as $\alpha$ is indecomposable and $\pi$ is not isomorphism,
$\alpha'$ is an isomorphism.
Then $\gamma := (\alpha')^{-1} \circ \phi' \colon G \onto B$
satisfies
$\alpha \circ \gamma = \phi$.

Conversely,
if
$\gamma \colon G \onto B$
satisfies
$\alpha \circ \gamma = \phi$,
then
$A' = B$,
$\alpha' = \id_B$,
$\phi' = \gamma$,
and
$\pi = \alpha$
satisfy (c).
\end{proof}

\begin{Corollary}\label{foot not compact}
If \eqref{cartesian square diagram} is a non-compact cartesian square,
then there is a diagram \eqref{two squares}
of cartesian squares,
in which $\alpha'$ is indecomposable
and the right-handed square is not compact.
\end{Corollary}

\begin{proof}
By Definition~\ref{compact cartesian square}
there are
$\gamma \colon B \onto B'$,
$\phi' \colon G \onto B'$,
and an indecomposable
$\alpha' \colon B' \onto A$
such that
$\alpha' \circ \gamma = \alpha$
and
$\alpha' \circ \phi' = \phi$.
By Remark~\ref{square trivialities}(c),
there exists a commutative diagram \eqref{two squares}
with cartesian squares.
By Proposition~\ref{indecomposable cartesian square}
the right-handed square is not compact.
\end{proof}

\begin{Lemma}\label{cube}
Let 
\begin{equation}\label{cube diagram}
\xymatrix@=15pt{ 
& \Hhat \arr[rr]^(.6){\hat{\eta}} \arr'[d][dd]^(.4){\hat{\beta}}
\arr[ld]_{\tau_H}
&& \Ghat \arr[dd]^(.6){\hat{\phi}} \arr[ld]^(.6){\tau_G} \\
H \arr[rr]^(.7){\eta} \arr[dd]_(.65){\beta} && G \arr[dd]_(.65){\phi} & \\
& \Bhat \arr'[r][rr]^{\hat{\alpha}} \arr[ld]_(.3){\tau_B}
&& \Ahat \arr[ld]^{\tau_A} \\
B \arr[rr]^{\alpha} && A &
}
\end{equation}
be a commutative diagram of epimorphisms
which is a cube with all faces
(resp. the top, the bottom, the front, and the right face)
cartesian squares.
If the top face is compact,
then so is the bottom face.
\end{Lemma}

\begin{proof}
If the above mentioned four faces are cartesian,
then so is the left face.
Indeed,
the concatenation of the top and the right face
is
the concatenation of the left and the bottom face.
Hence by Remark~\ref{square trivialities}(a)
all these faces,
-- in particular the left one --
are cartesian.
Similarly the bottom face is cartesian.

Assume that the bottom face
$(B,A,\Ahat,\Bhat)$
is not compact.
Then there is a group $A'$
and epimorphisms
$\pi \colon A' \onto A$,
$\alpha' \colon B \onto A'$,
and
$\tau_{A'} \colon \Ahat \onto A'$
such that
$\pi \circ \alpha' = \alpha$,
$\pi \circ \tau_{A'} = \tau_A$,
and
$\pi$ is not an isomorphism.
Put 
$G' = A' \times_A G$
and let $\phi' \colon G' \onto A'$ and $\rho \colon G' \onto G$
be the coordinate projections.
The universal property of the fiber product
yields homomorphisms
$\eta'$ and $\tau_{G'}$
such that the following diagram commutes:
\begin{equation*}
\xymatrix@=8pt{ 
&&&& &&&\Ghat
\dotar[dllll]_{\tau_{G'}}
\arr[ddlll]^{\tau_G}
\arr[ddd]^{\hat{\phi}}
\\
&&& G'
\arr[rd]^{\rho}
\arr'[d][ddd]_{\phi'}
\\
H \arr[rrrr]^{\eta} \arr[ddd]_{\beta}
\dotar[rrru]^{\eta'}
&&&& G \arr[ddd]_(.35){\phi}
\\
&&&& &&&\Ahat
\arr[ddlll]^{\tau_A}
\dotarr[dllll]_(.35){\tau_{A'}}
\\
&&& A' \arr[rd]^{\pi} \\
B \arr[rrrr]_{\alpha}
\dotarr[rrru]^{\alpha'}
&&&& A \\
}
\end{equation*}
As
$(B,A,G,H)$ and $(A',A,G,G')$
are cartesian squares,
$\eta'$ is surjective,
by Lemma~\ref{expand}.
so is $\eta'$ (\cite[Lemma 25.2.3(b)]{FJ}).
Similarly,
$\tau_{G'}$ is surjective.
As $\pi$ is not an isomorphism,
neither $\rho$ is.
It follows that the top face in \eqref{cube diagram}
is not compact.
\end{proof}

\begin{Lemma}\label{inverse limit of compact squares}
Let
$(H_i, \rho_{ji} \colon H_j \onto H_i)_{i,j \in I}$
and
$(G_i, \pi_{ji} \colon G_j \onto G_i)_{i,j \in I}$
be two inverse systems
over the same directed set $I$.
For every $i \in I$ let
$\beta_i \colon H_i \onto G_i$
such that
\eqref{invlim1} below

\noindent
\begin{minipage}{.495\linewidth}
\begin{equation}\label{invlim1}
\xymatrix@=24pt{
H_j \arr[d]^{\beta_j} \arr[r]^{\rho_{ji}}
& H_i \arr[d]^{\beta_i} 
\\
G_j \arr[r]^{\pi_{ji}}
& G_i
}
\end{equation}
\end{minipage}
\begin{minipage}{.495\linewidth}
\begin{equation}\label{invlim2}
\xymatrix@=27pt{
H \arr[d]^{\beta} \arr[r]^{\rho_i}
& H_i \arr[d]^{\beta_i} 
\\
G \arr[r]^{\pi_i}
& G_i 
}
\end{equation}
\end{minipage}
\newline
is a cartesian square
for all $j \ge i$.
Let
$H = \varprojlim_{i \in I} H_i$
and
$G = \varprojlim_{i \in I} G_i$,
with projections
$\rho_i \colon H \onto H_i$
and
$\pi_i \colon G \onto G_i$,
respectively.
Then $(\beta_i)_{i \in I}$ induce
$\beta \colon H \onto G$
such that
\eqref{invlim2} above
is a cartesian square,
for all $i \in I$.
If
\eqref{invlim1}
are compact,
for all $j \ge i$,
then
\eqref{invlim2}
are compact,
for all $i \in I$.
\end{Lemma}

\begin{proof}
Fix $i \in I$.

Let $h_i \in H_i$ and $g \in G$
such that
$\pi_i(g) = \beta_i(h_i)$.
For every $j \ge i$
put $g_j = \pi_j(g)$.
There is a unique $h_j \in H_j$
such that
$\beta_j(h_j) = g_j$ and $\rho_{ji}(h_j) = h_i$.
If $k \ge j \ge i$,
then
$$
\beta_j(\rho_{kj}(h_k)) 
= \pi_{kj}(\beta_k(h_k))
= \pi_{kj}(g_k)
= g_j
\quad
\textnormal{and}
\quad
\rho_{ki}(\rho_{kj}(h_k)) = \rho_{ki}(h_k) = h_i,
$$
so, by the uniqueness of $h_j$,
$\rho_{kj}(h_k) = h_j$.
So there is a unique $h \in H$
such that
$\rho_j(h) = h_j$ for every $j \ge i$.
Obviously, this $h$ is the unique element of $H$
such that
$\beta(h) = g$ and $\rho_i(h) = h_i$.
Therefore
\eqref{invlim2}
is cartesian.

Let $H ' \le H$
such that
$\beta(H') = G$
and
$\rho_i(H') = H_i$.
Put $H'_j = \rho_j(H')$
for all $j \in I$.
Then
$\beta_j(H'j) = G_j$
and 
$\rho_{ji}(H'_j) = H'_i$
for all $j \ge i$.
As \eqref{invlim1} is compact,
$H'_j = H_j$
for all $j \in I$.
Hence
$H' = H$.
Therefore \eqref{invlim2} is compact.
\end{proof}

\section{Multiple fiber products}\label{mfp I}

Let us fix a family
$\calH = (\eta_i\colon H_i\to G)_{i \in I}$
of homomorphisms,
for the moment
\textit{nonempty} (until Remark~\ref{empty}(b)).
Consider the category
\begin{multline*}
\calC(\calH) = 
\{(H,\{p_i\}_{i \in I}, \pbar) \st
H \textnormal{ is a group, }
\pbar \colon H \to G
\textnormal{ is a homomorphism, and }
\\
p_i \colon H \to H_i
\textnormal{ is a homomorphism}
\textnormal{ such that }
\eta_i \circ p_i = \pbar,
\textnormal{ for every } i \in I\}.
\end{multline*}
A morphism
$p \colon (H,\{p_i\}_{i \in I}, \pbar) \to (H',\{p'_i\}_{i \in I}, \pbar')$
in this category
is a homomorphism 
$p \colon H \to H'$
such that
$p_i = p'_i \circ p$ for every $i \in I$;
notice that this implies that
$\pbar = \pbar' \circ p$.
Indeed,
$
\pbar' \circ p =
\eta_i \circ p'_i \circ p =
\eta_i \circ p_i =
\pbar
$.

The \textbf{fiber product of $\calH$}
is the terminal object
$(\fprod{i \in I} H_i,\{\pr_{I,i}\}_{i \in I}, \eta_I)$
in $\calC(\calH)$.
In other words,
$(\fprod{i \in I} H_i,\{\pr_{I,i}\}_{i \in I}, \eta_I) \in \calC(\calH)$,
and if
$(H,\{p_i\}_{i \in I}, \pbar) \in \calC(\calH)$,
then there is a unique homomorphism
$p_I \colon H \to \fprod{i \in I} H_i$
such that
$\pr_{I,i} \circ p_I = p_i$ for every $i \in I$.
We write
$H = \fprod{i \in I} H_i$,
if this $p_I$ is an isomorphism.

The fiber product exists and is unique,
up to a unique isomorphism.
Namely,
(\cite[Section 2]{BHH})
let
\begin{equation}\label{concrete}
\fprod{i \in I} H_i =
\Big\{
(h_i)_{i \in I} \in \prod_{i \in I} H_i
\st
\eta_i(h_i) = \eta_j(h_j)
\textnormal{ for all } i,j \in I
\Big\},
\end{equation}
let
$\pr_{I,j}\colon \fprod{i \in I} H_i \to H_j$
be the projection on the $j$-th coordinate,
for every $j \in I$,
and let
${\eta_I = \eta_j\circ \pr_{I,j} \colon \fprod{i \in I} H_i \to G}$,
for any $j \in I$
(the definition is independent of the choice of $j$).
Then
$(\fprod{i \in I} H_i, \{\pr_{I,i}\}_{i \in I}, \eta_I)$
has the above universal property.
(We will usually use this concrete presentation.)
Its uniqueness is an abstract nonsense.

We sometimes write
$(H,\{p_i\}_{i \in I})$)
instead of
$(H,\{p_i\}_{i \in I},\pbar)$.
In particular, in the case of a fiber product,
we usually abbreviate
$(\fprod{i \in I} H_i,\{\pr_{I,i}\}_{i \in I}, \eta_I)$
or
$(\fprod{i \in I} H_i,\{\pr_{I,i}\}_{i \in I})$
to
$\fprod{i \in I} H_i$.

We call 
$\eta_I \colon \fprod{i \in I} H_i \to G$
the \textbf{structure map of $\fprod{i \in I} H_i$}.

In Remark~\ref{properties of fiber product} below
we introduce the notation
$K_j \defeq \Ker \pr_{I,I\smallsetminus \{j\}}$,
for every $j \in I$.

\begin{Remark}\label{empty}
(a)
If $I = \{j\}$ is a singleton,
then
$\fprod{i \in I} H_i = H_j$
and $\eta_I = \eta_j$.

(b)
The definition of 
$\fprod{i \in I} H_i$
does not change,
if we replace $I$ with $I' \supset I$,
where $\eta_i = \id_G$ for all $i \in I' \smallsetminus I$.
Indeed,
the fiber product of the family
$\calH' = (\eta_i\colon H_i\to G)_{i \in I'}$
is also the fiber product of $\calH$.
This presentation allows us to define
$\fprod{i \in I} H_i$ and $\eta_I$
also for $I = \emptyset$ 
--
simply replace $I$ with $\{j\}$,
where $\eta_j = \id_G \colon G \to G$.
Thus, by (a),
$\fprod{i \in \emptyset} H_i = G$ and $\eta_\emptyset = \id_G$.
Applying this, we may
\textit{drop the assumption $I \ne \emptyset$}
from now on.
\end{Remark}

\begin{Remark}\label{properties of fiber product}
We list a few easy properties of 
$(\fprod{i \in I} H_i,\{\pr_{I,i}\}_{i \in I}, \eta_I)$:

(a)
Given $I' \subseteq I$,
we may form the fiber product
$(\fprod{i\in I'} H_i, \{\pr_{I',i}\}_{i \in I}, \eta_{I'})$
of $\calH' = \{\eta_i\colon H_i\to G)_{i \in I'}$.
Let $\pr_{I,I'} \colon \fprod{i \in I} H_i \to \fprod{i\in I'} H_i$
be the projection on the coordinates in $I'$.
It is a homomorphism.
We have
$\pr_{I,j} = \pr_{I',j} \circ \pr_{I,I'}$,
for every $j \in I'$.

It is easy to see that
$\fprod{i \in I} H_i = \varprojlim_{I'} \fprod{i\in I'} H_i$,
where $I'$ runs through the finite subsets of $I$.

(b)
Let $j \in I$.
We identify $\Ker \eta_j$ 
via $\pr_{I,j}$
with
\[
K_j \defeq
\{(h_i)_{i \in I} \st h_i \in H_i \textnormal{ and } h_i=1
\textnormal{ for every } i\neq j \}
= \Ker \pr_{I,I\smallsetminus \{j\}}
\normal \fprod{i \in I} H_i.
\]
Then
$\Ker \eta_I = \prod_{i \in I} K_i$,
$\Ker \pr_{I,j} = \prod_{i \neq j} K_i$
for every $j \in I$,
and
$\Ker \pr_{I,I'} = \prod_{i \in I \smallsetminus I'} K_i$
for every $I' \subseteq I$.

(c)
Let $j \in I$ such that $\eta_j$ is surjective
and $\Ker \eta_j$ is abelian.
Then $K_j \le Z(\prod_{i \in I} K_i) = Z(\Ker \eta_I)$.
Hence,
by Lemma~\ref{action}(a),
$K_j$ and $\Ker \eta_j$ are $G$-modules.
By Lemma~\ref{action}(d),
the identification in (b) preserves the $G$-action.
Thus,
$\Ker \eta_j \isom_G K_j$.

(d)
The fiber product has the following transitivity property:
If $I$ is a disjoint union
$I = \bigdotcup_{\lambda \in \Lambda} I_\lambda$,
then
\begin{equation*}
\fprod{i \in I} H_i =
\fprod{\lambda \in \Lambda} \big(\fprod{i \in I_\lambda} H_i\big).
\end{equation*}
(Here
the left fiber product on the right handed side
-- the one over $\Lambda$ --
is associated with the family
$(\eta_{I_\lambda} \colon
\fprod{i \in I_\lambda} H_i \to G)_{\lambda \in \Lambda}$.)

Indeed,
the right handed side satisfies the universal property of 
$\fprod{i \in I} H_i$.

(e)
In particular,
if $I = I' \dotcup \{j\}$, then
$\fprod{i \in I} H_i =
(\fprod{i \in I'} H_i) \times_G H_j$,
that is,
\begin{equation}\label{simple partition}
\xymatrix@=30pt{
\fprod{i \in I} H_i \ar[r]_{\pr_{I,I'}} \ar[d]^{\pr_{I,j}}
& \fprod{i \in I'} H_i \ar[d]^{\eta_{I'}}\\
H_j \ar[r]_{\eta_j} & G \\
}
\end{equation}
is a cartesian square.


(f)
Let $j \in I$.
If $\eta_j$ is surjective,
then so is $\pr_{I,I \smallsetminus \{j\}}$.
If all $\eta_i$ are surjective,
then so are $\eta_I$,
$\pr_{I,j}$, for all $j \in I$,
and $\pr_{I,I'}$, for all $I' \subseteq I$.

(g)
Assume that all $\eta_i$ are surjective
and let $j \in I$.
By Corollary~\ref{cartesian indecomposable}
applied to \eqref{simple partition},
$\eta_j$ is indecomposable
if and only if
$\pr_{I,I\smallsetminus \{j\}}$ is indecomposable.

(h)
Let $\pi \colon E \to G$ be a homomorphism.
For every $i \in I$ let
$p_i \colon H_i \times_G E \to E$
be the projection on the second
coordinate.
Then
$(p_i)_{i \in I}$ gives rise to the fiber product
$\fprod[E]{i \in I} (H_i \times_G E)$.
It is easy to see that
$\fprod[E]{i \in I} (H_i \times_G E) =
(\fprod{i \in I} H_i) \times_G E$.
\end{Remark}

We will be interested in the question,
when
$(H,\{p_i\}_{i \in I}) \in \calC(\calH)$
is a fiber product of $\calH$.
This can be first simplified as follows:

\begin{Lemma}\label{injective surjective}
Let
$(H,\{p_i\}_{i \in I}) \in \calC(\calH)$
and let
$p_I \colon H \to \fprod{i \in I} H_i$
be the induced map.
Then
\begin{itemize}
\item[(a)]
$\Ker p_I = \bigcap_{i \in I} \Ker p_i$,
hence
$p_I$ is injective
if and only if
$\bigcap_{i \in I} \Ker p_i = 1$.
\item[(b)]
$p_I$ is surjective
if and only if
the map $p_{I'} \colon H \to \fprod{i \in I'} H_i$
induced from
$(H,\{p_i\}_{i \in I'}) \in \calC(\{\eta_i)_{i \in I'})$
is surjective,
for every finite $I' \subseteq I$.
\end{itemize}
\end{Lemma}

\begin{proof}
(a) Clear from \eqref{concrete}.

(b)
By Remark~\ref{properties of fiber product}(a),
$\fprod{i \in I} H_i = \varprojlim_{I'} \fprod{i \in I'} H_i$,
where $I'$ varies over all finite subsets of $I$;
we have 
$p_{I'} = \pr_{I,I'} \circ p_I$ for every such $I'$.
Thus $p_I$ is the inverse limit of the $(p_{I'})$,
from which the assertion follows.
\end{proof}

While (a) above
is a satisfactory criterion for injectivity,
(b) only reduces the question of surjectivity to a finite 
set of indices.

More can be said about surjectivity when
$\calH = (\eta_i\colon H_i\onto G)_{i \in I}$
is a family of epimorphisms.

\begin{Lemma}\label{surjectivity}
Let
$(H,\{p_i\}_{i \in I}, \pbar) \in \calC(\calH)$.
Assume that $\eta_i, p_i$ are epimorphisms,
for every $i \in I$.
For every $J \subseteq I$
let 
$p_J \colon H \to \fprod{i \in J} H_i$
be the unique map induced by
$(p_i)_{i \in J}$.
The following are equivalent:
\begin{itemize}
\item[(a)]
$p_I$ is surjective.
\item[(b)]
For all
$I' \subseteq I$
and $j \in I \smallsetminus I'$,
if $p_{I'}$ is surjective then
\begin{equation}\label{ss}
\xymatrix@=21pt{
H \arr[r]_{p_{I'}} \arr[d]^{p_j}
& \fprod{i \in I'} H_i \arr[d]^{\eta_{I'}}\\
H_j \arr[r]_{\eta_j} & G \\
}
\end{equation}
is semi-cartesian.
\item[(c)]
The same as (b), but only with $I'$ finite.
\end{itemize}
\end{Lemma}

\begin{proof}
(a) $\implies$ (b):
Let 
$I' \subseteq I$
and
$j \in I \smallsetminus I'$,
and put
$J = I' \dotcup \{j\}$.
As $p_I$ is surjective,
so are
$p_J = \pr_{I,J} \circ p_I$
and
$p_{I'} = \pr_{I,I'} \circ p_I$.
By Remark~\ref{properties of fiber product}(e)
the inner square in the following diagram
is cartesian.
\begin{equation}\label{sss}
\xymatrix@=12pt{
H 
\arr[dddr]_{p_j}
\ar[rd]^(.65){p_J}
\arr[rrrrrd]^{p_{I'}}
\\
& \fprod{i \in J} H_i \arr[rrrr]_{\pr_{J,I}} \arr[dd]^{\eta_j} &&&&
\fprod{i \in I'} H_i \arr[dd]^{\eta_{I'}}
\\
\\
& H_j \arr[rrrr]^{\eta_j} &&&& G
}
\end{equation}
By Lemma~\ref{expand}(a),
the outer square is semi-cartesian.

\noindent
(b) $\implies$ (c):
Clear.

\noindent
(c) $\implies$ (a):
By Lemma~\ref{injective surjective}(b)
it suffices to show that
$p_J$ is surjective for every finite subset $J$ of $I$.
We proceed by induction on $|J|$.
If $J = \emptyset$ or $|J| = 1$, 
then $p_J$ is surjective.

So assume $|J| \ge 2$.
Let $j \in J$
and put $I' = J \smallsetminus \{j\}$.
By the induction hypothesis,
$p_{I'}$ is surjective.
By Remark~\ref{properties of fiber product}(e)
the inner square in \eqref{sss}
is cartesian.
By Lemma~\ref{expand}(b),
$p_J$ is surjective.
\end{proof}

We have the following generalization of
\cite[Lemma 25.2.4]{FJ}:

\begin{Lemma}\label{kernels of a fiber product}
Let
$(H,\{p_i\}_{i \in I}, \pbar) \in \calC(\calH)$
with $|I| \ge 2$.
Assume that $\eta_i, p_i$ are epimorphisms,
for every $i \in I$.
Put 
$L_j = \bigcap_{i \ne j} \Ker p_i$,
for every $j \in I$,
and
$L = \Ker \pbar$.
Then the induced map
$p_I \colon H \to \fprod{i \in I} H_i$
is a group isomorphism
if and only if
$L = \prod_{i \in I} L_i$.
\end{Lemma}

\begin{proof}
If $p_I$ is an isomorphism,
we may assume that $H = \fprod{i \in I} H_i$
and $p_i = \pr_{I,i}$ for every $i \in I$.
Then 
$L_i = K_i$ for every $i \in I$
and
$L = \Ker \eta_I = \prod_{i \in I} L_i$
by Remark~\ref{properties of fiber product}(b).

Conversely, assume
$L = \prod_{i \in I} L_i$.

For $J \subseteq I$ let
$p_J \colon H \to \fprod{i \in J} H_i$
be the unique map induced by
$(p_i)_{i \in J}$.

\subdemoinfo{Claim}{
$\Ker p_J = \prod_{i \in J^c} L_i$,
where
$J^c = I \smallsetminus J$.}

Indeed,
if $i \ne j$,
then $L_i \le \Ker p_j$.
Thus
$\prod_{i \in J^c} L_i \le \bigcap_{j \in J} \Ker p_j = \Ker p_J$.
Conversely,
let $h \in \Ker p_J$.
Then $h \in \Ker \pbar = \prod_{i \in I} L_i$,
hence 
$h = h_J h'$,
where
$h_J \in \prod_{i \in J} L_i$
and
$h' \in \prod_{i \in J^c} L_i$.
As noted above,
$h' \in \Ker p_J$,
hence
$h_J = h(h')^{-1} \in \Ker p_J$.
Thus
$h_J \in ( \prod_{i \in J} L_i ) \cap \Ker p_J$.
But
$\prod_{i \in J} L_i \le \bigcap_{i \in J^c} \Ker p_i$
and
$\Ker p_J = \bigcap_{i \in J} \Ker p_i$,
so
$h_J \in \bigcap_{i \in I} \Ker p_i = 1$.
Therefore $h_J = 1$,
whence
$h = h' \in \prod_{i \in J^c} L_i$.
This proves the claim.

Now let $J = I' \dotcup \{j\}$
be a subset of $I$.
Then
$$
(\Ker p_{I'})(\Ker p_j) =
(\prod_{i \notin I'} L_i) (\prod_{i \ne j} L_i) =
\prod_{i \in I} L_i = L = \Ker \pbar.
$$
Hence \eqref{ss} is semi-cartesian.
By Lemma~\ref{surjectivity},
$p_I$ is surjective.
\end{proof}

We generalize 
Definition~\ref{compact cartesian square}:

\begin{Definition}\label{compact fiber product}
Assume that $\{\eta_i\}_{i \in I}$
are epimorphisms
and that $I \ne \emptyset$.
The fiber product
$\fprod{i \in I} H_i$
is \textbf{compact} if
no proper closed subgroup of $\fprod{i \in I} H_i$ is mapped,
simultaneously for all $i \in I$,
by $\pr_{I,i}$ onto $H_i$.

Obviously, if
$\fprod{i \in I} H_i$
is compact and
$(H,\{p_i\}_{i \in I}, \pbar)$
with $p_i$ epimorphisms,
is in $\calC(\calH)$,
then the induced map
$p_I \colon H \to \fprod{i \in I} H_i$
is surjective,
because the subgroup $p_I(H)$ of
$\fprod{i \in I} H_i$
cannot be proper.
\end{Definition}

\begin{Lemma}\label{compact invlim}
Assume that all the $\eta_i$ are surjective.
Then
$\fprod{i \in I} H_i$
is compact
if and only if
$\fprod{i \in I'} H_i$
is compact
for every finite $I' \subseteq I$.
\end{Lemma}

\begin{proof}
Assume that $\fprod{i \in I} H_i$ is compact
and let $I' \subseteq I$.
Let $H' \le \fprod{i \in I'} H_i$
such that
$\pr_{I',i}(H') = H_i$
for every $i \in I'$.
Let
$H = \pr_{I,I'}^{-1}(H') \le \fprod{i \in I} H_i$.
We claim that
$\pr_{I,i}(H) = H_i$
for every $i \in I$.
Indeed,
if $i \in I'$, then
$\pr_{I,i}(H) = \pr_{I',i}(\pr_{I,I'}(H)) = \pr_{I',i}(H') = H_i$.
If $i \in I \smallsetminus I'$,
denote
$E = \pr_{I,i}(H)$.
Then $E \le H_i$
and
$\eta_i(E) = \eta_I(H) = \eta_{I'}(H') = G$.
Let $K_i$ be as in Remark~\ref{properties of fiber product}(b),
then
$K_i \le H$,
hence 
$\Ker \eta_i = \pr_{I,i}(K_i) \le E$.
Thus $E = H_i$.

As $\fprod{i \in I} H_i$ is compact,
$H = \fprod{i \in I} H_i$.
Hence
$H' = \pr_{I,I'}(H) = \fprod{i \in I'} H_i$.

Now assume that $\fprod{i \in I'} H_i$ is compact
for every finite $I' \subseteq I$.
We have
$\fprod{i \in I} H_i = \varprojlim_{I'} \fprod{i \in I'} H_i$.
Let $H' \le \fprod{i \in I} H_i$
such that
$\pr_{I,i}(H') = H_i$
for every $i \in I$.
For every finite $I' \subseteq I$
put
$H'_{I'} = \pr_{I,I'}(H') \le \fprod{i \in I'} H_i$,
then 
$H = \varprojlim_{I'} H'_{I'}$.
But
$\pr_{I',i}(H'_{I'}) = \pr_{I,i}(H') = H_i$
for every $i \in I'$,
hence
$H'_{I'} = \fprod{i \in I'} H_i$
for every finite $I' \subseteq I$.
Thus $H' = \fprod{i \in I} H_i$.
\end{proof}

\begin{Lemma}\label{compact characterization}
Assume that $\{\eta_i\}_{i \in I}$
are epimorphisms
and $I \ne \emptyset$.
Then
$\fprod{i \in I} H_i$
is compact
if and only if
for every finite subset $I'$ of $I$
and every $j \in I \smallsetminus I'$
the following cartesian square,
with $J = I' \dotcup \{j\}$,
is compact.
\begin{equation}\label{simple partition 2}
\xymatrix@=15pt{
\fprod{i \in J} H_i \arr[rr]_{\pr_{J,I'}} \arr[d]^{\pr_{J,j}}
&& \fprod{i \in I'} H_i \arr[d]^{\eta_{I'}}\\
H_j \arr[rr]_{\eta_j} && G \\
}
\end{equation}
\end{Lemma}

\begin{proof}
The square is cartesian
by Remark~\ref{properties of fiber product}(e).
By Lemma~\ref{compact invlim}
we may assume that $I$ is finite.

First assume that $\fprod{i \in I} H_i$ is compact
and consider diagram \eqref{simple partition 2}.
Let $H' \le \fprod{i \in I} H_i$
such that
$\pr_{J,j}(H') = H_j$
and
$\pr_{J,I'}(H') = \fprod{i \in I'} H_i$.
Then
$\pr_{J,i}(H') = \pr_{I',i}(\fprod{i \in I'} H_i) = H_i$
for every $i \in I'$.
By Lemma~\ref{compact invlim},
$\fprod{i \in J} H_i$ is compact,
hence
$H' = \fprod{i \in J} H_i$.
Thus \eqref{simple partition 2} is compact.

Conversely, 
assume that
\eqref{simple partition 2} is compact,
for all $I' \subseteq I$ and $j \in I \smallsetminus I'$.
We proceed by induction on $|I|$.

If $|I| = 1$, say, $I = \{j\}$,
then $\fprod{i \in I} H_i = H_j$ is compact.
So assume that
$|I| \ge 2$.

Choose $j \in I$
and 
let $I' = I \smallsetminus \{j\}$.
Put $J = I$.
As all subsets of $I'$ are also subsets of $I$,
by the induction hypothesis
$\fprod{i \in I'} H_i$ is compact.
Let $H' \le \fprod{i \in I} H_i$
such that
$\pr_{J,i}(H') = H_i$
for every $i \in J$.
Put $H'' = \pr_{J,I'}(H') \le \fprod{i \in I'} H_i$.
Then $\pr_{I',i}(H'') = \pr_{J,i}(H') = H_i$
for every $i \in I'$.
Hence 
$H'' = \fprod{i \in I'} H_i$.
Thus
$\pr_{J,j}(H') = H_i$
and
$\pr_{J,I'}(H') = \fprod{i \in I'} H_i$.
As
\eqref{simple partition 2} is compact,
$H' = \fprod{i \in J} H_i$.
Thus $\fprod{i \in J} H_i$ is compact.
\end{proof}

\begin{Lemma}\label{ff}
Let $I$ be a set
and for every $i \in I$ let
\eqref{carti}
below
be a compact cartesian square.
Assume that $\fprod{i \in I} H_i$ is compact.

\noindent
\begin{minipage}{.325\linewidth}
\begin{equation}\label{carti}
\xymatrix@=36pt{
H_i \arr[r]_{\eta_i} \arr[d]^{\beta_i} & G \arr[d]^{\phi} \\
B_i \arr[r]^{\alpha_i} & A \\
}
\end{equation}
\end{minipage}
\begin{minipage}{.325\linewidth}
\begin{equation}\label{cart2}
\xymatrix{
\fprod{i \in I} H_i
\arr[r]_{\eta_I} \arr[d]^{\hat{\beta}} & G \arr[d]^{\phi} \\
\fprod[A]{i \in I} B_i
\arr[r]^{\alpha_I} & A \\
}
\end{equation}
\end{minipage}
\begin{minipage}{.325\linewidth}
\begin{equation}\label{cart3}
\xymatrix{
\fprod{i \in I} H_i
\arr[r]_{\eta_I} \arr[d]^{\hat{\beta}} & G \arr[d]^{\phi} \\
\fprod[A]{i \in I} B_i
\arr[r]^{\alpha_I} & A \\
}
\end{equation}
\end{minipage}
\newline
Then
the cartesian square
\eqref{cart2}
is compact.
Here 
$\eta_I$
and
$\alpha_I$
are the structure maps of the fiber products
formed from
$\{\eta_i\}_{i \in I}$
and
$\{\alpha_i\}_{i \in I}$,
respectively,
and $\hat{\beta}$ is the unique map
such that
\eqref{cart3}
above
commutes, for every $j \in I$.
\end{Lemma}

\begin{proof}
The existence and the uniqueness of $\hat{\beta}$
follow from the universal property of
$\fprod[A]{i \in I} B_i$.
Using the explicit presentation \eqref{concrete} of
$\fprod{i \in I} H_i$
and a similar one of
$\fprod[A]{i \in I} B_i$,
\ $\hat{\beta}$ is
$(h_i)_{i \in I} \mapsto (\beta_i(h_i))_{i \in I}$.
It easily follows that \eqref{cart2} is a cartesian square.
Therefore, as $\phi$ is surjective, so is $\hat{\beta}$.
As the $\eta_i$ and $\alpha_i$ are surjective,
so are 
$\eta_I$ and $\alpha_I$.

To show that \eqref{cart2} is compact,
let $H' \le \fprod{i \in I} H_i$
such that
$\hat{\beta}(H') = \fprod[A]{i \in I} B_i$
and
$\eta_I(H') = G$.
Put $H'_j = \pr_{I,j}(H') \le H_j$ for every $j \in I$.
Then, from \eqref{cart3},
$\beta_j(H'_j) = B_j$
and
$\eta_j(H'_j) = \eta_I(H') = G$,
because $\eta_j \circ \pr_{I,j} = \eta_I$.
As \eqref{carti} is compact,
$H'_j = H_j$.
Thus, as $\fprod{i \in I} H_i$ is compact,
$H' = H$.
\end{proof}

\section{A duality theory for 
isotypic profinite semi\-simple modules}\label{duality}

In this section
let $R$ be a profinite ring
(\cite[Section 5.1]{RZ})
and let $K$ be
a \textbf{profinite $R$-module},
that is,
a profinite abelian group on which $R$ acts continuously.
Then $K$ is an inverse limit of finite $R$-modules
(\cite[Lemma 5.1.1(b)]{RZ}).
Thus,
a submodule $L$ of $K$ is profinite
if and only if
$L$ is closed in $K$.

Fix a finite simple $R$-module $A$.
We denote
$$
K^* \defeq \Hom_R(K,A) \defeq \{\phi \colon K \to A \st
\phi \textnormal{ is a continuous } R\textnormal{-homomorphism}\}
$$
and $F = A^* = \Hom_R(A,A) = \End_R(A)$.
We say that $K$ is \textbf{$A$-generated} if
$$
\bigcap_{\phi \in K^*} \Ker \phi = 0.
$$

\begin{Remark}\label{A-generated rudiments}
(a)
Let $N = \bigcap_{\phi \in K^*} \Ker \phi$.
Then $K/N$ is $A$-generated.
Indeed,
every $\phi \in K^*$ induces a unique $\phi' \in (K/N)^*$;
the map $\phi \mapsto \phi'$ is a bijection $K^* \to (K/N)^*$
and $\Ker \phi' = (\Ker \phi)/N$,
for every $\phi \in K^*$.
Thus
$$
\bigcap_{\phi' \in (K/N)^*} \Ker \phi' =
\bigcap_{\phi \in K^*} (\Ker \phi)/N =
(\bigcap_{\phi \in K^*} \Ker \phi)/N =
N/N = 0.
$$

(b)
A closed $R$-submodule $C$
of an $A$-generated $R$-module $K$
is also $A$-generated.
Indeed,
every $\phi \in K^*$ restricts to an element of $C^*$,
and so
$$
\bigcap_{\psi \in C^*} \Ker \psi \le
\bigcap_{\phi \in K^*} (C \cap \Ker \phi) \le
\bigcap_{\phi \in K^*} \Ker \phi = 0.
$$

(c)
Suppose that $K$ is finite (as a set).
Then there is an open ideal $I$ of $R$
that acts trivially on $K$,
so that $K$ is an $R/I$-module
(in the usual sense, without topological consideration).
Every submodule and every quotient module of $K$
is also an $R/I$-module.
Thus, for a finite $K$,
the theory of semisimple modules applies.

(d)
If $K$ is finite,
then $K^*$ is finite,
say, $K^* = \{\phi_i\}_{i=1}^n$,
and the map
$x \mapsto (\phi_i(x))_{i=1}^n$
is an $R$-homomorphism
$\pi \colon K \to \prod_{i=1}^n A_i = \bigoplus_{i=1}^n A_i$,
where $A_i \isom_R A$ for each $i$.
As $\prod_{i=1}^n A_i$ is isotypic semisimple,
so is $\pi(K)$, and hence
$\pi(K)$ is a direct sum of finitely many copies of $A$.
If $K$ is also $A$-generated,
$\pi$ is injective,
hence $K$ is
a direct sum of finitely many copies of $A$.
\end{Remark}

\begin{Example}\label{direct power of A}
Let $I$ be a set and let
$A^I$ be the direct product of $|I|$ copies of $A$,
that is,
$A^I = \prod_{i \in I} A_i$, where $A_i = A$ for every $i \in I$.
(If $I = \emptyset$, then $A^I = 0$.)
Then $A^I$ is $A$-generated.
Indeed,
for every $i \in I$ let $\phi_i \colon A^I \onto A = A_i$
be the projection on the $i$-th coordinate.
This is an element of $(A^I)^*$,
and $\bigcap_{i \in I} \Ker \phi_i = 0$.

Proposition~\ref{A-generated module is a free product}
will show
that every $A$-generated $R$-module is of this form.
\end{Example}

\begin{Remark}\label{vector space}
(a)
Let $\phi \in K^*$.
Then $\phi(K)$ is an $R$-submodule of the simple $R$-module $A$,
hence either $\phi = 0$ or $\phi$ is surjective.

(b)
By Schur's Lemma,
$F = \End_R(A)$ is a division ring; 
as $A$ is finite, so is $F$,
hence, by Wedderburn's little theorem,
$F$ is a field.
In particular, $A$ is a vector space over $F$:
here $F$ acts on $A$ by
$\alpha a = \alpha(a)$,
for $\alpha \in F$ and $a \in A$.
The actions of $R$ and $F$ on $A$ commute.

(c)
The field $F$ acts also on $K^*$, by
$\alpha \phi \defeq \alpha \circ \phi$,
for $\alpha \in F$ and $\phi \in K^*$.
It is easy to see that
$K^*$ is also a vector space over $F$.

(d)
%
A continuous homomorphism of $R$-modules
$\beta \colon K \to K'$
defines an $F$-linear map
$\beta^* \colon (K')^* \to K^*$ by
$\phi' \mapsto \phi' \circ \beta$.
Thus $K \mapsto K^*$ is a contravariant functor from
the category of profinite $R$-modules
to the category of vector spaces over $F$.
We restrict it to $A$-generated $R$-modules.
\end{Remark}

Dually to the definition of $K^*$ for an $R$-module we have the following:

\begin{Definition}\label{dual of a vector space}
Let $V$ be a vector space over $F$.
The set of maps
from $V$ to $A$
can be identified with the direct product $A^V$ of copies of $A$,
and as such it is a profinite $R$-module in the product topology;
explicitly, the $R$-action is given by
$(r\psi)(v) = r\psi(v)$,
where $r \in R$, $\psi \colon V \to A$, and $v \in V$.

This action is continuous:
First,
for every $a \in A$ there is an open neighborhood $R_a$ of $0$ in $R$
such that $R_a a = 0$.
Then $R' = \bigcap_{a \in A} R_a$
is an open neighborhood of $0$ in $R$
such that $R' A = 0$.
Hence $R'(A^V) = 0$.
%
%
%

Moreover, $A^V$ is $A$-generated, by Example~\ref{direct power of A}.

The subset $V^*$ of $A^V$
consisting of $F$-linear transformations $V \to A$
is clearly closed in $A^V$.
%
It is an $R$-module, since the actions of $R$ and $F$ on $A$ commute.
As $A^V$ is $A$-generated, so is $V^*$,
by Remark~\ref{A-generated rudiments}(b).

Notice that if $C$ is an $F$-linear basis of $V$,
then the restriction $A^V \onto A^C$ maps $V^*$ isomorphically onto $A^C$.
Thus $V^*$ is isomorphic to the direct product of
$|C| = \dim_F V$
copies of $A$.

An $F$-linear map
$S \colon V' \to V$
of vector spaces over $F$ defines
a continuous $R$-homomorphism
$S^* \colon V^* \to (V')^*$ by
$\psi \mapsto \psi \circ S$.
%
%
Thus $V \mapsto V^*$ is a contravariant functor from
the category of vector spaces over $F$
to the category of $A$-generated profinite $R$-modules.
\end{Definition}

\begin{Definition}\label{natural isomorphisms}
If $K$ is an $A$-generated $R$-module,
there is a natural homomorphism
$\Theta_K \colon K \to (K^*)^*$:
for every $b \in K$ the linear transformation
$\Theta_K(b) \colon K^* \to A$
is given by $\phi \mapsto \phi(b)$.

Similarly,
if $V$ is a vector space over $F$,
there is a natural homomorphism
$\Lambda_{V} \colon V \to (V^*)^*$:
for every $v \in V$ the continuous $R$-homomorphism
$\Lambda_{V}(v) \colon V^* \to A$
is given by $\psi \mapsto \psi(v)$.
\end{Definition}

\begin{Lemma}\label{categories equivalence diagram}
Let $\beta \colon K \to K'$ be a continuous homomorphism of
$A$-generated $R$-modules.
Then the following diagram commutes
$$
\xymatrix@=30pt{
K \ar[r]_{\Theta_{K}} \ar[d]^{\beta} & K^{**} \ar[d]^{\beta^{**}}\\
K' \ar[r]_{\Theta_{K'}} & (K')^{**}\rlap{.} \\
}
$$
\end{Lemma}

\begin{proof}
Let $\phi' \in (K')^*$
and $\psi \in K^*$.
Recall that
$\beta^*(\phi') = \phi' \circ \beta$
and
$\beta^{**}(\psi) = \psi \circ \beta^*$,
hence
\begin{equation}\label{interpretation}
\big(\beta^{**}(\psi)\big) (\phi') =
\big(\psi \circ \beta^* \big) (\phi') =
\psi \big(\beta^*(\phi')\big) =
\psi(\phi' \circ \beta)
.
\end{equation}
Let $b \in K$.
By the definition of $\Theta_{K'}$ we have
$$
\big((\Theta_{K'} \circ \beta)(b)\big) (\phi') =
\Big(\Theta_{K'} \big( \beta(b)\big)\Big) (\phi')=
\phi' \big(\beta(b) \big) =
(\phi' \circ \beta)(b)
$$
and on the other hand by \eqref{interpretation}
$$
\big((\beta^{**} \circ \Theta_{K})(b) \big) (\phi') =
\Big(\beta^{**} \big(\Theta_{K}(b) \big) \Big) (\phi') =
\big(\Theta_K(b) \big)(\phi' \circ \beta) =
(\phi' \circ \beta)(b)
.
$$
Thus
$\big((\Theta_{K'} \circ \beta)(b)\big) (\phi') =
\big((\beta^{**} \circ \Theta_{K})(b) \big) (\phi')$
for every $\phi' \in (K')^*$.
\end{proof}

Completely analogously we have:

\begin{Lemma}\label{categories equivalence diagram 2}
Let $S \colon V \to V'$ be an $F$-linear map of
vector spaces over $F$.
Then the following diagram commutes
$$
\xymatrix@=30pt{
V \ar[r]_{\Lambda_{V}} \ar[d]^{S} & V^{**} \ar[d]^{S^{**}}\\
V' \ar[r]_{\Lambda_{V'}} & (V')^{**} \\
}
$$
\end{Lemma}

\begin{proof}
Let $\psi' \in (V')^*$
and $\phi \in V^*$.
Recall that
$S^*(\psi') = \psi' \circ S$
and
$S^{**}(\phi) = \phi \circ S^*$,
hence
\begin{equation}\label{interpretation 2}
\big(S^{**}(\phi)\big) (\psi') =
\big(\phi \circ S^* \big) (\psi') =
\phi \big(S^*(\psi')\big) =
\phi(\psi' \circ S\big)
.
\end{equation}
Let $v \in V$.
By the definition of $\Lambda_{V'}$ we have
$$
\big(\Lambda_{V'} \circ S(v)\big) (\psi') =
\Big(\Lambda_{V'} \big( S(v)\big)\Big) (\psi')=
\psi' \big(S(v) \big) =
(\psi' \circ S)(v)
$$
and on the other hand by \eqref{interpretation 2}
$$
\big((S^{**} \circ \Lambda_{V})(v) \big) (\psi') =
\Big(S^{**} \big(\Lambda_{V}(v) \big) \Big) (\psi') =
\big(\Lambda_{V}(v) \big)(\psi' \circ S) =
(\psi' \circ S)(v)
.
$$
Thus
$\big(\Lambda_{V'} \circ S(v)\big) (\psi') =
\big((S^{**} \circ \Lambda_{V})(v) \big) (\psi')$
for every $\psi' \in (V')^*$.
\end{proof}

\begin{Theorem}\label{Duality}
Let $K$ be an $A$-generated profinite $R$-module
and let $V$ be a vector space over $F$.
Then
$\Theta_K \colon K \to (K^*)^*$
and
$\Lambda_{V} \colon V \to (V^*)^*$
are isomorphisms.

The $-^*$ functors establish an equivalence
between the categories of $A$-generated $R$-modules and
vector spaces over $F$.
\end{Theorem}

We need some auxiliary results to prove the theorem.

\begin{Lemma}\label{basis for a free product}
Let $K = A^I$ and for every $i \in I$
let $\phi_i \colon K \onto A$
be the projection on the $i$-th coordinate.
Then $\{\phi_i \st i \in I\}$ is a basis of $K^*$ over $F$.
\end{Lemma}

\begin{proof}
We first show that the set is linearly independent over $F$.
Let $\alpha_i \in F$,
almost all zero,
such that $\sum_i \alpha_i \phi_i = 0$.
Let $j \in I$ and let $a \in A$.
Define $b \in K$ by
$\phi_j(b) = a$ and $\phi_i(b) = 0$ for $i \neq j$.
Then
$0 = \left(\sum_i \alpha_i \phi_i\right)(b) = \alpha_j(a)$.
Hence $\alpha_j = 0$.

Now we show that
$\left(\phi_i\right)_{i \in I}$
spans $K^*$.
Write $K$ as $\prod_{i \in I} A_i$,
where $A_j = A$ for every $j \in I$;
explicitly,
$A_j = \bigcap_{i\neq j} \Ker \phi_i$.
Let $\psi \colon K \to A$ be an element of $K^*$.
For every $i \in I$
let $\alpha_i \colon A \to A$
be the restriction of $\psi$ to $A_i$.
This is an $R$-homomorphism and hence $\alpha_i \in F$.
As $\psi$ is continuous, $\Ker \psi$ is open in $K$,
and hence contains every $A_i$,
except for $i$ in a finite subset $I'$ of $I$.
Thus $\alpha_i = 0$ for all $i \notin I'$.
We claim that
$\psi = \sum_{j \in I} \alpha_j \phi_j$,
that is,
$\psi = \sum_{j \in I'} \alpha_j \phi_j$.

Indeed,
let $b \in K$.
Let $a_i$ be its $i$th coordinate.
Then $b = b' + \sum_{j\in I'} b_j$, where
$(b')_i =
\begin{cases}
0 &\textnormal{if } i \in I'\\
a_i &\textnormal{if } i \notin I'\\
\end{cases}
$
and
$(b_j)_i =
\begin{cases}
a_i &\textnormal{if } i=j\\
0 &\textnormal{otherwise }\\
\end{cases}
$.
Hence
$$
\psi(b) =
\psi(b') + \sum_{j \in I'} \psi(b_j) =
0 + \sum_{j \in I'} \alpha_j(a_j) =
\sum_{j \in I'} \alpha_j( \phi_j(b)) =
\left(\sum_{j \in I'} \alpha_j \phi_j\right) (b)
,
$$
whence
$\psi = \sum_{j \in I'} \alpha_j \phi_j
= \sum_{j \in I} \alpha_j \phi_j$.
\end{proof}

We will need a version of
the Chinese Remainder Theorem:

\begin{Lemma}\label{CRT}
Let $\phi_1,\ldots, \phi_n \in K^*$ be linearly independent over $F$.
Then the $R$-homomorphism
$K \to A^n$
given by
$b \mapsto (\phi_1(b),\ldots,\phi_n(b))$
is surjective.
\end{Lemma}

\begin{proof}
By induction on $n$.
For every $i$ let $N_i = \Ker \phi_i$.

By induction hypothesis
$K \to \prod_{i=1}^{n-1} K/N_i$
is surjective.
Its kernel is
$N \defeq \bigcap_{i=1}^{n-1} N_i$,
so $K/N \onto \prod_{i=1}^{n-1} K/N_i$ is an isomorphism.
Therefore so is
$K/N \times K/N_n \onto \prod_{i=1}^{n} K/N_i$.
As $K \to \prod_{i=1}^{n} K/N_i$ is the composition of
$K \to K/N \times K/N_n$
and
the latter map,
we have to show that
$K \to K/N \times K/N_n$
is surjective,
that is, that
$N + N_n = K$.

Let $\Kbar = K/N$
and let $\bar\phi_1,\ldots, \bar\phi_{n-1} \colon \Kbar \to A$
be the maps induced from $\phi_1,\ldots, \phi_{n-1}$.
The above isomorphism
$\Kbar \to \prod_{i=1}^{n-1} K/N_i$
gives $\Kbar$ the structure of the direct product $A^{n-1}$
such that
$\bar\phi_1,\ldots, \bar\phi_{n-1}$ are the coordinate projections.
By Lemma~\ref{basis for a free product},
$\bar\phi_1,\ldots, \bar\phi_{n-1}$ are a basis of $\Kbar^*$.
Therefore,
if $N \subseteq N_n$,
the map $\bar\phi_n \colon \Kbar \to A$ induced from $\phi_n$
is a linear combination of $\bar\phi_1,\ldots, \bar\phi_{n-1}$,
say,
$\bar\phi_n = \sum_{i=1}^{n=1} \alpha_i \bar\phi_i$.
It follows that
$\phi_n = \sum_{i=1}^{n=1} \alpha_i \phi_i$,
a contradiction.
Thus,
$N \not\subseteq N_n$.
Hence
$N_n \subsetneqq N + N_n$.

As $A = K/N_n$ is a simple $R$-module,
there is no $R$-submodule $M$ of $K$
such that
$N_n \subsetneqq M \subsetneqq K$.
By assumption
$N_n \subsetneqq N + N_n \subseteq K$,
hence $N + N_n = K$.
\end{proof}

\begin{Proposition}\label{A-generated module is a free product}
Let $K$ be an $A$-generated $R$-module.
Let $\{\phi_i\}_{i \in I}$ be a basis of $K^*$ as a vector space over $F$.
Then 
$\phi \colon K \to A^I$
given by
$\phi(b) = \left( \phi_i(b)\right)_{i \in I}$
is an isomorphism of $R$-modules.
\end{Proposition}

\begin{proof}
Clearly $\phi$ is an abstract $R$-homomorphism.

It is continuous:
A basic open neighborhood of $0$ in $A^I$
is of the form
$\{(a_i)_{i \in I} \st a_i = 0 \text{ for all } i \in I'\}$,
for some finite subset $I'$ of $I$,
and its inverse image under $\phi$
is the open subset
$\bigcap_{i \in I'} \Ker \phi_i$.

We show that $\phi$ is injective.
Let $b \in K$ such that $\phi(b) = 0$.
Then $\phi_i(b) = 0$ for every $i \in I$.
As every $\rho \in K^*$
is a linear combination of finitely many $\phi_i$'s,
we have $\rho(b) = 0$ for every $\rho \in K^*$.
Therefore, since $K$ is $A$-generated,
$b \in \bigcap_{\rho\in K^*} \Ker \rho = \{0\}$,
that is, $b = 0$.

If $I'$ is finite subset of $I$,
the induced map
$\pr_{I,I'} \circ \phi \colon K \to A^{I'}$
is surjective
by Lemma~\ref{CRT}.
As
$A^I = \varprojlim\limits_{I' \subseteq I\textnormal{ finite}} A^{I'}$,
where the structure maps $A^I \to A^{I'}$
are the restrictions from $I$ to $I'$,
the assertion follows.
\end{proof}

\begin{proof}[Proof of Theorem~\ref{Duality}]
(a)
We first show that $\Theta_K$ is an isomorphism.

By Proposition~\ref{A-generated module is a free product}
we may assume that $K = A^I$ for some set $I$;
By Lemma~\ref{basis for a free product}
the coordinate projections
$(\phi_i \colon K \onto A)_{i \in I}$ are a basis of $K^*$.
As noted in Definition~\ref{dual of a vector space},
$(K^*)^*$ is the set of all functions from this basis to $A$;
hence it can be identified with $A^I$.

Under this identification $\Theta_K$ is the identity.
Indeed,
let $b = (a_i)_{i \in I} \in K$.
Then $\Theta_K(b)$ maps $\phi_i$ onto $\phi_i(b) = a_i$,
hence it is identified with the map $I \to A$
that maps $i$ onto $a_i$, that is, with $b$.

We now show that $\Lambda_{V}$ is an isomorphism.

Let $C$ be a basis of $V$.
By Definition~\ref{dual of a vector space},
$V^* = A^C$.
The coordinate projections of $A^C$ can be identified with the elements
of $C$ (a projection $\phi$ is identified with the unique $v \in C$
such that $\phi(v) \neq 0$).
By Lemma~\ref{basis for a free product}
$(V^*)^* = (A^C)^*$
is the vector space with basis consisting of these coordinate projections,
that is, it can be identified with a vector space with basis $C$,
that is, with $V$.

Under this identification $\Lambda_{V}$ is the identity on $C$
and hence also on $V$.
Indeed,
let $v \in C$.
Then $\Lambda_{V}(v)$ is the projection $\phi \colon A^C \onto A$
that maps $(a_u)_{u \in C}$ to $a_v$,
hence it is identified with $v$.

By Lemmas~\ref{categories equivalence diagram}
and \ref{categories equivalence diagram 2},
$\Theta_K$ and $\Lambda_{V}$ are natural isomorphisms between
the composition of the $-^*$ functors and the identity functors.
\end{proof}



%

\begin{Remark}\label{inj surj dual}
Let
$T \colon V \to W$
be an $F$-linear map of vector spaces over $F$.
By linear algebra we know that
\begin{itemize}
\item[(a)]
$T^* \colon W^* \to V^*$ is
injective
if and only if
$T$ is surjective;
\item[(b)]
$T^* \colon W^* \to V^*$ is
surjective
if and only if
$T$ is injective.
\end{itemize}
It follows by the duality that
if
$\alpha \colon K \to C$ is a continuous homomorphism
of $A$-generated $R$-modules, then
\begin{itemize}
\item[(c)]
$\alpha^* \colon C^* \to K^*$ is
surjective
if and only if
$\alpha$ is injective;
\item[(d)]
$\alpha^* \colon C^* \to K^*$ is
injective
if and only if
$\alpha$ is surjective.
\end{itemize}
\end{Remark}

\begin{Corollary}\label{A-generated submodule complement}
Let $L$ be a closed submodule of
an $A$-generated $R$-module $K$.
Then there is
an $A$-generated $R$-module $M$
such that
$K = L \directsum M$.
\end{Corollary}
\begin{proof}
Let $\alpha \colon L \to K$ be the inclusion map.
By Remark~\ref{inj surj dual}(c),
$\alpha^* \colon K^* \to L^*$ is surjective.
Let $V = \Ker \alpha^*$.
By linear algebra
we may view $L^*$
as an $F$-subspace of $K^*$
such that
$K^* = L^* \directsum V$
and $\alpha^*$ is the identity on $L^*$.
By Theorem~\ref{Duality},
$K = L \directsum V^*$,
where
$V^*$ is $A$-generated.
\end{proof}

\begin{Corollary}\label{quotient is A-generated}
A quotient module of
an $A$-generated $R$-module
is $A$-gen\-er\-at\-ed.
\end{Corollary}
\begin{proof}
Let $K \onto M$
be an epimorphisms of $R$-modules,
such that
$K$ is $A$-generated.
Let $L$ be its kernel.
By Corollary~\ref{A-generated submodule complement},
there is
an $A$-generated $R$-module $M'$
such that
$K = L \directsum M'$.
Thus $M \isom K/L \isom M'$
is $A$-generated.
\end{proof}

\begin{Corollary}\label{submodule simplified}
Let $K$ be an $A$-generated $R$-module
and let $L$ be a closed submodule of $K$.
Then we can choose a basis
$\{\phi_i\}_{i \in I}$ of $K^*$ 
such that the isomorphism
$\phi \colon K \to A^I$
of Proposition~\ref{A-generated module is a free product}
defined by
$\phi(b) = \left( \phi_i(b)\right)_{i \in I}$
maps $L$ onto 
$A^J$ for some subset $J$ of $I$.
\end{Corollary}

\begin{proof}
By Corollary~\ref{A-generated submodule complement}
there is an $A$-generated $R$-module $M$
such that
$K = L \directsum M$.
Choose a basis
$\{\phi_i\}_{i \in J}$ of $L^*$
and extend every element of it to an element of $K^*$
by the zero on $M$.
Analogously
choose a basis
$\{\phi_i\}_{i \in J'}$ of $M^*$
and extend every element of it to an element of $K^*$
by the zero on $L$.
Put $I = J \dotcup J'$.
Then
$\{\phi_i\}_{i \in I}$
is a basis of $K^*$.
By Proposition~\ref{A-generated module is a free product},
applied to $L$ instead of $K$,\
$\phi$ maps $L$ onto $A^J$.
\end{proof}

\section{A duality theory for extensions by
isotypic profinite semisimple modules}
\label{second cohomology}

We consider the cohomology of a profinite group $G$
with coefficients in profinite modules
(i.e., objects in the category $\calC_R(G)$
in \cite[Section 2.6]{Symonds}).

A homomorphism of profinite $G$-modules
$\beta \colon K \to L$
induces a map $\beta_* \colon H^2(G,K) \to H^2(G,L)$.
By the correspondence between the second cohomology and group extensions
(\cite[p.~233]{RZ})
it induces a map between extensions of $G$ by $K$
to extensions of $G$ by $L$.
It is an exercise
to give the following
explicit description of the latter map:

\begin{Lemma}\label{map of extensions}
Let $G$ be a profinite group
and let
$\beta \colon K \to L$
be a homomorphism of profinite $G$-modules.
Let
\begin{equation}\label{2 extensions}
\calH \colon\
\xymatrix{
0 \shortarrow K \ar[r] & H \ar[r]^{\eta} & G \shortarrow 1
}
\quad
\textnormal{and}
\quad
\E \colon\
\xymatrix{
0 \shortarrow L \ar[r] & E \ar[r]^{\pi}  & G \shortarrow 1
}
\end{equation}
be two extensions.
Then $\E = \beta_*(\calH)$
if and only if
there exists a commutative diagram
\begin{equation}\label{map of extensions diagram}
\xymatrix{
\calH \colon &
0 \ar[r] & K \ar[r] \ar[d]^{\beta}& H \ar[r]^{\eta} \ar[d]^{\eps}& G \ar[r] \ar@{=}[d] & 1 \\
\E \colon &
0 \ar[r] & L \ar[r] & E \ar[r]^{\pi} & G \ar[r] & 1. \\
}
\end{equation}
Moreover, if such a diagram exists,
$\eps$ is unique up to a composition with an automorphism
$\omega$ of $E$ such that $\pi \circ \omega = \pi$
and $\omega|_{L} = \id_{L}$.

If $\beta$ is surjective\slash isomorphism, then so is $\eps$.
\end{Lemma}

Recall that
extensions~\eqref{2 extensions}
are \textbf{congruent}
(\cite[p.~100]{Ribes}),
if there is a commutative diagram
\eqref{map of extensions diagram}
with $K = L$ and $\beta = \id_K$.
Thus, given $\calH$,
an extension $\E$ such that
\eqref{map of extensions diagram}
commutes
is unique up to a congruence of extensions.
We identify congruent extensions,
and thus $\E$ is unique.

We will also have to classify extensions up to isomorphism:

Extensions \eqref{2 extensions}
-- even with non-abelian kernels $K,L$ --
are \textbf{isomorphic},
if there is an isomorphism of profinite groups
$\eps \colon H \to E$
such that
$\pi \circ \eps = \eta$.
If this is the case
and $K,L$ are abelian $G$-modules,
then 
$\beta := \eps|_K \colon K \to L$
is an isomorphism of profinite $G$-modules
such that \eqref{map of extensions diagram}
commutes;
thus we may assume that $L = K$
and $\beta \in \Aut_G(K)$,
and 
$\E = \beta_*(\calH)$.
Conversely, by Lemma~\ref{map of extensions},
every $\beta \in \Aut_G(K)$  defines
an extension isomorphic to $\calH$,
namely,
$\beta_*(\calH)$.
So,
$\Aut_G(K)$ acts on $H^2(G,K)$
and
the isomorphism classes of extensions of $G$
by $K$
bijectively correspond to the elements of
$H^2(G,K)/\Aut_G(K)$.

Since we deal with commutative diagrams with epimorphisms
of profinite groups,
it will be convenient to redefine the inflation map as follows:
Let $\phi \colon G \onto \Gbar$ be an epimorphism
of profinite groups.
A profinite $\Gbar$-module $K$ is also a $G$-module --
$G$ acts on $K$ by composing $\phi$ with the $\Gbar$-action on $K$.
The \textbf{inflation along $\phi$}
is the map
$\Inf_\phi \colon H^2(\Gbar,K) \to H^2(G,K)$
induced from the map on cochains
given by composing the $\Gbar$-cochains with $\phi$.
It is an exercise to show that:

\begin{Lemma}\label{inflation in diagram}
Let $\Gbar$ be a profinite group
and  $K$ a profinite $\Gbar$-module.
Let $\phi \colon G \onto \Gbar$
and let $G$ act on $K$ via $\phi$.
Let
$\alpha \colon B \onto \Gbar$
and 
$\eta \colon H \onto G$
be extensions with kernel $K$
(with the above action).
Then
$\eta$ is isomorphic to the inflation $\Inf_\phi(\alpha)$
of $\alpha$ along $\phi$
if and only if
there exists a cartesian square
\begin{equation}\label{mycartesian}
\xymatrix@=30pt{
H \arr[r]_{\eta} \arr[d]^{\beta} & G \arr[d]^{\phi}
\\
B \arr[r]^{\alpha} & \Gbar \rlap{.} \\
}
\end{equation}
\end{Lemma}

Following this we write
$\eta \isom_G \Inf_\phi(\alpha)$
if there is a cartesian square \eqref{mycartesian},
even if $\Ker \alpha, \Ker \eta$ are not abelian.

\begin{Corollary}\label{H2 is extensions}
The following two categories are equivalent:
\begin{enumerate}
\item[(a)]
The class of pairs $(K,f)$,
where $K$ is a profinite $G$-module
and $f \in H^2(G,K)$;
morphism $(K,f) \to (L,g)$ in this category
is a $G$-homo\-mor\-phism $\beta \colon K \to L$
such that $g = \beta_* f = \beta \circ f$.
\item[(b)]
The class of profinite group extensions
$0 \to K \to H \to G \to 1$,
up to congruence,
where $K$ is a profinite $G$-module;
morphism in this category
is a commutative diagram \eqref{map of extensions diagram}.
\end{enumerate}
\end{Corollary}

For the rest of this section fix a
\emph{finite simple $G$-module $A$}. 
Thus $A$ is a finite simple $[[\Z G]]$-module,
and so we may apply the material of section~\ref{duality}
with $R = [[\Z G]]$.
Recall (Remark~\ref{vector space}(b))
that $F = \End_G(A)$ is a finite field.

\begin{Remark}\label{cohomology is a vector space}
We use the notation of \cite[Chapter II \S1]{Ribes}.

(a)
Notice that the
\textbf{group of non-homogeneous cochains}
$\Cbar^n(G,A)$
is a vector space over $F$:
For every $\alpha \in F$ we have
$\alpha f := \alpha \circ f$.
Moreover,
the boundary operator
$\deriv \colon \Cbar^n(G,A) \to \Cbar^{n+1}(G,A)$
commutes with $\alpha$,
hence it is an $F$-linear operator.

It follows that
$\Bbar^n(G,A) \subseteq \Zbar^n(G,A)$
are vector subspaces of $\Cbar^n(G,A)$.
Thus
$H^n(G,A) = \Zbar^n(G,A)/\Bbar^n(G,A)$
is a vector space over $F$.

(b)
In particular,
by Corollary~\ref{H2 is extensions},
the set of congruence classes 
of epimorphisms onto $G$
with kernel $A$ (as a $G$-module)
is also a vector space over $F$.
By Lemma~\ref{action}(c),
these epimorphisms are necessarily indecomposable.
In this sense we speak below
of linear combinations
of indecomposable epimorphisms onto $G$ with kernel $A$.

(c)
If $\pi \colon H \onto G$ is an epimorphism of profinite groups,
then $A$ is also a simple $H$-module
and $\End_H(A) = \End_G(A) = F$.
It is easy to see that the
inflation
$\Inf_\pi \colon H^n(G,A) \to H^n(H,A)$
is an $F$-linear map.
\end{Remark}

\begin{Lemma}\label{cocycle of a fiber product}
Let
$(H_i \onto G \st i \in I)$
be a family of extensions with kernel $A$,
as a $G$-module.
For every $i \in I$
let $f_i \in \Zbar^2(G,A)$
be a (non-homogeneous) cochain representing $H_i \onto G$.
Then
$\eta_I \colon \fprod{i \in I} H_i \onto G$
is represented by a cochain
$g \in \Zbar^2(G,A^I)$
given by
$g(\sigma,\tau) = \big(f_i(\sigma,\tau)\big)_{i \in I}$.
\end{Lemma}

\begin{proof}
Let $i \in I$.
We think of $f_i$ as a continuous function
$f_i \colon G\times G \to A$
with $\deriv f_i = 0$.
Thus
there is a continuous section $u_i \colon G \to H_i$
of $H_i \onto G$ such that
$$
u_i(\sigma) u_i(\tau) =
f_i(\sigma,\tau) u_i(\sigma \tau),
\qquad
\sigma, \tau \in G
.
$$
Now,
$\sigma \mapsto \big(u_i(\sigma)\big)_{i \in I}$
is a continuous section
$u \colon G \to \fprod{i \in I} H_i$ of $\fprod{i \in I} H_i \onto G$,
and the map $g \colon G \times G \to A^I$
given by
$g(\sigma,\tau) = \big(f_i(\sigma,\tau)\big)_{i \in I}$
is continuous and satisfies
$$
u(\sigma) u(\tau) =
g(\sigma,\tau) u(\sigma \tau),
\qquad
\sigma, \tau \in G
.
$$
Hence $g$ is a cochain 
and it represents $\fprod{i \in I} H_i \onto G$.
\end{proof}

\begin{Definition}\label{two categories}
Let $n \in \N$.
We will consider the following two categories:
\begin{itemize}
\item[(a)]
The class $\calH^n$ of pairs $(K,f)$,
where $K$ is a profinite $A$-generated $G$-module
and $f \in H^n(G,K)$;
morphism $(K_1,f_1) \to (K_2,f_2)$ in this category
is a $G$-homomorphism $\beta \colon K_1 \to K_2$
such that $f_2 = \beta_* f_1 = \beta \circ f_1$.
\item[(b)]
The class $\T^n$ of pairs $(V,S)$,
where $V$ is a vector space over $F$ and
$S \colon V \to H^n(G,A)$ is an $F$-linear transformation;
morphism $(V_1,S_1) \to (V_2,S_2)$ in this category
is
an $F$-linear map $T \colon V_2 \to V_1$ such that
$S_2 = S_1 \circ T$.
\end{itemize}
In particular, if $n = 2$,
then (a) is equivalent to
\begin{itemize}[resume]
\item[(a')]
The class of profinite group extensions
$0 \to K \to H \to G \to 1$,
up to congruence,
where $K$ is a profinite $A$-generated $G$-module;
morphism in this category
is a commutative diagram \eqref{map of extensions diagram}.
\end{itemize}
\end{Definition}

\begin{Construction}\label{two functors}
We construct two contravariant functors,
$X^n \colon \calH^n \to \T^n$ and $Y^n \colon \T^n \to \calH^n$,
which give an equivalence of categories.

(a)
\textit{
Construction of
$X^n \colon \calH^n \to \T^n$.
}
Let $(K,f) \in \calH^n$.
By Remark~\ref{vector space}(c),
$K^* = \Hom_G(K,A)$ is a vector space over $F$.
A $G$-homomorphism
$\phi \colon K \to A$
defines a map
$\phi_* \colon \Zbar^n(G,K) \to \Zbar^n(G,A)$ by
$f' \mapsto \phi \circ f'$.
Thus a representative $f' \in \Zbar^n(G,K)$ of $f$ defines a map
$$
S_{K,f'} \colon K^* \to \Zbar^n(G,A)
$$
by
$S_{K,f'}(\phi) = \phi_*(f') = \phi \circ f'$.
It is easy to see that
$S_{K,f'}$ is an $F$-linear transformation.
Hence it induces an $F$-linear transformation
\begin{equation}
S_{K,f} \colon K^* \to H^n(G,A).
\end{equation}
Notice that $S_{K,f}$ does not depend on the choice of $f'$
(and hence we may write $S_{K,f}(\phi) = \phi \circ f$):
If $f' \sim f''$, that is, $f' = f'' + \deriv g$
for some $g \in \Cbar^{n-1}(G,K)$,
then
$$
\phi \circ f' = \phi \circ f'' + \phi \circ \deriv g,
\quad
\textnormal{ where }
\phi \circ \deriv g = \deriv(\phi \circ g) \in \Bbar^n(G,K),
$$
hence $f',f''$ define the same map 
$S_{K,f} \colon K^* \to H^n(G,A)$.
Put $X^n(K,f) \defeq (K^*,S_{K,f})$.

We remark that if $n=2$
and we identify $(K,f)$ with an extension $\eta$ of $G$ by $K$,
by Corollary~\ref{H2 is extensions},
then
$S_{K,f}(\phi)$ is the extension $\pi$ of $G$ by $A$
such that the following diagram commutes:
\begin{equation}\label{induced}
\xymatrix{
0 \ar[r] & K \ar[r] \ar[d]^{\phi}& H \ar[r]_{\eta} \ar[d]
& G \ar[r] \ar@{=}[d] & 1
\\
0 \ar[r] & A \ar[r] & E \ar[r]^{\pi} & G \ar[r] & 1\rlap{,}
}
\end{equation}
i.e., 
$\Img S_{K,f}$
consists of all indecomposable extensions of $G$
dominated by $\eta$.

Let $\beta \colon (K_1,f_1) \to (K_2,f_2)$
be a morphism in $\calH^n$.
The $G$-homomorphism $\beta \colon K_1 \to K_2$
induces an $F$-linear map
$\beta^* \colon K_2^* \to K_1^*$
by $\psi \mapsto \psi \circ \beta$.
We claim that
$S_{K_2,f_2} = S_{K_1,f_1} \circ \beta^*$.
Indeed,
let $\psi \in K_2^*$.
Let $f' \in \Zbar^n(G,K_1)$ be a representative of $f_1 \in H^n(G,K_1)$.
Then $\beta \circ f' \in \Zbar^n(G,K_2)$ is a representative of
$f_2 = \beta_*(f_1) \in H^n(G,K_2)$.
Therefore
$$
S_{K_1,f'} \circ \beta^*(\psi) =
S_{K_1,f'}(\psi \circ \beta) =
\psi \circ \beta \circ f' =
S_{K_2,\beta \circ f'}(\psi) =
S_{K_2,\beta_*(f')}(\psi).
$$
Modulo $\Bbar^n(G,K)$
this gives
$S_{K_1,f_1} \circ \beta^*(\psi) =
S_{K_2,f_2}(\psi)$.

Thus $\beta^*$ is a morphism in $\T^n$.
Put $X^n(\beta) \defeq \beta^*$.

This completes the definition of $X^n \colon \calH^n \to \T^n$.


(b)
\textit{
Construction of $Y^n \colon \T^n \to \calH^n$.
}
Let $(V,S) \in \T^n$.
Lift $S \colon V \to H^n(G,A)$ to an $F$-linear transformation
$S' \colon V \to \Zbar^n(G,A)$.
Thus, for every $v \in V$, the map
$S'(v) \colon G^n \to A$
is continuous, i.e., locally constant, and satisfies
\begin{multline*}
0 =
\Big(\deriv \big(S'(v)\big)\Big)
(\sigma_1, \sigma_2, \ldots,\sigma_{n+1}) \defeq
\sigma_1
\big(S'(v)\big)(\sigma_2, \sigma_3, \ldots, \sigma_{n+1}) +
\\
\sum_{i=1}^n
(-1)^i
\big(S'(v)\big)(\sigma_1, \sigma_2, \ldots,
\sigma_i\sigma_{i+1},
\ldots,\sigma_{n+1}) 
+(-1)^n
\big(S'(v)\big)(\sigma_1, \sigma_2, \ldots, \sigma_n)  .
\end{multline*}
Recall
(Definition~\ref{dual of a vector space})
that $V^*$ is a profinite $G$-module.
Define a map
$f' \colon G^n \to V^*$
by
$
\big(f'(\sigma_1,\ldots,\sigma_n)\big)(v) =
\big(S'(v)\big)(\sigma_1,\ldots,\sigma_n)
$.

%
%
It is continuous:
A basic open subset $V^*$ is of the form
$$
R = \{{\phi \in V^*} \st S'(v_i) = a_i, \ i = 1,\ldots, n\},
\text{ for some } \{v_i\}_{i=1}^n \subseteq V, 
\{a_i\}_{i=1}^n \subseteq A.
$$
Then
$(f')^{-1}(R) = 
\{\sigma \in G^n \st S'(v_i)(\sigma) = a_i, \ i = 1,\ldots, n\}$
is open in $G^n$,
because $S'(v_i)$ is continuous for every $i$.

It is easy to check that
\begin{equation*}
\big( (\deriv f')(\sigma_1,\ldots,\sigma_{n+1})\big)(v) =
\big(\deriv \big(S'(v)\big)\big)(\sigma_1,\ldots,\sigma_{n+1})
,
\end{equation*}
hence, as $\deriv \big(S'(v)\big)=0$, we have
$f' \in \Zbar^n(G,V^*)$.
Thus $f'$ defines an element $f \in H^n(G,V^*)$,
the class of $f'$.
Put $Y^n(V,S) \defeq (V^*,f)$.

Notice that $f$ does not depend on the choice of $S'$:
If $S'' \colon V \to \Zbar^n(G,A)$ is another lifting of $S$,
it defines $f'' \in \Zbar^n(G,V^*)$.
As $S' - S'' \colon V \to \Bbar^n(G,A)$,
for every $v \in V$ there is
$g_v \in \Cbar^{n-1}(G,A)$ such that
$S'(v) - S''(v) = \deriv g_v$.
Define $g \in \Cbar^{n-1}(G,V^*)$ by
$\big(g(\sigma_1,\ldots,\sigma_{n-1})\big)(v) =
\big(g_v\big)(\sigma_1,\ldots,\sigma_{n-1})$.
Then
\begin{multline*}
\big((f'-f'')(\sigma_1,\ldots,\sigma_n)\big)(v) =
\big(S'(v)\big)(\sigma_1,\ldots,\sigma_n)
-
\big(S''(v)\big)(\sigma_1,\ldots,\sigma_n)
=
\\
=
\big(\deriv g_v\big)(\sigma_1,\ldots,\sigma_n)
=
\sigma_1 g_v(\sigma_2,\ldots, \sigma_n)
+
\\
\sum_{i=1}^n
(-1)^i
g_v(\sigma_1,\ldots, \sigma_i \sigma_{i+1},\ldots, \sigma_n)
+ (-1)^n
g_v(\sigma_1,\ldots, \sigma_n)
=
\\
\sigma_1 g(\sigma_2,\ldots, \sigma_n)(v)
+
\sum_{i=1}^n
(-1)^i
g(\sigma_1,\ldots, \sigma_i \sigma_{i+1},\ldots, \sigma_n)(v)
+
\\
(-1)^n
gv(\sigma_1,\ldots, \sigma_n)(v)
=
\big(\deriv g(\sigma_1,\ldots, \sigma_n)\big) (v),
\end{multline*}
hence $f'-f'' = \deriv g \in \Bbar^n(G,V^*)$,
whence $S',S''$ define the same element $f \in H^n(G,V^*)$.

Let $T \colon (V_1,S_1) \to (V_2,S_2)$
be a morphism in $\T^n$.
Then $T \colon V_2 \to V_1$ is an $F$-linear map
(that induces a continuous $G$-homomorphism
$T^* \colon V_1^* \to V_2^*$
by $T^*(\phi) = \phi \circ T$)
and $S_2 = S_1 \circ T$.
Lift $S_1$ to an $F$-linear map $S'_1 \colon V_1 \to \Zbar^n(G,A)$,
then $S'_2 = S'_1 \circ T$ lifts $S_2$.
For $i = 1, 2$
let $(V_i^*,f_i) = Y^n(V_i,S_i)$ and let $f_i' \in \Zbar^n(G,A)$
be a representative of $f_i$, defined by
$
\big(f_i'(\sigma)\big)(v) =
\big(S'_i(v)\big)(\sigma)
$,
for all $\sigma = (\sigma_1,\ldots,\sigma_n) \in G^n$ and $v \in V_i$.

Let $\sigma \in G^{n+1}$ and $v \in V_2$. Then
$\big(S'_2(v)\big)(\sigma) = \big(S'_1(T(v))\big)(\sigma)$,
hence
$$
\big(f_2'(\sigma)\big)(v) = \big(f_1'(\sigma)\big)\big(T(v)\big)
=\Big( T^*\big(f_1'(\sigma)\big)\Big)(v).
$$
This is true for every $v \in V_2$,
so
$f_2'(\sigma) = T^*\big(f_1'(\sigma)\big)$
for every $\sigma$,
whence $f_2' = T \circ f_1'$.
Thus $T^*$ is a morphism in $\calH^n$.
Put $Y^n(T) \defeq T^*$.

This completes the definition of $Y^n \colon \T^n \to \calH^n$.
\end{Construction}

\begin{Proposition}\label{equivalence of categories}
The categories $\T^n$ and $\calH^n$ are dually equivalent,
via the functors $X^n$ and $Y^n$.
\end{Proposition}

\begin{proof}
We use the natural isomorphisms defined in 
Definition~\ref{natural isomorphisms}.

Let $(K,f) \in \calH^n$,
let $(V,S) = X^n(K,f)$,
and put $(C,g) = Y^n(V,S)$.
Then $C = \Theta_K(V)$ and $g = \Theta_K \circ f$.
Indeed,
$V = K^*$ and $C = V^* = K^{**}$;
by Theorem~\ref{Duality}, $K^{**} = \Theta_K(K)$.
Let $f' \in \Zbar^n(G,K)$ lift $f$.
Then $S$ is induced from $S' \colon K^* \to \Zbar^n(G,A)$
given by
$S'(\phi) = \phi \circ f'$
and $g$ is induced from $g' \in \Zbar^n(G,V^*)$ given by
$\big(g'(\sigma)\big)(\phi) = \big(S'(\phi) \big)(\sigma)$
for every $\phi \in V = K^*$ and every $\sigma \in G^{n+1}$.
Hence
$$
\big(g'(\sigma)\big)(\phi) =
\big(S'(\phi)\big)(\sigma) =
(\phi \circ f')(\sigma) =
\phi \big(f'(\sigma)\big)
.
$$
On the other hand,
by the definition of $\Theta_K$,
$$
[(\Theta_K \circ f')(\sigma)](\phi) =
\big[\Theta_K\big(f'(\sigma)\big)\big](\phi) =
\phi \big(f'(\sigma)\big)
.
$$
Hence
$g' = \Theta_K \circ f'$.
It follows that
$g = \Theta_K \circ f$.

Conversely,
let $(V,S) \in \T^n$,
let $(K,f) = Y^n(V,S)$,
and $(W,R) = X^n(K,f)$.
Then $W = \Lambda_{V}(V)$ and $R \circ \Lambda_{V} = S$.
Indeed,
$K = V^*$ and $W = K^* = V^{**}$;
by Theorem~\ref{Duality}, $V^{**} = \Lambda_{V}(V)$.
Lift $S$ to an $F$-linear map $S' \colon V \to \Zbar^n(G,A)$.
Then $f$ is induced from $f' \in \Zbar^n(G,V^*)$ given by
$f'(\sigma)(v) = \big(S'(v)\big)(\sigma)$,
and $R$ is induced from $R' \colon K^* \to \Zbar^n(G,A)$ given by
$R'(\phi) = \phi \circ f'$.
Let $v \in V$ and fix $\sigma \in G^{n+1}$.
Then
$R'\big(\Lambda_{V}(v)\big) = \Lambda_{V}(v) \circ f'$,
hence
$$
R'\big(\Lambda_{V}(v)\big)(\sigma)=
\big(\Lambda_{V}(v) \circ f'\big)(\sigma) =
\Lambda_{V}(v)\big(f'(\sigma)\big) =
\big(f'(\sigma)\big)(v) =
\big(S'(v)\big)(\sigma),
$$
whence
$R'\circ \Lambda_{V}(v) = S'(v)$.
It follows that
$R\circ \Lambda_{V}(v) = S(v)$.
\end{proof}

\begin{Example}\label{ex1}
Let $I$ be a set.
For every $i \in I$ let
\begin{equation}\label{individual extension}
\xymatrix{
0 \to A \ar[r] & H_i \ar[r]^{\eta_i} & G \to 1}
\end{equation}
be an extension of $G$ by the simple $G$-module $A$
and let
$f_i \in H^2(G,A)$
be the corresponding cocycle.
Let
\begin{equation}\label{fiber product extension}
\xymatrix{
0 \to K \ar[r] & \fprod{i \in I} H_i \ar[r]^{\eta_I} & G \to 1}
\end{equation}
be the fiber product extension
and let
$f \in H^2(G,A^I)$
be the corresponding cocycle.
Then $K = A^I$ and $(K,f) \in \calH^2$.
What is $(K^*, S_{K,f}) = X^2(A^I,f)$?

Let $\phi_i \colon K \onto A$
be the projection on the $i$-th coordinate.
By Lemma~\ref{basis for a free product},
$K^* = \bigoplus_{i \in I} F \phi_i$.
By Lemma~\ref{cocycle of a fiber product},
$f$ is given by
$\phi_i \circ f = f_i$, for every $i \in I$.
By Construction~\ref{two functors}(a),
$S_{K,f}(\phi_i) = f_i$,
for every $i \in I$.
This completely determines
$(K^*, S_{K,f})$,
up to isomorphism.

By elementary linear algebra 
$\dim K^* = |I|$
and
$\dim \Ker S_{K,f} = |I \smallsetminus I'|$,
where
$I'$ is a subset of $I$
such that
$(f_i)_{i \in I'}$
is a maximal linearly independent subset of
$(f_i)_{i \in I}$.
\end{Example}

\begin{Corollary}\label{A-gen ext is fiber product}
Let $\eta \colon H \onto G$
be an epimorphism
such that $K := \ker \eta$
is an $A$-generated module.
Let $\{\psi_i\}_{i \in I}$ be a basis of $K^*$.
Denote $f_i = S(\psi_i) \in H^2(G,A)$
and let \eqref{individual extension} be the corresponding extension,
with $\eta_i$ indecomposable,
for every $i \in I$.
Let $\fprod{i \in I} H_i$
be the fiber product of
$(\eta_i \colon H_i \onto G)_{i \in I}$.
Then there is an isomorphism
$\beta \colon H \to \fprod{i \in I} H_i$
such that
$\eta_I \circ \beta = \eta$.
\end{Corollary}

\begin{proof}
Let $f \in H^2(G,K)$
be the cocycle corresponding to $\eta$.
Then $(K,f) \in \calH^2$.
Put $(K^*,S) = X^2(K,f) \in \T^2$;
then $(K,f) = Y^2(K^*,S)$.

Then the fiber product extension
\eqref{fiber product extension} corresponds,
by Example~\ref{ex1}, to
$(\bigoplus_{i \in I} F \phi_i, S')$,
where
$S'(\phi_i) = f_i$,
for every $i \in I$.
The isomorphism
$\beta^* \colon K^* \to \bigoplus_{i \in I} F \phi_i$,
given by $\psi_i \mapsto \phi_i$,
for every $i \in I$,
satisfies
$S = S' \circ \beta^*$,
so it induces an isomorphism
$(\bigoplus_{i \in I} F \phi_i, S') \to (K^*,S)$
in $\T^2$.
By the duality
there is an isomorphism
$(K,f) \to (A^I, g)$ in $\calH^2$,
where $g$ is the cocycle that corresponds to the fiber product.
Thus
there is an isomorphism
$\beta \colon H \to \fprod{i \in I} H_i$
such that
$\eta_I \circ \beta = \eta$.
This isomorphism satisfies
$\beta^*(\phi) = \phi \circ \beta$,
for every $\phi \in (A^I)^*$.
\end{proof}

\begin{Remark}\label{basis choice}
In the above construction of the corresponding extension
\eqref{fiber product extension}
we have the freedom to choose the basis
$\{\psi_i\}_{i \in I}$
of $K^*$.
So we may assume that
$I = I_0 \dotcup I_1$,
where
$\{\psi_i\}_{i \in I_0}$ is a basis of $\Ker S$
and
$\{S(\psi_i)\}_{i \in I_1}$ is a basis of $S(K^*)$.
Then
$\{f_i\}_{i \in I_1}$ is a linearly independent subset
of $H^2(G,A)$
and
\eqref{individual extension} is a copy of the split extension,
for all $i \in I_0$.

Another choice of the basis gives the following:
\end{Remark}

\begin{Corollary}\label{isom with subgroup A}
Let $\eta \colon H \onto G$
be an epimorphism
such that $K := \ker \eta$
is an $A$-generated module.
Let $L$ be a normal subgroup of $H$ contained in $K$.
Then there is
a family of indecomposable epimorphisms
$\calH = (\eta_i\colon H_i\onto G \st i \in I)$,
with kernel $A$,
and a $G$-isomorphism
$\beta \colon H \to \fprod{i \in I} H_i$
such that
$\eta_I \circ \beta = \eta$
and
$\beta(L) = \{(k_i)_{i \in I} \in \prod_{i \in I} \Ker \eta_i \st
k_i = 0 \textnormal{ for all } i \notin I'\}
= \prod_{i \in I'} \Ker \eta_i$,
for some $I' \subseteq I$.
\end{Corollary}

\begin{proof}
By Corollary~\ref{A-generated submodule complement},
$K = L \directsum M$
for some normal subgroup $M$ of $H$ contained in $K$.
Let
$\{\psi_i\}_{i \in J}$
be a basis of $M^*$,
and extend every $\psi_i$ to $\psi_i \in K^*$
by $\psi_i|_L = 0$.
Complete $\{\psi_i\}_{i \in J}$ to a basis
$\{\psi_i\}_{i \in I}$
of $K^*$.
Let $I' = I \smallsetminus J$.


Put $(K^*,S) = X^2(K,f) \in \T^2$;
then $(K,f) = Y^2(K^*,S)$.

The construction in the proof of
Corollary~\ref{A-gen ext is fiber product},
gives an isomorphism
$\beta \colon H \to \fprod{i \in I} H_i$
such that
$\eta_I \circ \beta = \eta$.
Moreover,
the isomorphism
$\beta^* \colon K \to K$,
given by
$\beta^*(\phi) = \phi \circ \beta$,
for every $\phi \in K^*$,
maps $\phi_i$ onto $\psi_i$,
for every $i \in I$.

Let $x \in K$.
Then
\begin{multline}
x \in L
\iff
\psi(x) = 0 \textnormal{ for all } \psi \in M^*
\iff
\psi_i(x) = 0 \textnormal{ for all } i \in J
\iff
\\
\iff
\phi_i(\beta(x)) = 0 \textnormal{ for all } i \in J
\iff
\beta(x) \in \prod_{i \in I'} \Ker \eta_i.
\end{multline}
%
%
Thus $\beta(L) = \prod_{i \in I'} K_i$.
\end{proof}

\begin{Lemma}\label{duality and inflation}
Let $(K,f) \in  \calH^n(G)$.
Let $\theta \colon \Ghat \onto G$ be an epimorphism
and view the $G$-modules $A$, $K$
as $\Ghat$-modules via $\theta$.
Let
$\fhat = \Inf_\theta(f) \in H^n(\Ghat,K)$.
Then the following diagram commutes
$$
\xymatrix{
K \ar@{=}[r] \ar[d]^{S_{K,f}} & K \ar[d]^{S_{K,\fhat}}
\\
H^n(G, A) \ar[r]^{\Inf_\theta} & H^n(\Ghat, A)
}
$$
\end{Lemma}

\begin{proof}
Let $f' \in \Zbar^n(G,K)$
be a representative of $f$.
Then $f' \circ \theta \in \Zbar^n(\Ghat,K)$
is a representative of $\fhat = \Inf_\theta(f)$.
As $\Inf_\theta$ is induced from
$\Inf_\theta \colon \Zbar^n(\Ghat,K) \to \Zbar^n(G,K)$,
it suffices to show that
$\Inf_\theta \circ S_{K,f'} = S_{K,f' \circ \theta}$.

Let $\phi \in K^*$. Then
$S_{K,f'}(\phi) = \phi \circ f'$,
hence
$$
(\Inf_\theta \circ S_{K,f'})(\phi) = 
\Inf_\theta(\phi \circ f') =
(\phi \circ f') \circ \theta =
\phi \circ (f' \circ \theta) =
S_{K,f' \circ \theta}(\phi).
$$
\end{proof}

\section{Multiple fiber products of indecomposable epimorphisms}
\label{mfp II}

The preceding section has established some properties of
fiber products
of families of indecomposable epimorphisms,
all of them
with the same abelian kernel.
We now want to extend it to general families
of indecomposable epimorphisms.

We first fix the notation.

\begin{Notation}\label{Lambda}
For a profinite group $G$
let
\begin{itemize}
\item[(a)]
$\Lambda_\ab(G)$ be the set of isomorphism classes of
finite simple $G$-modules;
\item[(b)]
$\Lambda_\na(G)$
be a set of representatives of the isomorphism classes 
of indecomposable epimorphisms onto $G$ 
with non-abelian kernel;
\item[(c)]
Put $\Lambda(G) = \Lambda_\ab(G) \dotcup \Lambda_\na(G)$.
\end{itemize}
(Although we will not need this,
notice that
$\Lambda_\na(G)$ is the non-abelian analogue of $\Lambda_\ab(G)$.
Indeed,
every indecomposable epimorphism $\eta \colon H \onto G$
with a non-abelian kernel 
is uniquely determined by its kernel $A$
and the outer $G$-action on $A$
(\cite[Theorem 6.6]{Br}) --
the non-abelian analogue of a $G$-module.)
\end{Notation}

\begin{Notation}\label{notation1}
In this section let
$\calH = (\eta_i\colon H_i\onto G \st i \in I)$
be a family of \textit{indecomposable} epimorphisms
and let
$\fprod{i \in I} H_i$
be its fiber product.
We keep and enhance
the notation from Section~\ref{mfp I}.
So,
\begin{itemize}
\item[(a)]
$\eta_I \colon \fprod{i \in I} H_i \onto G$
is the structure map of 
$\fprod{i \in I} H_i$
and
$K_j$
is the kernel of
$\pr_{I,I\smallsetminus\{j\}} \colon
\fprod{i \in I} H_i \onto \fprod{i\in I\smallsetminus\{j\}} H_i$,
for every $j \in I$.
\end{itemize}
Furthermore,

\begin{itemize}[resume]
\item[(b)]
$
\calC'(\calH) = 
\{(H,\{p_i\}_{i \in I}, \pbar) \st
H \textnormal{ is a group and }
\pbar \colon H \onto G
\textnormal{ and }
\\
p_i \colon H \onto H_i
\textnormal{ satisfy }
\eta_i \circ p_i = \pbar,
\textnormal{ for every } i \in I\} =
\\
\{(H,\{p_i\}_{i \in I}, \pbar) 
\in \calC(\calH) \st
p_i \textnormal{ is an epimorphism, for every } i \in I \}.
$
\item[(c)]
For every $J \subseteq I$ denote
$K_J = \prod_{i \in J} K_i$;
in particular,
$K_I = \Ker \eta_I$.
\item[(d)]
For every $A \in \Lambda_\ab(G)$ let
$I_A = \{i \in I \st \Ker \eta_i \isom_G A\}$
and
\\
$F_A = \End_G(A)$.
\item[(e)]
Let
$I_{\na} = \{i \in I \st \Ker \eta_i \in \Lambda_\na(G)\}$.
\item[(f)]
For every $\eta \in \Lambda_\na(G)$ let
$I_\eta = \{i \in I \st \eta_i \isom_G \eta\}$.
\end{itemize}
\end{Notation}

\begin{Lemma}\label{min normal subgr of fprod}
Let $L$ be a minimal normal subgroup of
$\fprod{i \in I} H_i$
contained in $K_I$.
Then
\begin{itemize}
\item[(a)]
if $L$ is non-abelian,
then $L = K_j$ for some $j \in I_\na$;
\item[(b)]
if $L$ is abelian,
then
$L \in \Lambda_\ab(G)$,
and $L \le K_{I_L}$.
\end{itemize}
\end{Lemma}

\begin{proof}
Let $j \in I$.
If $K_j \ne L$,
then $K_j \cap L = 1$,
because $K_j \cap L$ is a proper subgroup of the minimal $K_j$.
Then $[K_j, L] =1$.

(a)
Assume that $L$ is non-abelian.
If $K_j \ne L$,
then, by the first paragraph of this proof,
$1 = [\pr_{I,j}(K_j), \pr_{I,j}(L)] =
[\Ker \eta_j, \Ker \eta_j]$,
a contradiction to $\Ker \eta_j \isom L$ being non-abelian.

(b)
Assume that $L$ is abelian.
Then, by the first paragraph of this proof,
$[K_i, L] =1$,
that is, $K_i \le C_{K_I}(L)$,
for every $i \in I$.
So
$K_I \le C_{K_I}(L)$,
that is,
$L \le Z(K_I)$.
By Lemma~\ref{action}(a),
$L$ is a $G$-module, that is, $L \in \Lambda_\ab(G)$.

Let $j \in I \smallsetminus I_L$.
Then 
$\pr_{I,j}(L) = 1$.
Indeed, assume the contrary.
As $L \le K_I$ and 
$\pr_{I,j}(K_I) = \Ker \eta_j$,
we have 
$\pr_{I,j}(L) \le \Ker \eta_j$.
So, by Lemma~\ref{image and preimage}(a),
$\pr_{I,j}(L) = \Ker \eta_j$
and
$L \isom \Ker \eta_j$
via $\pr_{I,j}$.
By Lemma~\ref{action}(d)
this is an isomorphism of $G$-modules,
and hence $j \in I_L$,
a contradiction.

Thus
$L \le \prod_{j \in I_L} K_j = K_{I_L}$.
\end{proof}

\begin{Proposition}\label{normal subgr of fprod}
Let $H = \fprod{i \in I} H_i$
and let $L$ be a normal subgroup of $H$ contained in $K_I = \Ker \eta_I$.
Then
$K_I = 
( \prod_{i \in I_{\na}} K_i )
\times
\prod_{A \in \Lambda_\ab(G)} K_{I_A}$
and
\begin{equation}\label{presentation}
L
=
\prod_{i \in I_{\na}} (L \cap K_i )
\times
\prod_{A \in \Lambda_\ab(G)} \big(L \cap K_{I_A} \big)
=
\prod_{i \in I_{\na,L}} K_i
\times
\prod_{A \in \Lambda_\ab(G)} \big(L \cap K_{I_A} \big),
\end{equation}
where
$I_{\na,L} = \{i \in I_{\na} \st K_i \le L\}$.
\end{Proposition}

\begin{proof}
The presentation of $K_I$
follows from
$I = I_{\na} \dotcup \bigdotcup_{A \in \Lambda_\ab(G)} I_A$.
If $i \in I_{\na}$, then
$
L \cap K_i = 
\begin{cases}
K_i & \textnormal{ if $K_i \le L$} \\
1 & \textnormal{ if $K_i \not\le L$}
\end{cases}
$,
because $K_i$ is a minimal normal subgroup in $H$.
Thus the two presentations of $L$ in \eqref{presentation} are equivalent.
We note that the above decomposition of $I$ also implies,
by Remark~\ref{properties of fiber product}(d),
that
\begin{equation}\label{decomp of H}
H = (\fprod{i \in I_{\na}} H_i)
\times_G
\fprod{A \in \Lambda_\ab}(\fprod{i \in I_A} H_i).
\end{equation}

First assume that $I$ is finite.
Then $K_I$ and $L$ are finite.
The case $L = 1$ being trivial,
we assume $|L| > 1$ and proceed by induction on $|L|$.
There is a minimal normal subgroup $N$ of $H$
contained in $L$.
By Lemma~\ref{min normal subgr of fprod}:
If $N$ is non-abelian,
there is $j \in I_{\na}$ such that
$N = K_j$.
If $N$ is abelian,
there is a unique $A \in \Lambda_\ab(G)$ such that $N \isom_G A$,
and
$N \le K_{I_A}$;
by Corollary~\ref{isom with subgroup A}
we may assume that 
$N = K_j$ for some $j \in I_A$.
Thus in both cases
$N$ is contained in the right handed side of
\eqref{presentation}.

Put $J = I \smallsetminus \{j\}$
and $\Hbar = \fprod{i \in J}H_i$.
Then $N$ is the kernel of
$\pr_{I,J} \colon H \onto \Hbar$.
By the induction hypothesis, the image $\Lbar$ of $L$ in $\Hbar$
is of the form
$$
\Lbar
=
\prod_{i \in J_{\na,\Lbar}} K_i \times
\prod_{A \in \Lambda_\ab(G)} \big( \Lbar \cap \prod_{i \in I_A \smallsetminus \{j\}} K_i \big),
$$
where
$J_{\na,\Lbar} = \{i \in I_{\na} \smallsetminus \{j\} \st
K_i \le \Lbar\}
= 
I_{\na,L} \smallsetminus \{j\}
$.
The right handed side of \eqref{presentation}
maps onto $\Lbar$
and contains $N$.
Hence by the third isomorphism theorem
that group is $L$.

Now assume that $I$ is infinite.
Then 
$H = \varprojlim_{I'} \fprod{i \in I'} H_i$
and
$K_I = \varprojlim_{I'} \prod_{i \in I'} K_i$,
where $I'$ runs through the collection of finite subsets of $I$.
The kernel of
$\pr_{I,I'} \colon H \onto \fprod{i \in I'} H_i$
is $K_{I \smallsetminus I'}$.
By the finite case,
the image
$L K_{I \smallsetminus I'}/K_{I \smallsetminus I'}$
of $L$ in $\fprod{i \in I'} H_i$
is of the form
\begin{multline*}
\prod_{i \in I'_{\na}}
(L K_{I \smallsetminus I'}/K_{I \smallsetminus I'})
\cap
(K_i K_{I \smallsetminus I'}/K_{I \smallsetminus I'})
\quad \times \quad
\\
\prod_{A \in \Lambda_\ab(G)}
(L K_{I \smallsetminus I'}/K_{I \smallsetminus I'})
\cap
(K_{I'_A} K_{I \smallsetminus I'}/K_{I \smallsetminus I'})
=
\\
=
\prod_{i \in I'_{\na}}
(L K_{I \smallsetminus I'}
\cap
K_i K_{I \smallsetminus I'}) /K_{I \smallsetminus I'}
\quad \times
\prod_{A \in \Lambda_\ab(G)}
(L K_{I \smallsetminus I'}
\cap
K_{I'_A} K_{I \smallsetminus I'}) /K_{I \smallsetminus I'}.
\end{multline*}
If $I' \supseteq I''$,
then the map
$\fprod{i \in I'} H_i \onto \fprod{i \in I''} H_i$
of the inverse system
maps
the factors into the corresponding factors,
that is,
\begin{itemize}
\item[$\bullet$]
$
(L K_{I \smallsetminus I'}
\cap
K_i K_{I \smallsetminus I'}) /K_{I \smallsetminus I'}
$
into
$
(L K_{I \smallsetminus I''}
\cap
K_i K_{I \smallsetminus I''}) /K_{I \smallsetminus I''}
$,
for every $i \in I'_{\na}$;
and
\item[$\bullet$]
$
(L K_{I \smallsetminus I'}
\cap
K_A K_{I \smallsetminus I'}) /K_{I \smallsetminus I'}
$
into
$
(L K_{I \smallsetminus I''}
\cap
K_A K_{I \smallsetminus I''}) /K_{I \smallsetminus I''}
$,
for every $A \in \Lambda_\ab(G)$.
\end{itemize}
Therefore
$L = \prod_{i \in I_{\na}} M_i \times \prod_{A \in \Lambda_\ab(G)} M_A$,
where,
for every $i \in I'_{\na}$,
$$
M_i =
\bigcap_{I'} 
(L K_{I \smallsetminus I'} \cap K_i K_{I \smallsetminus I'})
=
\big(
\bigcap_{I'} L K_{I \smallsetminus I'}
\big)
\cap
\big(
\bigcap_{I'} 
K_i K_{I \smallsetminus I'} 
\big)
=
L \cap K_i.
$$
Similarly,
$M_A = L \cap K_A$
for every simple $G$-module $A$.
This gives the desired presentation of $L$.
\end{proof}

We can simplify the presentation of $L$
if we present
$\fprod{i \in I} H_i$
as a fiber product of a different family
of indecomposable epimorphisms
(which depends on $L$):

\begin{Corollary}\label{isom with subgroup gen}
Let $H = \fprod{i \in I} H_i$
and let $L$ be a normal subgroup of $H$ contained in $\Ker \eta_I$.
Then there is a family
$\bar\calH = (\etabar_i \colon \Hbar_i \onto G)_{i \in I}$
of indecomposable epimorphisms
and an isomorphism
$\omega \colon 
\fprod{i \in I} H_i \to \fprod{i \in I} \Hbar_i$
such that
$\etabar_I \circ \omega = \eta_I$
and $\omega(L) = \prod_{i \in \Ibar} \Kbar_i$,
for some $\Ibar \subseteq I$.

Here
$\etabar_I \colon \fprod{i \in I} \Hbar_i \onto G$
is the fiber product of $\bar\calH$,
with projection maps
$\prbar_{I,\Ibar} \colon \fprod{i \in I} \Hbar_i \onto \fprod{i \in \Ibar} \Hbar_i$,
for $\Ibar \subseteq I$,
and
$\Kbar_i = \Ker \prbar_{I,I \smallsetminus\{i\}}$,
for every $i \in I$.
\end{Corollary}

\begin{proof}
If $I = I_A$ for some $A \in \Lambda_\ab(G)$,
the assertion follows by
Corollary~\ref{isom with subgroup A}.

In the general case
$I = I_\na \dotcup \bigdotcup_{A \in \Lambda_\ab(G)} I_A$,
so, by Remark~\ref{properties of fiber product}(d),
$$
\fprod{i \in I} H_i =  (\fprod{i \in I_\na} H_i) \times_G
\fprod{A \in \lambda_\ab(G)} (\fprod{i \in I_A} H_i).
$$
By Proposition~\ref{normal subgr of fprod},
$$
L = \prod_{i \in I_{\na,L}} K_i
\times
\prod_{A \in \Lambda_\ab(G)} \big(L \cap K_{I_A} \big).
$$
By the first paragraph,
for every $A \in \Lambda_\ab(G)$
there is an isomorphism
$\omega_A \colon \fprod{i \in I_A} H_i \onto \fprod{i \in I_A} \Hbar_i$
that maps $L \cap K_{I_A}$ onto $\prod_{i \in \Ibar_A} \Kbar_i$,
for some $\Ibar_A \subseteq I_A$.
Then
the identities of $H_i$,
for $i \in I_\na$,
together with the isomorphisms
$\omega_A$, 
for $A \in \Lambda_\ab(G)$,
define an isomorphism
$\omega$
that maps $L$ onto
$\prod_{i \in \Ibar} \Kbar_i$,
where
$\Ibar = I_{\na,L} \dotcup \bigdotcup_{A \in \Lambda_\ab(G)} \Ibar_A
\subseteq I$.
\end{proof}

Theorem~\ref{indecomposable quotients of a fiber product pi is id}
below is the main technical result of this paper:
It characterizes the indecomposable covers of 
a profinite group $G$
dominated by a fiber product of indecomposable covers of $G$.
We deduce it from an even more general characterization --
of indecomposable covers of a quotient of $G$.

\begin{Theorem}\label{indecomposable quotients of a fiber product}
Let
$\pi \colon G \onto D$,
and let
$\zeta \colon E \onto D$ be indecomposable.
Put $C =\Ker \zeta$;
if $C$ is non-abelian,
let
$\zetahat = \Inf_\pi \zeta$
and
$I_{\zetahat} = \{i \in I \st \eta_i \isom_G \zetahat\}$,
and
if $C$ is abelian,
let
$\Chat = \Inf_\pi C$
and
$I_{\Chat} = \{i \in I \st \Ker \eta_i \isom_G \Chat\}$.
Then
\begin{itemize}
\item[(*)]
there is
$\eps \colon \fprod{i \in I} H_i \onto E$
such that
the following square is semi-cartesian
\begin{equation}\label{start}
\xymatrix{
\fprod{i \in I} H_i 
\arr[rr]_{\eps} \arr[d]^{\eta_I} && E \arr[d]^{\zeta}
\\
G \arr[rr]^{\pi} && D \\
}
\end{equation}
\end{itemize}
if and only if
exactly one of the following two conditions holds:
\begin{itemize}[resume]
\item[(a)]
$C$ is non-abelian
and
$I_{\zetahat} \ne \emptyset$.
\item[(b)]
$C$ is abelian,
$I_{\Chat} \ne \emptyset$,
and $\zetahat$ is
a nontrivial linear combination
of $(\eta_i)_{i \in I_{\Chat}}$
over $F_{\Chat} = \End_G(\Chat)$ \
(in the sense of Remark~\ref{cohomology is a vector space}(b)).
\end{itemize}
Moreover,
if \eqref{start} is semi-cartesian,
then
\begin{itemize}[resume]
\item[(a')]
if $C$ is non-abelian,
there is a unique $j \in I$
such that
$\Ker \pr_{I,j} = \Ker \eps \cap \Ker \eta_I$;
this $j$ is in $I_{\zetahat}$ and 
there is $\eps_j \colon H_j \onto E$
such that
\eqref{one j} below commutes
and
the square in it is cartesian.

\item[(b')]
if $C$ is abelian,
there is
$\eps' \colon \fprod{i \in I_{\Chat}} H_i \onto E$
such that
$\eps = \eps' \circ \pr_{I,I_{\Chat}}$.
\end{itemize}
\vskip-.5cm
\noindent\begin{minipage}{.42\linewidth}
\begin{equation}\label{one j}
\xymatrix@=6pt{
\fprod{i \in I} H_i 
\arr@/_1pc/[dddr]_{\eta_I}
\arr[rd]^(.65){\pr_{I,j}}
\arr@/^1pc/[rrrrrd]^{\eps}
\\
& H_j \arr[rrrr]_{\eps_j} \arr[dd]^{\eta_j} &&&& E \arr[dd]^{\zeta}
\\
\\
& G \arr[rrrr]^{\pi} &&&& D
}
\end{equation}
\end{minipage}%
\begin{minipage}{.58\linewidth}
\begin{equation}\label{zetahat}
\xymatrix@=8pt{
\fprod{i \in I} H_i 
\arr[rd]^(.65){\pr_{I,I_\Chat}}
\arr@/_1pc/[rdddd]_{\eta_I}
\arr@/^1pc/[rrrrrrd]^{\eps}
\\
& \fprod{i \in I_\Chat} H_i \arr[rd]^(.65){\eps'} \arr[rrrrr]^{\epsbar} 
\arr[ddd]^{\eta_{I_\Chat}}
 &&&&& E \ar@{=}[d]
\\
&& H' \arr[rrrr]_{\rho} \arr[dd]^{\zetahat} &&&& E \arr[dd]^{\zeta}
\\
\\
& G \ar@{=}[r] & G \arr[rrrr]^{\pi} &&&& D
}
\end{equation}
\end{minipage}
\end{Theorem}

\begin{proof}
Let $H = \fprod{i \in I} H_i$.
First assume (*).
Let $K = \Ker \eta_I$ and $N = \Ker \eps$.
As \eqref{start} is semi-cartesian,
$\Ker \pi \circ \eta_I = \Ker \zeta \circ \eps = KN$.
Put $L = K \cap N$.
By Corollary~\ref{indecomposable semi},
$K \not\le N$, 
hence $L \subsetneqq K$.
We may assume that
$E = H/N$ and $\eps$ is the quotient map
$H \onto H/N$,
and $\zeta$ is the quotient map
$H/N \onto H/KN$.
Hence
$C = \Ker \zeta = K/L$.
By Proposition~\ref{normal subgr of fprod}
(using its notation
$\{K_i\}_{i \in I}$, $I_{\na}$, $I_{\na,L}$, $I_A$)
\begin{equation}\label{K and L}
K = 
\big( \prod_{i \in I_{\na}} K_i \big)
\times
\prod_{A \in \Lambda_\ab(G)} \big(\prod_{i \in I_A} K_i \big)
\ \textnormal{ and } \
L
=
\prod_{i \in I_{\na,L}} K_i \times
\prod_{A \in \Lambda_\ab(G)} \big( L \cap \prod_{i \in I_A} K_i \big),
\end{equation}
and hence,
by the second isomorphism theorem for groups,
\begin{equation}\label{C}
C = K/L =
\prod_{i \in I_{\na} \smallsetminus I_{\na,L}} K_i \times
\prod_{A \in \Lambda_\ab(G)}
\big( \prod_{i \in I_A} K_i \big) L /L .
\end{equation}
The right handed side is a direct product
of normal subgroups in $E$.
However, as $C$ is a minimal normal subgroup in $E$,
exactly one direct factor in this presentation
is non-trivial.
So, there are two cases:

(i)
$C$ is non-abelian.
Then the nontrivial factor in \eqref{C} is $K_j$
for a unique $j \in I_{\na}$.
Hence
$I_{\na} \smallsetminus I_{\na,L} = \{j\}$
and
$\prod_{i \in I_A} K_i \le L$
for every $A \in \Lambda_\ab(G)$.
Then
$K_j$ is non-abelian and \eqref{K and L} gives
$$
K \cap N =
L =
\prod_{i \in I_{\na \smallsetminus \{j\}}} K_i \times
\prod_{A} \prod_{i \in I_A} K_i 
=
\prod_{i \in I\smallsetminus \{j\}} K_i 
=
\Ker \pr_{I,j}.
$$
Thus
there is $\eps_j \colon H_j \onto E$
such that
$\eps_j \circ \pr_{I,j}= \eps$.
From the commutativity of \eqref{start} we deduce that
\eqref{one j} commutes.
By Lemma~\ref{expand}(a)
the square in \eqref{one j}
is semi-cartesian.
As
$\Ker \eps \cap \Ker \eta_I = \Ker \pr_{I,j}$,
we have
$\Ker \eps_j \cap \Ker \eta_j = 1$,
hence the square in \eqref{one j} is cartesian.
Thus
$\zetahat =\Inf_\pi \zeta \isom_G \eta_j$,
and hence
$j \in I_{\zetahat} \ne \emptyset$.
Finally, $j \in I$ with
$L = \Ker \pr_{I,j}$
is unique,
because
distinct $j$'s in $I$ give distinct
$\Ker \pr_{I,j}$'s.
This gives (a) and (a').

(ii)
$C$ is abelian.
Then
$I_{\na} \smallsetminus I_{\na,L} = \emptyset$,
and
$\prod_{i \in I_A} K_i \le L$
for all but a unique $A \in \Lambda_\ab(G)$.
This $A$ satisfies
$\prod_{i \in I_A} K_i \not\le N$,
that is,
$\eps(\prod_{i \in I_A} K_i) \ne 1$,
so there is $j \in I_A$
(hence $A \isom_G  K_j$)
such that
$\eps(K_j) \ne 1$.
As $K_j \le K = \Ker \eta_I$,
we have $\eps(K_j) \le \eps(K) = \Ker \zeta = C$.
By Lemma~\ref{image and preimage}(a),
the restriction $\eps|_{K_j} \colon K_j \to C$
is an isomorphism.
In particular,
$K_j \isom_G \Chat$,
by Lemma~\ref{action}(e),
and hence
$A \isom_G \Chat$.
So $j \in I_{\Chat}$,
whence $I_{\Chat} \ne \emptyset$,
and 
$$
L =
\prod_{i \in I_{\na}} K_i \times
\prod_{A' \ne \Chat} \prod_{i \in I_{A'}} K_i
\ \times 
\big( L \cap \prod_{i \in I_{\Chat}} K_i \big)
\supseteq
\kern-3pt
\prod_{i \in I_{\na}} K_i \times
\kern-2pt
\prod_{A' \ne \Chat} \prod_{i \in I_{A'}} K_i 
=
\Ker \pr_{I,I_{\Chat}}.
$$

So $\eps \colon H \onto E$
induces 
$\epsbar \colon \fprod{i \in I_{\Chat}} H_i \onto E$
such that
$\epsbar \circ \pr_{I,I_{\Chat}} = \eps$
in \eqref{zetahat}.
It follows that
$
\zeta \circ \epsbar \circ \pr_{I,I_{\Chat}} =
\zeta \circ \eps =
\pi \circ \eta_I = \pi \circ \eta_{I_{\Chat}} \circ \pr_{I,I_{\Chat}}
$,
and hence
$\zeta \circ \epsbar = \pi \circ \eta_{I_{\Chat}}$.
By Lemma~\ref{expand}(a)
the square formed by $\epsbar,\zeta, \eta_{I_{\Chat}}, \pi$
is still semi-cartesian.

Let $H' = G \times_{D} E$
and let $\rho \colon H' \onto E$ be the projection
of the fiber product.
The other projection
$H' \onto G$ is, by definition, $\zetahat = \Inf_\pi(\zeta)$.
The kernel of $\zetahat$ is $\Chat$
and $\rho$ maps $\Chat$ isomorphically onto $C$.
There is a unique homomorphism $\eps'$ such that 
\eqref{zetahat}
commutes.
By Lemma~\ref{expand}(b),
$\eps'$ is an epimorphism.

The restriction of $\eps'$ to
$K_{I_\Chat} = \Ker \eta_{I_\Chat}$
is a homomorphism $\beta$ of $G$-modules
such that
\begin{equation}\label{two extensions again}
\xymatrix{
0 \ar[r] & K_{I_\Chat} \ar[r] \ar[d]^{\beta}&
\fprod{i \in I_{\Chat}}H_i \arr[r]^{\eta_{I_\Chat}}
\arr[d]^{\eps'}& G \ar[r] \ar@{=}[d] & 1 \\
0 \ar[r] & \Chat \ar[r] & H' \arr[r]^{\zetahat} & G \ar[r] & 1. \\
}
\end{equation}
commutes.
Clearly, $\beta$ is surjective.
By Lemma~\ref{map of extensions},
$\beta_*(\eta_{I_\Chat}) = \zetahat$.

As $\beta$ is a $G$-homomorphism,
$\alpha_i := \beta|_{K_i}$ is a $G$-homomorphism,
for all $i \in I_\Chat$.
As $\beta$ is continuous,
$\alpha_i = 0$
for all but finitely many $i \in I_\Chat$.
So $\beta \colon K_{I_\Chat} = \prod_{i \in I_\Chat} K_i \onto \Chat$
is of the form
$(a_i)_{i \in I_\Chat} \mapsto \sum_i \alpha_i(a_i)$.
As $\beta$ is surjective,
not every $\alpha_i$ is zero.

For every $i \in I_\Chat$
choose $f_i \in \Zbar^2(G,\Chat)$ representing $\eta_i$.
By Lemma~\ref{cocycle of a fiber product},
$\eta_{I_\Chat} \colon H \to G$ is represented by a cochain
$g \in \Zbar^2(G,K_{I_\Chat})$
given by
$g(\sigma,\tau) = \big(f_i(\sigma,\tau)\big)_{i \in I_\Chat}$.
Then,
for all $\sigma, \tau \in G$,
$$
\big(\beta_*(g)\big) (\sigma,\tau) =
\beta\big(g(\sigma,\tau)\big) =
\beta\big(f_i(\sigma,\tau)_{i \in I_\Chat}\big) =
\sum_i \alpha_i f_i(\sigma,\tau) ,
$$
whence
$\beta_*(g) = \sum_i \alpha_i f_i$.
Hence
$\zetahat = \beta_*(\eta_{I\Chat}) = \sum_i \alpha_i \eta_i$.

This finishes the proof of (b) and (b').

Conversely,
suppose that (a) or (b) holds.
We have to show (*).

If (a) holds,
fix $j \in I_{\zetahat}$,
so that $\eta_j \isom_G \zetahat$.
By the definition of $\zetahat$
there is
$\eps_j \colon H_j \onto E$
be such that
the square in \eqref{one j} is cartesian.
Then
$\eps := \eps_j \circ \pr_{I,j}$
makes \eqref{one j} commutative,
and hence \eqref{start} commutes.
By Lemma~\ref{expand}(a) it is 
a semi-cartesian square.

If (b) holds,
the definition of $\zetahat$
gives a cartesian square
like the one in \eqref{zetahat},
and
$\zetahat = \sum_i \alpha_i \eta_i$
for some
$\{0\} \ne \{\alpha_i\}_{i \in I_{\Chat}} \subseteq F_{\Chat}$.
The map
$\beta \colon \prod_{i \in I_{\Chat}} K_i \to \Chat$
given by
$(a_i)_{i \in I_{\Chat}} \mapsto \sum_i \alpha_i(a_i)$
is a homomorphism of $G$-modules.
It is non-trivial and $\Chat$ is simple,
hence $\beta$ is surjective.
By Lemma~\ref{map of extensions}
it extends to 
$\eps' \colon \fprod{i \in I_{\Chat}} H_i \onto H'$
such that \eqref{two extensions again} commutes.
Let
$\epsbar = \rho \circ \eps'$
and
$\eps = \epsbar \circ \pr_{I,I_{\Chat}}$.
Then
\eqref{zetahat} commutes.
By Lemma~\ref{expand}(a),
\eqref{start} is a semi-cartesian square.
\end{proof}

We write out
the particular case of $\pi = \id_G$:

\begin{Theorem}
\label{indecomposable quotients of a fiber product pi is id}
Let
$\zeta \colon E \onto G$ be indecomposable.
Let $C =\Ker \zeta$;
if $C$ is non-abelian,
let
$I_\zeta = \{i \in I \st \eta_i \isom_G \zeta\}$,
and
if $C$ is abelian,
let
$I_C = \{i \in I \st \Ker \eta_i \isom_G C\}$.
%
%
Then there is
$\eps \colon \fprod{i \in I} H_i \onto E$
such that
$\zeta \circ \eps = \eta_I$
if and only if
exactly one of the following two conditions holds:
\begin{itemize}
\item[(a)]
$C$ is non-abelian
and
$I_\zeta \ne \emptyset$.
\item[(b)]
$C$ is abelian,
$I_C \ne \emptyset$,
and $\zeta$ is
a nontrivial linear combination
of $(\eta_i)_{i \in I_C}$
over $F_C = \End_G(C)$ \
(in the sense of Remark~\ref{cohomology is a vector space}(b)).
\end{itemize}
Moreover,
if $\eps$ is an epimorphism
such that
$\zeta \circ \eps = \eta_I$,
then
\begin{itemize}[resume]
\item[(a')]
if $C$ is non-abelian,
there is a unique $j \in I$
such that $\Ker \eps = \Ker \pr_{I,j}$;
this $j$ is in $I_\zeta$ and 
there is an isomorphism $\eps_j \colon H_j \to E$
such that
$\eps_j \circ \pr_{I,j}= \eps$
and
$\zeta \circ \eps_j = \eta_j$.

\item[(b')]
if $C$ is abelian,
there is
$\eps' \colon \fprod{i \in I_C} H_i \onto E$
such that
$\eps = \eps' \circ \pr_{I,I_C}$
and hence
$\zeta \circ \eps' = \eta_{I_C}$.
\end{itemize}
\end{Theorem}

\begin{Lemma}\label{predecompose}
For every $\lambda \in \Lambda(G)$
let $H_\lambda = \fprod{i \in I_\lambda} H_i$.
Then
$\fprod{\lambda \in \Lambda(G)} H_\lambda$,
the fiber product of
$(\eta_{I_\lambda} \colon H_\lambda \onto G)_{\lambda \in \Lambda(G)}$,
is compact.
\end{Lemma}

\begin{proof}
By Lemma~\ref{compact characterization}
we have to show
for every
finite $\Lambda \subseteq \Lambda(G)$,
$\mu \in \Lambda(G) \smallsetminus \Lambda$,
and
$\Lambdahat = \Lambda \dotcup \{\mu\}$
that
the square in 
\begin{equation}\label{simple partition lambda}
\xymatrix@=5pt{
\fprod{\lambda \in \Lambdahat} H_\lambda \arr[rrrr]_{\pr_{\Lambdahat,\Lambda}} \arr[ddd]_{\pr_{\Lambdahat,\mu}}
&&&& \fprod{\lambda \in \Lambda} H_\lambda \arr[ddd]^{\eta_{I_\Lambda}}
\dotarr[ddll]_(.6){\eps'}
\\
\\
&& E \dotarr[rrd]^{\zeta}
\\
H_\mu \arr[rrrr]_{\eta_{I_\mu}} \dotarr[rru]^{\eps}  &&&& G
}
\end{equation}
is compact.
Thus,
by Definition~\ref{compact cartesian square}(d),
we have to show that
there are no epimorphisms $\eps, \eps', \zeta$
such that
$\zeta$ is indecomposable,
$\zeta \circ \eps' = \eta_{I_\Lambda}$
and
$\zeta \circ \eps = \eta_{I_\mu}$.

Assume that there is such a diagram.
Then $\zeta$ is dominated both by $\eta_{I_\Lambda}$ and $\eta_{I_\mu}$.
By definition
$H_\mu$ is the fiber product of
$(\eta_i)_{i \in I_\mu}$;
by Remark~\ref{properties of fiber product}(d),
$\fprod{\lambda \in \Lambda} H_\lambda$
is the fiber product of 
$(\eta_i)_{i \in \bigdotcup_{\lambda \in \Lambda} I_\lambda}$.
If $C = \Ker \zeta$ is non-abelian,
then by Theorem~\ref{indecomposable quotients of a fiber product pi is id},
there is a unique $j \in I$
such that
$\eta_j \isom_G \zeta$
and 
$j \in I_\mu$ and 
$j \in \bigdotcup_{\lambda \in \Lambda}I_\lambda$,
so $\mu \in \Lambda$,
a contradiction.
If $C$ is abelian,
then by Theorem~\ref{indecomposable quotients of a fiber product pi is id},
$\mu = C$ and $C \in \Lambda$, a contradiction.
\end{proof}

\begin{Lemma}\label{predecompose na}
Let
$(H,\{p_i\}_{i \in I},\pbar) \in \calC'(\calH)$.
If $I = I_\na$
and 
$\Ker p_i \ne \Ker p_j$
for distinct $i,j \in I$,
then the induced map $p_I \colon H \to \fprod{i \in I} H_i$
is surjective.
\end{Lemma}

\begin{proof}
We have to show condition (b) of Lemma~\ref{surjectivity}.
So let $I' \dotcup \{j\} \subseteq I$
and assume that $p_{I'}$ in \eqref{ss} is surjective.
If \eqref{ss} is not semi-cartesian,
by Corollary~\ref{indecomposable semi}
there is $\eps \colon \fprod{i \in I'} H_i \onto H_j$
such that
$\eps \circ p_{I'} = p_j$,
and hence
$\eta_j \circ \eps = \eta_{I'}$.
By Theorem~\ref{indecomposable quotients of a fiber product pi is id}
there is a unique $i \in I'$ 
such that
$\Ker \eps = \Ker \pr_{I',i}$.
Thus
$p_{I'}^{-1}(\Ker \eps) = p_{I'}^{-1}(\Ker \pr_{I',i})$.
As
$\eps \circ p_{I'} = p_j$
and
$\pr_{I',i} \circ p_{I'} = p_i$,
the above equation is
$\Ker p_j = \Ker p_i$,
a contradiction.
\end{proof}

\begin{Proposition}\label{decompose}
Let 
$\calH = (\eta_i \colon H_i \onto G)_{i \in I}$
be a family of indecomposable epimorphisms
and let
$(H,\{p_i\}_{i \in I},\pbar) \in \calC'(\calH)$.
Assume that
$\Ker p_i \ne \Ker p_j$
for distinct $i,j \in I_\na$,
and that the induced map
$p_A \colon H \to \fprod{i \in I_A} H_i$
is surjective,
for every  $A \in \Lambda_\ab(G)$.
Then the induced map $p_I \colon H \to \fprod{i \in I} H_i$
is surjective.
If also
$\bigcap_{i \in I} \Ker p_i = 1$,
then $p_I$ is an isomorphism.
\end{Proposition}

\begin{proof}
As 
$I = \bigdotcup_{\lambda \in \Lambda(G)} I_\lambda$,
by Remark~\ref{properties of fiber product}(d) \
$\fprod{i \in I} H_i = \fprod{\lambda \in \Lambda(G)} H_\lambda$,
where
$H_\lambda = \fprod{i \in I_\lambda} H_i$
is the fiber product of
$(\eta_{I_\lambda} \colon H_\lambda \onto G)_{\lambda \in \Lambda(G)}$,
for all $\lambda \in \Lambda(G)$.

For every $\lambda \in \Lambda(G)$
let
$p_\lambda \colon H \to H_\lambda$
be the map induced from $(p_i)_{i \in I_\lambda}$.
Then $p_I$ is the map
$H \to \fprod{\lambda \in \Lambda(G)} H_\lambda$
induced from $(p_\lambda)_{\lambda \in \Lambda(G)}$.
If $\lambda \in \Lambda_\na(G)$,
then $p_\lambda$ is surjective
by Lemma~\ref{predecompose na}.
If $\lambda \in \Lambda_\ab(G)$,
then $p_\lambda$ is surjective by assumption.
By Lemma~\ref{predecompose},
$\fprod{i \in I_\lambda} H_i$
is compact.
Hence $p$ is surjective.

The last assertion follows from
Lemma~\ref{injective surjective}(a).
\end{proof}

\begin{Corollary}\label{compact of indec}
Assume that $(\eta_i)_{i \in I}$
are pairwise non-isomorphic
indecomposable epimorphisms,
$I \ne \emptyset$,
and for every abelian simple $G$-module $A$
the family
$T_A = (\eta_i \st i \in I,\ \Ker \eta_i \isom_G A,
\ \eta_i \textnormal{ does not split})$
is linearly independent over $\End_G(A)$.
Then the fiber product
$\fprod{i \in I} H_i$
is compact.
\end{Corollary}

\begin{proof}
Let $I'$ be a finite subset of $I$
and $j \in I \smallsetminus I'$.
Let
$\eps' \colon \fprod{i \in I'} H_i \onto G'$
and
$\pi \colon G' \onto G$
be epimorphisms,
$\pi$ indecomposable,
such that
$\pi \circ \eps' = \eta_{I'}$.
By Lemma~\ref{compact characterization}
it suffices to show that
there is no
$\eps_j \colon H_j \onto G'$
such that
$\pi \circ \eps_j = \eta_j$,
that is,
since $\eta_j$ is indecomposable,
that
$\eta_j \not\isom_G \pi$.

Denote
$A = \Ker \pi$.
By Theorem~\ref{indecomposable quotients of a fiber product}
either
\begin{itemize}
\item[(a)]
$A$ is non-abelian and there is $k \in I'$ such that
$\eta_k \isom_G \pi$;
or
\item[(b)]
$A$ is abelian and 
$\pi$ is a nontrivial linear combination of
$T'_A = \{\eta_i \st i \in I',\ \Ker \eta_i \isom_G A\}$.
\end{itemize}

In the first case
$j \ne k$,
hence 
$\eta_j \not\isom_G \eta_k$,
whence
$\eta_j \not\isom_G \pi$.

In the second case,
if $\eta_j$ is the split extension
(corresponding to $0 \in H^2(G,A)$),
then no $\eta_i \in T'_A$ can be split,
because $i \in I'$, and hence $i \ne j$.
Thus $T'_A \subseteq T_A$,
hence a nontrivial linear combination of
elements of $T'_A$
is not $0$,
whence not isomorphic to $\eta_j$.

If, in the second case,
$\eta_j$ is not split,
then $\eta_j \in T_A$, hence $\eta_j$
is linearly independent
of $T'_A$,
whence $\pi$ is not isomorphic to $\eta_j$.
\end{proof}

\section{Epimorphisms between fiber products
of indecomposable epimorphisms}
\label{epi}

Fix a profinite group $G$
and recall the notation
$\Lambda_\ab(G)$,
$\Lambda_\na(G)$,
and
$\Lambda(G)$
from Notation~\ref{Lambda}.

Let
$\calH = (\eta_i \colon H_i \onto G)_{i \in I}$
be a family
of indecomposable epimorphisms.
We will consider the map
$\eta_I \colon \fprod{i \in I} H_i \onto G$.
For every $A \in \Lambda_\ab(G)$
denote
$I_A = \{ i \in I \st \Ker \eta_i \isom_G A\}$
and
$H_A = \fprod{i \in I_A} H_i$.
Let
$B_A = \Ker(\eta_{I_A} \colon H_A \onto G)$,
let
$f_A \in H^2(G, B_A)$
be the cocycle corresponding to the extension
$1 \to B_A \to H_A \to G \to 1$,
and let $(V_A, S_A) = X^2(B_A, f_A)$,
where $V_A$ is a vector space over $F_A = \End_G(A)$.
Then 
$\Img S_A$ is an $F_A$-subspace of
$H^2(G,A)$
and
$\dim \Ker S_A$ 
is a cardinality.

For every indecomposable cover $\zeta$ of $G$
denote
$I_\zeta= \{ i \in I \st \eta_i \isom_G \zeta\}$.

Notice that
$I = \big(\bigdotcup_{\zeta \in \Lambda_\na(G)} I_\zeta\big)
\dotcup
\big( \bigdotcup_{A \in \Lambda_\ab(G)} I_A\big)$.

Now let
$\pi \colon G \onto G'$ 
be an epimorphism,
and let
$\calH' = (\eta'_i \colon H'_i \onto G')_{i \in I'}$
be a family
of indecomposable covers of $G'$.
For $A' \in \Lambda(G')$
and for an indecomposable cover $\zeta'$ of $G'$
we define
in the same way
the invariants
$\eta'_{I'} \colon \fprod[G']{i \in I'} \onto G'$,
$I'_{A'}$,
$H'_{A'}$,
$B_{A'}$,
$f'_{A'}$,
$V'_{A'}$,
$S'_{A'}$,
$F'_{A'}$,
and
$I'_{\zeta'}$.

We address the question,
when there is
$\theta \colon \fprod{i \in I} H_i \onto \fprod{i \in I'} H'_i$
such that
$\eta'_{I'} \circ \theta = \pi \circ \eta_I$.
But first we assume that
$G' = G$ and $\pi = \id_G$.

\begin{Theorem}\label{epimorphism G' is G}
Assume that
$G' = G$ and $\pi = \id_G$.
There is 
$\theta \colon \fprod{i \in I} H_i \onto \fprod{i \in I'} H'_i$
such that
$\eta'_{I'} \circ \theta = \eta_{I}$
if and only if
\begin{itemize}
\item[(a)]
$|I'_\zeta| \le |I_\zeta|$,
for every $\zeta \in \Lambda_\na(G)$;
and
\item[(b)]
$\Img S'_A \subseteq \Img S_A$
and 
$\dim \Ker S'_A \le \dim \Ker S_A$,
for every $A \in \Lambda_\ab(G)$.
\end{itemize}
\end{Theorem}

\begin{proof}
Put $H = \fprod{i \in I} H_i$,
$H' = \fprod{i \in I'} H'_i$,
$K = \Ker \eta_I$,
and
$K' = \Ker \eta'_{I'}$.

\subdemoinfo{Part A}
{If $\theta$ exists, then (a) holds.}
Fix $\zeta \in \Lambda_\na(G)$.
Let $i \in I'$.
Then
$\eta'_i \circ (\pr_{I',i} \circ \theta) = 
\eta'_{I'} \circ \theta = \eta_I$.
By Theorem~\ref{indecomposable quotients of a fiber product pi is id}(a')
there is a unique $j \in I_\zeta$
such that
$\Ker \pr_{I,j} = \Ker (\pr_{I',i} \circ \theta)
= \theta^{-1}(\Ker \pr_{I',i})$.
This defines a map
$i \mapsto j$ 
from $I'_\zeta$ to $I_\zeta$.
This map is injective:
If $i_1,i_2 \in I'_\zeta$ map onto the same $j \in I_\zeta$,
then
$\theta^{-1}(\Ker \pr_{I',i_1}) = \theta^{-1}(\Ker \pr_{I',i_2})$,
hence
$\Ker \pr_{I',i_1} = \Ker \pr_{I',i_2}$,
whence
$i_1 = i_2$,
by
Theorem~\ref{indecomposable quotients of a fiber product pi is id}(a').
Thus 
$|I'_\zeta| \le |I_\zeta|$.

\subdemoinfo{Part B}
{If $\theta$ exists, then (b) holds.}
Fix $A \in \Lambda_\ab(G)$
and let
$F_A = \End_G(A)$.

For  every $i \in I'_A$,\  
$\eps_i := \pr_{I',i} \circ \theta \colon H \onto H'_i$
satisfies
$\eta'_i \circ \eps_i = \eta_I$.
By Theorem~\ref{indecomposable quotients of a fiber product pi is id}(b'),
there is an epimorphism $\eps'_i \colon H_A \onto H'_i$
such that
$\eps'_i \circ \pr_{I,I_A} = \eps_i = \pr_{I',i} \circ \theta$.
We have
$$
\eta_{I_A} \circ \pr_{I,I_A} =
\eta_I =
\eta'_{I'} \circ \theta =
\eta'_i \circ \pr_{I',i} \circ \theta =
\eta'_i \circ \eps'_i \circ \pr_{I,I_A},
$$
hence 
$
\eta_{I_A} = \eta'_i \circ \eps'_i
$.
Thus
$\eta'_i \circ \eps'_i$
does not depend on $i$.
By the universal property of the fiber product
$\fprod{i \in I'_A} H'_i$, \
$(\eps'_i)_{i \in I'_A}$
defines a homomorphism
$\theta_A \colon H_A \to H'_A$
such that
$\pr_{I'_A,i} \circ \theta_A = \eps'_i$
for all $i \in I'_A$.
So
$$
\eta'_{I'_A} \circ \theta_A = 
\eta'_i \circ \pr_{I'_A,i} \circ \theta_A = 
\eta'_i \circ \eps'_i =
\eta_{I_A}.
$$
We claim that
\begin{equation}\label{7}
\theta_A \circ \pr_{I, I_A} =
\pr_{I',I'_A} \circ  \theta.
\end{equation}
Indeed,
both sides are maps
$H \to H'_A = \fprod{i \in I'_A} H'_i$,
and for every $i \in I'_A$
$$
\pr_{I'_A,i} \circ \theta_A \circ \pr_{I, I_A} =
\eps'_i \circ \pr_{I, I_A} =
\eps_i =
\pr_{I',i} \circ  \theta =
\pr_{I'_A,i} \circ \pr_{I',I'_A} \circ \theta,
$$
hence 
by the universal property of the fiber product
both sides of \eqref{7} are equal.
In particular, as
$\pr_{I',I'_A}, \theta$ are surjective, so is $\theta_A$.
We have
$$
\eta_{I_A} \circ \pr_{I, I_A} =
\eta_I =
\eta'_{I'} \circ \theta =
\eta'_{I'_A} \circ \pr_{I',I'_A} \circ \theta =
\eta'_{I'_A} \circ \theta_A \circ \pr_{I, I_A},
$$
hence
\begin{equation}
\eta_{I_A} = \eta'_{I'_A} \circ \theta_A.
\end{equation}

\begin{equation*}
\xymatrix@=10pt{
&&& H_{A} \arr[rrrr]^{\eps'_i}
&&&& H'_i \arr@/^2pc/[lllddddd]^{\eta'_i}
\\
&& H_{A} \dotarr[rrrr]^{\theta_A} \arr@{=}[ru]
\arr'[ldd][lldddd]^(0.3){\eta_{I_A}}
&&&& H'_A \arr[ru]^(.25){\pr_{I'_A,i}} \arr[lldddd]^{\eta'_{I'_A}}
\\
 \\
H \arr[rrrr]^{\theta} \arr[dd]_{\eta_I} \arr[rruu]^{\pr_{I,I_A}}
&&&& H' \arr[dd]_{\eta'_{I'}} \arr[rruu]^{\pr_{I',I'_A}}
 \\
 \\
G \arr@{=}[rrrr]
&&&& G 
}
\end{equation*}
Thus,
replacing $\theta$ by $\theta_A$,
we may assume that $I = I_A$ and $I' = I'_A$.

By the equivalence of categories $\T^2$ and $\calH^2$
of Proposition~\ref{equivalence of categories},
$\theta$ induces an $F_A$-linear map
$T_A \colon V'_A \to V_A$
such that
$S_A \circ T_A = S'_A$.
As $\theta$ is surjective,
$T_A$ is injective
(Remark~\ref{inj surj dual}).
The last equation implies that
$\Img S'_A = S_A(\Img T_A) \subseteq \Img S_A$
and 
$T_A(\Ker S'_A)  \subseteq \Ker S_A$.
This gives (b).

\subdemoinfo{Part C}
{If (a) and (b) hold, then $\theta$ exists.}
Let $\zeta \in \Lambda_\na(G)$.
By (a) we may assume that
$I'_\zeta \subseteq I_\zeta$.
As $\eta_i \isom_G \eta'_i \isom_G \zeta$,
for every $i \in I'_\zeta$,
we may assume that
$H'_i = H_i$ and $\eta'_i = \eta_i$,
for every $i \in I'_\zeta$.

Let $A \in \Lambda_\ab(G)$.
By (b)
there is a basis $\{f_i\}_{i \in J_A}$ of $\Img S_A$
such that,
for some $J'_A \subseteq J_A$,
the set
$\{f_i\}_{i \in J'_A}$ is a basis of $\Img S'_A$.
By Remark~\ref{basis choice} we may assume that
$I_A = J_A \dotcup R_A$,
where
$|R_A| = \dim \Ker S_A$
and the extension $\eta_i$ of $G$ by $A$
corresponds to
$
\begin{cases}
f_i & \textnormal{if $i \in J_A$}\\
0 & \textnormal{if $i \in R_A$}
\end{cases}
\in  H^2(G, A)
$.
Similarly we may assume that
$I'_A = J'_A \dotcup R'_A$,
where
$|R'_A| = \dim \Ker S'_A$
and the extension $\eta'_i$ of $G$ by $A$
corresponds to
$
\begin{cases}
f_i & \textnormal{if $i \in J'_A$}\\
0 & \textnormal{if $i \in R'_A$}
\end{cases}
\in  H^2(G, A)
$.
As $|R'_A| \le |R_A|$,
without loss of generality $R'_A \subseteq R_A$.
Thus $I'_A \subseteq I_A$
and $H'_i = H_i$ and $\eta'_i = \eta_i$,
for every $i \in I'_A$.

After these adjustments,
it follows that $I' \subseteq I$
and $H'_i = H_i$ and $\eta'_i = \eta_i$,
for every $i \in I'$.
Therefore $\theta := \pr_{I,I'} \colon H \onto H'$
satisfies
$\eta'_{I'} \circ \theta = \eta_{I}$.
\end{proof}

\begin{Corollary}\label{isomorphism}
Assume that
$G' = G$ and $\pi = \id_G$.
There is an isomorphism
$\theta \colon \fprod{i \in I} H_i \to \fprod{i \in I'} H'_i$
such that $\eta'_{I'} \circ \theta = \eta_{I}$
if and only if
\begin{itemize}
\item[(a)]
$|I_\zeta| = |I'_\zeta|$,
for every $\zeta \in \Lambda_\na(G)$;
and

\item[(b)]
$\Img S_A = \Img S'_A$
and 
$\dim \Ker S_A = \dim \Ker S'_A$,
for every $A \in \Lambda_\ab(G)$.
\end{itemize}
\end{Corollary}

\begin{proof}
Conditions (a), (b) are necessary by
Theorem~\ref{epimorphism G' is G}.
If they hold,
in Part C of the proof of Theorem~\ref{epimorphism G' is G}
we may assume that
$I' = I$
and $H'_i = H_i$ and $\eta'_i = \eta_i$,
for every $i \in I$.
Therefore
$\fprod{i \in I} H_i = \fprod{i \in I'} H'_i$.
\end{proof}

\begin{Corollary}\label{from epi to iso}
Assume that
$G' = G$ and $\pi = \id_G$.
Suppose there are epimorphisms
$\rho \colon \fprod{i \in I} H_i \onto \fprod{i \in I'} H'_i$
and
$\rho' \colon \fprod{i \in I'} H'_i \onto \fprod{i \in I} H_i$
such that
$\eta'_{I'} \circ \rho = \eta_{I}$
and
$\eta_{I} \circ \rho' = \eta'_{I'}$.
Then there is an isomorphism
$\theta \colon \fprod{i \in I} H_i \to \fprod{i \in I'} H'_i$
such that $\eta_{I} \circ \theta = \eta'_{I'}$.
\end{Corollary}

\begin{proof}
By Theorem~\ref{epimorphism G' is G},
the existence of $\rho$ implies
conditions (a), (b) of Theorem~\ref{epimorphism G' is G}.
Similarly,
the existence of $\rho'$ implies
the symmetrical conditions.
Thus conditions (a), (b) of Corollary~\ref{isomorphism}
hold.
The conclusion now follows by Corollary~\ref{isomorphism}.
\end{proof}

Now we approach the general case,
of an epimorphism $\pi \colon G \onto G'$.

Every $A' \in \Lambda_\ab(G')$
can be also viewed as a $G$-module
(with $G$ acting on $A'$ via $\pi$),
denoted $\Inf_\pi(A')$,
in $\Lambda_\ab(G)$.
Let
$\Inf_{\pi,A'}$ denote the inflation
$H^2(G',A') \to H^2(G, \Inf_\pi(A'))$.
Every indecomposable cover
$\zeta' \colon H' \onto G'$
of $G'$
defines an indecomposable cover
$\zeta := \Inf_\pi(\zeta') \colon H \onto G$
of $G$.
If $\Ker \zeta' = A' \in \Lambda_\ab(G')$,
then
$\Ker \zeta = \Inf_\pi(A') \in \Lambda_\ab(G)$.

\begin{Theorem}\label{epimorphism}
There is $\theta$
such that the following square is semi-cartesian
\begin{equation}\label{start gen}
\xymatrix@=25pt{
\fprod[G]{i \in I} H_i 
\arr[rr]_{\theta} \arr[d]^{\eta_I}
&& \fprod[G']{i \in I'} H'_i \arr[d]^{\eta'_{I'}}
\\
G \arr[rr]^{\pi} && G' \\
}
\end{equation}
if and only if
\begin{itemize}
\item[(a)]
$|I'_{\zeta'}| \le |I_{\zeta}|$,
for every $\zeta' \in \Lambda_\na(G')$
and $\zeta := \Inf_\pi(\zeta') \in \Lambda_\na(G)$;
and
\item[(b)]
$\Img(\Inf_{\pi,A'} \circ S'_{A'})  \subseteq \Img S_{A}$
and 
$\dim \Ker(\Inf_{\pi,A'} \circ S'_{A'}) \le \dim \Ker S_{A}$,
for every $A' \in \Lambda_\ab(G')$
and $A := \Inf_\pi(A') \in \Lambda_\ab(G)$.
\end{itemize}
\end{Theorem}

\begin{proof}
Put $H = \fprod[G]{i \in I} H_i$,
$H' = \fprod[G']{i \in I'} H'_i$,
$K = \Ker \eta_I$,
and
$K' = \Ker \eta'_{I'}$.

\subdemoinfo{Part A}
{A special case.}
Let $I = I'$
and,
for every $i \in I$,
let
$H_i = G \times_{G'} H'_i$
(with respect to $\pi, \eta'_i$)
and $\eta_i = \Inf_\pi(\eta'_i) \colon H_i \onto G_i$
be the projection on the first coordinate.
By Remark~\ref{properties of fiber product}(h),
$\fprod[G]{i \in I} H_i$
can be viewed as
$G \times_{G'} (\fprod{i \in I'} H'_i)$
with respect to $\pi$ and $\eta'_{I'}$,
so there is a cartesian square \eqref{start gen}.

For every $\zeta' \in \Lambda_\na(G')$ and $\zeta = \Inf_\pi(\zeta')$
we have
$I_{\zeta} = I'_{\zeta'}$,
hence (a) holds.

Let $A' \in \Lambda_\ab(G')$.
If we replace $I = I'$ by $I_{A'} = I'_{A'}$,
then,
by Remark~\ref{properties of fiber product}(h),
we still get a cartesian square \eqref{start gen},
say, with $\theta_A$ instead of $\theta$.
Then $\theta_A$ maps $K$ isomorphically onto $K'$.
So there is an isomorphism
$\theta_A^* \colon V'_{A'} = (K')^* \to V_A = K^*$
and it follows from the definitions that 
there is a commutative diagram
\begin{equation}\label{V}
\xymatrix@=20pt{
V_A
\arr[d]^{S_A}
&& V'_{A'} \arr[ll]_{\theta_A^*} \arr[d]^{S'_{A'}}
\\
H^2(G,A) && H^2(G',A') \arr[ll]^{\Inf_{\pi,A'}}  \rlap{.} \\
}
\end{equation}
Thus
$\Img(\Inf_{\pi,A'} \circ S'_{A'}) = \Img S_{A}$
and 
$\dim \Ker(\Inf_{\pi,A'} \circ S'_{A'}) = \dim \Ker S_{A}$,
and hence (b) holds.

\subdemoinfo{Part B}
{The general case.}
By Part A
there is
a fiber product
$\eta''_{I''} \colon H'' \onto G$
of indecomposables $(\eta''_i)_{i \in I''}$
with $I'' = I'$ and
the invariant
$S''_A \colon V''_A \to H^2(G,A)$,
for every $A' \in \Lambda_\ab(G')$
and $A := \Inf_\pi(A') \in \Lambda_\ab(G)$,
and there is
a cartesian square
- the right-handed square in \eqref{with inflation} below -
\begin{equation}\label{with inflation}
\xymatrix{
H \arr[r]_{\theta'} \arr@/^1pc/[rr]^{\theta} \arr[d]_{\eta_I} 
& H'' \arr[r]_{\theta''} \arr[d]^{\eta''_{I''}}
& H' \arr[d]^{\eta'_{I'}} 
 \\
G \ar@{=}[r] & G \arr[r]^{\pi}
& G' 
}
\end{equation}
such that
\begin{itemize}
\item[(a')]
$|I'_{\zeta'}| = |I''_{\zeta}|$,
for every $\zeta' \in \Lambda_\na(G')$
and $\zeta := \Inf_\pi(\zeta') \in \Lambda_\na(G)$;
and
\item[(b')]
$\Img(\Inf_{\pi,A} \circ S'_A) = \Img S''_{A}$
and 
$\dim \Ker(\Inf_{\pi,A} \circ S'_A) = \dim \Ker S''_{A}$,
for every $A' \in \Lambda_\ab(G')$
and $A := \Inf_\pi(A') \in \Lambda_\ab(G)$.
\end{itemize}
Using this,
conditions (a) and (b) are equivalent to:
\begin{itemize}
\item[(a'')]
$|I''_\zeta| \le |I_\zeta|$,
for every $\zeta \in \Lambda_\na(G)$;
and
\item[(b'')]
$\Img S''_A \subseteq \Img S_A$
and 
$\dim \Ker S''_A \le \dim \Ker S_A$,
for every $A \in \Lambda_\ab(G)$.
\end{itemize}

Suppose there is $\theta$ 
such that
\eqref{start gen}
is semi-cartesian.
As
$(\theta'', \eta'_{I'}, \eta''_{I''}, \pi)$
is cartesian,
there is a unique homomorphism
$\theta' \colon H \to H''$
such that \eqref{with inflation} commutes.
By Lemma~\ref{expand}(b),
$\theta'$ is an epimorphism.
By Theorem~\ref{epimorphism G' is G}
(a'') and (b'') hold.

Conversely, suppose that
(a'') and (b'') hold.
By Theorem~\ref{epimorphism G' is G}
there is 
$\theta' \colon H \onto H''$
such that
the left-handed square in \eqref{with inflation}
commutes.
Put $\theta = \theta'' \circ \theta'$,
then \eqref{start gen} commutes.
By Lemma~\ref{expand}(a)
it is a semi-cartesian square.
\end{proof}

\begin{Example}\label{gen epi example}
With the notation at the beginning of this section
assume that
$I' \subseteq I$
and,
for every $i \in I'$,
$\eta_i = \Inf_\pi \eta'_i$,
so that
there is
a cartesian square
\begin{equation}\label{pullback}
\xymatrix{
H_i \arr[r]_{\rho_i} \arr[d]^{\eta_i} & H'_i \arr[d]^{\eta'_i} \\
G \arr[r]^\pi & G' \\
}
\end{equation}
Then
$I'_{\zeta'} \subseteq I_{\zeta}$
for every $\zeta \in \Lambda_\na(G')$,
so condition (a) of Theorem~\ref{epimorphism} is satisfied.
So is condition (b),
for every $A' \in \Lambda_\ab(G')$,
because there is a homomorphism
$\theta_A^*$
such that
\eqref{V} commutes.
Indeed,  $F_A = F_{A'}$,
and
$V_A = \directsum_{i \in I_A} F_A \phi_i$,
where
$S_A(\phi_i) = \eta_i$, for every $i$;
similarly,
$V_A' = \directsum_{i \in I'_{A'}} F_A \phi'_i$,
where
$S'_{A'}(\phi_i) = \eta'_i$, for every $i$,
so
$\phi'_i \mapsto \phi_i$,
for $i \in I'$,
defines the required homomorphism
$\theta_A^*$.

Thus, by Theorem~\ref{epimorphism},
there is
$\theta \colon \fprod{i \in I} H_i \to \fprod[G']{i \in I'} H'_i$
such that
\eqref{start gen} is a semi-cartesian square.

An explicit definition of $\theta$ is:
$(h_i)_{i \in I} \mapsto (\rho_i(h_i))_{i \in I'}$.
This map is well defined by the commutativity of \eqref{pullback}
and extends $\pi$,
that is, the following (not necessarily cartesian) diagram commutes:
\begin{equation}\label{gen pullback}
\xymatrix{
\fprod{i \in I} H_i \arr[r]_{\theta} \arr[d]^{\eta_I}
& \fprod[G']{i \in I'} H'_i \arr[d]^{\eta'_{I'}} \\
G \arr[r]^\pi & G' \\
}
\end{equation}
Moreover, $\theta$ is an epimorphism.
Indeed,
let $(h'_i)_{i \in I'} \in \fprod[G']{i \in I'} H'_i$.
Choose $j \in I'$ and $h_j \in H_j$ such that
$\rho_j(h_j) = h'_j$.
Put $g = \eta_j(h_j)$
and
$g' = \eta'_j(h'_j)$.
Then $\pi(g) = g'$.
As \eqref{pullback} is cartesian,
for every $i \in I'$
there is $h_i \in H_i$
such that
$\eta_i(h_i) = g$ and $\rho_i(h_i) = h'_i$.
For every $i \in I \smallsetminus I'$,
there is $h_i \in H_i$
such that
$\eta_i(h_i) = g$.
Thus
$(h_i)_{i \in I} \in \fprod{i \in I} H_i$
and
$(h_i)_{i \in I} \mapsto (h'_i)_{i \in I'}$.

Notice that both $\theta$ and $\eta_I$
factor through
$\pr_{I,I'} \colon \fprod{i \in I} H_i \onto \fprod{i \in I'} H_i$,
so \eqref{gen pullback} remains commutative, if
$\fprod{i \in I} H_i$ 
is replaced by
$\fprod{i \in I'} H_i$;
then $I' = I$.
If $I' = I$,
then \eqref{gen pullback} is a cartesian square.

If $I' = \emptyset$,
then
$\fprod{i \in I'} H'_i = G$  and $\eta'_{I'}$ is the identity,
so $\theta = \pi \circ \eta_I$,
whence
\begin{multline}
\Ker \theta = \eta_I^{-1}(\Ker \pi) =
\{(h_i)_{i \in I} \in \fprod{i \in I} H_i \st
(\forall i \in I) \eta_i(h_i) \in \Ker \pi\} =
\\
\fprod[\Ker \pi]{i \in I} \eta_i^{-1}(\Ker \pi).
\end{multline}

If $I' \ne \emptyset$,

$$
\Ker \theta =
\big(\fprod{i \in I'} \Ker \rho_i \big)
\times_G
\big(\fprod{i \in I \smallsetminus I'} H_i\big)
$$
and $\eta_{I'} \colon \fprod{i \in I'} H_i$ maps
$\fprod{i \in I'} \Ker \rho_i$
isomorphically onto $\Ker \pi$.
\end{Example}

Conversely,
every epimorphism from
$\theta \colon \fprod{i \in I} H_i \onto H'$
onto some profinite group
is essentially of this form.
I.e.,
we show below that
if $L \normal \fprod{i \in I} H_i$,
then
$(\fprod{i \in I} H_i)/L$
is a fiber product of indecomposables over
$G' = G/\eta_I(L)$.

\begin{Proposition}\label{gen epi}
Let
\eqref{another}
be a semi-cartesian square.

\noindent\begin{minipage}{.495\linewidth}
\begin{equation}\label{another}
\xymatrix{
\fprod{i \in I} H_i \arr[r]_{\rho} \arr[d]^{\eta_I}
& H' \arr[d]^{\eta'} \\
G \arr[r]^\pi & G' \\
}
\end{equation}
\end{minipage}%
\begin{minipage}{.495\linewidth}
\begin{equation}\label{bar}
\xymatrix{
\Hbar_i \arr[r]_{\rhobar_i} \arr[d]^{\etabar_i} & H'_i \arr[d]^{\eta'_i} \\
G \arr[r]^\pi & G' \\
}
\end{equation}
\end{minipage}
\newline
Then there are families
$\bar\calH = (\etabar_i \colon \Hbar_i \onto G)_{i \in I}$
and
$\calH' = \{\eta'_i \colon H'_i \onto G')_{i \in I'}$
of indecomposable epimorphisms
such that
$I' \subseteq I$
and
$\etabar_i = \Inf_\pi(\eta'_i)$
with a cartesian square \eqref{bar}
for every $i \in I'$,
and there are isomorphisms $\omega$ and $\omega'$
such that the following diagram commutes.
\begin{equation}\label{triangle}
\xymatrix@=13pt{ 
& \fprod{i \in I} H_i
\arr[rr]_(.6){\rho}
\arr'[d][dd]^(.4){\eta_I}
\arr[ld]_{\omega}
&& H'
\arr[dd]^(.55){\eta'}
\arr[ld]_(.55){\omega'}
\\
\fprod{i \in I} \Hbar_i
\arr[rr]^(.65){\theta}
\arr[dr]_(.65){\etabar_I}
&& \fprod[G']{i \in I'} H'_i
\arr[dr]^(.55){\eta'_{I'}} 
\\
& G
\arr[rr]^{\pi}
&& G'
}
\end{equation}
Here
$\etabar_I$
and $\eta'_{I'}$
are the structure maps of the corresponding fiber products
and $\theta$ is the map
$(h_i)_{i \in I} \mapsto (\rhobar_i(h_i))_{i \in I'}$.
\end{Proposition}

\begin{proof}
Let $L = \Ker \rho$
and
$K = \Ker \eta_I$.
By
Corollary~\ref{isom with subgroup gen},
there is a family
$\bar\calH = (\etabar_i \colon \Hbar_i \onto G)_{i \in I}$
of indecomposable epimorphisms
and an isomorphism
$\omega \colon 
\fprod{i \in I} H_i \to \fprod{i \in I} \Hbar_i$
such that
$\etabar_I \circ \omega = \eta_I$
and $\omega(K \cap L) = \prod_{i \in I \smallsetminus I'} \Kbar_i$,
for some $I' \subseteq I$.
Here
$\Kbar_j$ is the kernel of the projection
$\fprod{i \in I} \Hbar_i \onto
\fprod{j \in I \smallsetminus\{j\}} \Hbar_j$,
for every $j \in I$.
So, replacing
$\fprod{i \in I} H_i$ by $\fprod{i \in I} \Hbar_i$
and $\rho$ by $\omega^{-1} \circ \rho$
we may assume that
$K \cap L = \prod_{j \in I \smallsetminus I'} K_j$,
where
$K_j = \Ker \pr_{I, I \smallsetminus \{j\}}$.
Thus
$K \cap L = \Ker \pr_{I, I'}$.
Of course, then
\eqref{pullback} replaces \eqref{bar}.

We may now replace
$\fprod{i \in I} H_i$ by $\fprod{i \in I'} H_i$
and $\rho$ by the induced map
\break
${\fprod{i \in I'} H_i \onto H'}$.
The new square, that replaces \eqref{another},
is still semi-cartesian, by Lemma~\ref{expand}(a).
So we may assume that
$K \cap L = 1$
and
$I' = I$.
Put $H = \fprod{i \in I} H_i$.

As $K \cap L = 1$,
\eqref{another}
is cartesian
and $\rho$ is injective on
$K = \prod_{i \in I} K_i$.
Thus
$\rho(K) = \prod_{i \in I} \rho(K_i)$.
For every $i \in I$
let $H'_i = H'/\prod_{j \ne i} \rho(K_j)$,
let $p_i \colon H' \onto H'_i$ be the quotient map,
and let $\eta'_i \colon H'_i \onto G'$
be the map induced from $\eta'$.
By Lemma~\ref{image and preimage}(a),
$\Ker \eta'_i \isom K_i$ is a minimal normal subgroup of $H'_i$.
Hence $\eta'_i$ is indecomposable.
By Lemma~\ref{kernels of a fiber product}.
$H'$ is the fiber product of $H'_i$ over $G'$,
with $\eta' = \eta'_I$
and $p_i$ is the coordinate projection,
for every $i \in I$.

Let $\rho_i \colon H_i \onto H'_i$
be the map induced from $\rho \colon H \onto H'$.
Then \eqref{tower} below commutes.

\noindent
\begin{minipage}{.395\linewidth}
\begin{equation}\label{tower}
\xymatrix@=24pt{
H \arr[r]_{\rho} \arr[d]^{\pr_{I,i}}
\arr@/_1pc/[dd]_{\eta_{I}}
& H' \arr[d]_{p_i} \arr@/^1pc/[dd]^{\eta'}
\\
H_i \arr[r]^{\rho_i} \arr[d]^{\eta_i} & H'_i \arr[d]_{\eta'_i}
\\
G \arr[r]^{\pi} & G'
}
\end{equation}
\end{minipage}
\begin{minipage}{.595\linewidth}
\begin{equation}\label{replace by bar}
\xymatrix{
\fprod{i \in I} \Hbar_i \arr[r]_{\omega^{-1}} \arr[d]_{\pr_{I,i}} 
& \fprod{i \in I} H_i \arr[r]_{\rho} \arr[d]^{\pr_{I',i}}
& H' \arr[d]^{p_i}
 \\
\Hbar_i \arr[r]^{\omega_i^{-1}} & H_i \arr[r]^{\rho_i}
& H'_i
}
\end{equation}
\end{minipage}
\newline
Its upper square is cartesian,
because
$\rho$ maps
$\Ker \pr_{I,i} = \prod_{j \ne i} K_j$
isomorphically onto
$\Ker p_i = \prod_{j \ne i} \rho(K_j)$.
The concatenation of the uppper and the lower square
is the semi-cartesian square \eqref{another},
which is cartesian, because $K \cap L = 1$.
Thus the lower square,
that is \eqref{pullback},
is cartesian by Remark~\ref{square trivialities}(a).
The commutativity of the upper square gives that
$\rho$ is
$(h_i)_{i \in I} \mapsto (\rho_i(h_i))_{i \in I}$.

If $H_i = \Hbar_i$ and $\rho_i = \rhobar_i$, for every $i \in I'$,
we are done.
If not,
\eqref{bar}
and
the lower square in \eqref{tower}
are both the fiber products of $\pi$ and $\eta'_i$.
By the uniqueness of the fiber product,
there is an isomorphism
$\omega_i \colon H_i \to \Hbar_i$
such that
$\etabar_i \circ \omega_i = \eta_i$
and
$\rhobar_i \circ \omega_i = \rho_i$,
for every $i \in I'$.
Let $\Hbar = H_i$ and $\omega_i = \id_{H_i}$
for $i \in I \smallsetminus I'$
Then
$(\omega_i)_{i \in I}$
define an isomorphism
$\omega \colon \fprod{i \in I} H_i \to \fprod{i \in I} \Hbar_i$
such that
$\etabar_I \circ \omega = \eta_I$
and
\eqref{replace by bar} above
commutes
for every $i \in I'$.
Replace
$\fprod{i \in I} H_i$
by
$\fprod{i \in I} \Hbar_i$
and $\rho$ by $\rho \circ \omega^{-1}$.
Then $\rho_i \colon H_i \onto H'_i$
in \eqref{tower}
is replaced by $\rhobar_i$,
as required.
\end{proof}

\section{Fundaments of epimorphisms}\label{fundaments}

We now come to the `fundamental' notion of this work:

\begin{Definition}\label{fundament}
Let $\pi \colon H \onto G$ be an epimorphism of profinite groups.
The \textbf{fundament kernel} of $\pi$,
denoted $\calM(\pi)$,
is the intersection of all
$N \in \calN_{\Ker \pi}(H)$,
i.e., $N$ that satisfy the equivalent conditions of
Lemma~\ref{trivial}
with $M = \Ker \pi$.
Clearly $\calM(\pi) \normal H$ and $\calM(\pi) \le \Ker \pi$.
An epimorphism
$\bar\pi \colon \Hbar \onto G$
is the \textbf{fundament} of $\pi$,
if it is isomorphic to the quotient map
$\bar\pi \colon H/\calM(\pi) \onto G$,
that is,
if there is
$\rho \colon H \onto \Hbar$
such that
$\bar\pi \circ \rho = \pi$
and
$\Ker \rho = \calM(\pi)$;
we then also say that
$\bar\pi$ is
\textbf{the fundament of $\pi$ by $\rho$}.
We say that $\pi$ is \textbf{fundamental},
if $\bar\pi = \pi$, that is,
if $\calM(\pi) = 1$.
Notice that the fundament of an epimorphism
is itself fundamental.
\end{Definition}

\begin{Remark}
Here is the Galois-theoretic interpretation of Definition~\ref{fundament}:

Let $L' \subseteq L$ be
Galois extensions of a field $K$,
and let 
$\pi \colon \Gal(L/K) \onto \Gal(L'/K)$ 
be the restriction map.
Then the fundament of $\pi$ is
the restriction
$\bar\pi \colon \Gal(\Lbar/K) \onto \Gal(L'/K)$,
where $\Lbar$ is the compositum of all
Galois extensions of $K$ contained in $L$, properly containing $L'$,
and minimal with respect to these conditions.
The fundament kernel of $\pi$ is $\Gal(L/\Lbar)$.
\end{Remark}

The \textbf{Melnikov group}
of a profinite group $H$
is the intersection of 
the maximal open normal subgroups of $H$
(\cite[p.~637]{FJ}).
Thus
the Melnikov group of $H$
is the fundament kernel of $H \onto 1$.

Using Notation~\ref{Lambda} we have:

\begin{Definition}\label{characteristic}
Let $\pi \colon H \onto G$
be an epimorphism.

(a)
For every
$\zeta \in \Lambda_\na(G)$ let
$\calC_\zeta = 
\{N \normal H \st N \le \Ker \pi,\ (H/N \onto G) \isom_G \zeta\}$.
We call $|\calC_\zeta|$
the \textbf{$\zeta$-multiplicity}
of $\pi$,
denoted $\mult_\zeta(\pi)$.
\newline
Put
$N_\zeta = \bigcap_{N \in \calC_\zeta} N$
and
$H_\zeta = H/N_\zeta$.

(b)
For every
$A \in \Lambda_\ab(G)$ let
$\calC_A = \{N \normal H \st N \le \Ker \pi,\ H/N \isom_G A\}$.
Put
$N_A = \bigcap_{N \in \calC_A} N$
and
$H_A = H/N_A$.
\newline
Let $\pibar_A \colon H_A \onto G$ be the induced map.
Then $B_A = \Ker \pibar_A$
is a $G$-submodule of
$\prod_{N \in \calC_\zeta} H/N$,
and hence,
by Remark~\ref{A-generated rudiments}(b),
is an $A$-generated $G$-module.
Let
$f_A \in H^2(G, B_A)$
be the cocycle corresponding to the extension
$1 \to B_A \to H_A \to G \to 1$
and let $(B_A^*, S_A) = X^2(B_A, f_A)$.
So
$B_A^*$ is a vector space over $F_A = \End_G(A)$
and 
$S_A \colon B_A^* \to H^2(G,A)$
is the $F_A$-linear map
$\phi \mapsto (f_A)_*\phi$.
\newline
We call $\dim \Ker S_A$ the \textbf{$A$-multiplicity}
of $\pi$,
denoted $\mult_A(\pi)$.
We call
$\Img S_A$
the \textbf{$A$-support}
of $\pi$,
denoted $\supp_A(\pi)$.

(c)
Furthermore,
let $(v_i)_{i \in I}$
be a family of vectors in a vector space $V$ over a field $F$.
Denoting by $F^I$ the vector space with basis $I$,
let $T \colon F^I \to V$
be the unique linear map that maps $i$ onto $v_i$,
for all $i \in I$.
Then 
$\dim \Ker T$ is the \textbf{relation dimension}
of $(v_i)_{i \in I}$.

Thus,
if $(\phi_i)_{i \in I}$ is a basis of $B_A^*$ in (b),
then $i \mapsto \phi_i$ is an isomorphism
$F^I \isom_G B_A^*$,
and hence
$\supp_A(\pi) = \Span((\eta_i)_{i \in I})$
and
$\mult_A(\pi)$
is the relation dimension of
$(\eta_i)_{i \in I}$,
where
$\eta_i = (f_A)_*\phi_i$,
for every $i \in I$.
\end{Definition}

\begin{Example}\label{fprod mult and supp}
Let $H$ be the fiber product of a family
$(\eta_I \colon H_i \onto G)_{i \in I}$
of indecomposable epimorphisms
and let $\pi = \eta_I \colon H \onto G$.
Fix
$\zeta \in \Lambda_\na(G)$
and
$A \in \Lambda_\na(G)$,
and let
$I_\zeta = 
\{i \in I \st \eta_i \isom_G \zeta\}$
and
$I_A = 
\{i \in I \st \Ker \eta_i \isom_G A\}$.

Let $i \in I$. Then
$\Ker \pr_{I,i} \in \calN_{\Ker \pi}(H)$
and
$H/\Ker \pr_{I,i} \onto H/\Ker \pi$
is
$\eta_i \colon H_i \onto G$.
Therefore,
$\Ker \pr_{I,i} \in \calC_\zeta$
if and only if
$i \in I_\zeta$.
Similarly,
$\Ker \pr_{I,i} \in \calC_A$
if and only if
$i \in I_A$.

As
$\calM(\pi) \le \bigcap_{j \in I} \Ker \pr_{I,j} = 1$,\
$\pi$ is fundamental.

By Theorem~\ref
{indecomposable quotients of a fiber product pi is id}(a'),
every $N \in \calC_\zeta$ is $\Ker \pr_{I,i}$ for some $i \in I$,
hence
$\mult_\zeta(\pi) = |\calC_\zeta| = |I_\zeta|$.

We have
$\bigcap_{i \in I_A}\Ker \pr_{I,i} = \Ker \pr_{I,I_A}$.
By Theorem~\ref
{indecomposable quotients of a fiber product pi is id}(b'),
every $N \in \calC_A$ contains $\Ker \pr_{I,I_A}$,
hence
$N_A = \bigcap_{N \in \calC_A} N = \Ker \pr_{I,I_A}$,
whence
$H/N_A = \fprod{i \in I_A} H_i$.
Therefore
$B_A = \Ker(\eta_{I_A} \colon \fprod{i \in I_A} H_i \onto G)$.
For every $i \in I_A$ we have
$\eta_i \circ \pr_{I_A,i} = \eta_{I_A}$,
hence by diagram~\eqref{induced} in
Construction~\ref{two functors}
we have
$S_A(\pr_{I_A,i}) = \eta_i$.
But
$(\pr_{I_A,i})_{i \in I_A}$
is a basis of $B_A^*$,
hence
$\supp_A(\pi) = \Span(\eta_i \st i \in I_A)$,
and
$\mult_A(\pi) = \Ker S_A$
is the relation dimension of
$(\eta_i)_{i \in I_A}$.
\end{Example}

\begin{Proposition}\label{fundamental is fiber product}
Let $\pi \colon H \onto G$ be
a fundamental epimorphism.
Then
$
H \isom_G \fprod{\lambda \in \Lambda(G)} H_\lambda
\ \big[ =
(\fprod{\zeta \in \Lambda_\na(G)} H_\zeta)
\times_G
(\fprod{A \in \Lambda_\na(G)} H_A)
\big]
$.
Moreover:
\begin{itemize}
\item[(a)]
For every $\zeta \in \Lambda_\na(G)$, 
$H_\zeta \isom_G \fprod{i \in I_\zeta} H_i$,
where
$|I_\zeta| = |\calC_\zeta|$
and
${\eta_i \colon H_i \onto G}$
is isomorphic to
$\zeta$,
for every $i \in I_\zeta$.
\item[(b)]
For every $A \in \Lambda_\ab(G)$
let
$\{\phi_i\}_{i \in I_A}$ be a basis of $B_A^*$
over $F_A = \End_G(A)$.
Then
$H_A \isom_G \fprod{i \in I_A} H_i$,
where
$\eta_i \colon H_i \onto G$
is
$(\phi_i)_* f_A = S_A(\phi_i) \in H^2(G,A)$,
for every $i \in I_A$.
\end{itemize}

Thus
$H$ is the fiber product
of the family
$(\eta_i \colon H_i \onto G)_{i \in I}$
of indecomposable extensions of $G$,
where
$I =
\bigdotcup_{\zeta \in \Lambda_\na(G)}
I_\zeta
\dotcup
\bigdotcup_{A \in \Lambda_\ab(G)}
I_A
$.
\end{Proposition}

\begin{proof}
Let $\zeta \in \Lambda_\na(G)$.
Write
$\calC_\zeta = \{N_i\}_{i \in I_\zeta}$.
For every $i \in I_\zeta$
put $H_i = H/N_i$
and let $\eta_i \colon H_i \onto G$ be the map induced from
$\pi \colon H \onto G$.
Then
$\eta_i \isom_G \zeta$
and
$H_\zeta = H/(\bigcap_{i \in I_\zeta} N_i)$.
Therefore
$H_\zeta \to \fprod{i \in I_\zeta} H_i$
is an isomorphism by 
Proposition~\ref{decompose}.

Let $A \in \Lambda_\ab(G)$.
Then $B_A$ is an $A$-generated $G$-module.
For every $i \in I_A$
let $\eta_i \colon H_i \onto G$
be the extension represented by
$(\phi_i)_* f_A = S_A(\phi_i) \in H^2(G,A)$.
Then 
$H_A \to \fprod{i \in I_A} H_i$
is an isomorphism
by Corollary~\ref{A-gen ext is fiber product}.

By
Remark~\ref{properties of fiber product}(d),
$\fprod{\lambda \in \Lambda(G)} \kern-3pt H_\lambda \isom_G \fprod{i \in I} H_i$,
where
$I =
\bigdotcup_{\zeta \in \Lambda_\na(G)}
I_\zeta
\dotcup
\bigdotcup_{A \in \Lambda_\ab(G)}
I_A
$.
Thus,
$H \to  \kern-3pt \fprod{\lambda \in \Lambda(G)} \kern-5pt H_\lambda$
is an isomorphism
by Proposition~\ref{decompose}.
\end{proof}

From Proposition~\ref{fundamental is fiber product}
and Example~\ref{fprod mult and supp}
we get:

\begin{Corollary}\label{fiber product is fundamental}
An epimorphism $\pi \colon H \onto G$
is fundamental
if and only if it is isomorphic to
the structure map
of the fiber product
of indecomposable extensions of $G$.
\end{Corollary}

\begin{Theorem}\label{fundamental mult supp}
Let $\pi \colon H \onto G$ be fundamental
and
let
$\eta_I \colon \fprod{i \in I} H_i \onto G$
be the fiber product
of a family
of indecomposable extensions
$(\eta_i \colon H_i \onto G)_{i \in I}$
of $G$.
Denote
$I_\zeta =\{i \in I \st \eta_i \isom_G \zeta\}$,
for every $\zeta \in \Lambda_\na(G)$,
and
$I_A =\{i \in I \st \Ker \eta_i \isom_G A\}$,
for every $A \in \Lambda_\ab(G)$.
Then $\pi \isom_G \eta_I$
if and only if
\begin{itemize}
\item[(a)]
$|I_\zeta| = \mult_\zeta(\pi)$,
for every $\zeta \in \Lambda_\na(G)$,
\item[(b)]
$\Span(\eta_i)_{i \in I_A} = \supp_A(\pi)$,
for every $A \in \Lambda_\ab(G)$,
and
\item[(c)]
$\mult_A(\pi)$
is the relation dimension of
$(\eta_i)_{i \in I_A}$,
for every $A \in \Lambda_\ab(G)$.
\end{itemize}
\end{Theorem}

\begin{proof}
If $\pi \isom_G \eta_I$,
then
(a), (b), (c) hold
by Example~\ref{fprod mult and supp}.

Conversely,
assume that
(a), (b), (c) hold.
By
Corollary~\ref{fiber product is fundamental}
we may assume that
$H$ is the fiber product of a family
$(\eta'_i \colon H'_i \onto G)_{i \in I'}$
of indecomposable extensions of $G$.
If we define
$I'_\zeta$ and $I'_A$
for $\pi$, analogously to
$I_\zeta$ and $I_A$ for $\eta_I$,
then, again 
by Example~\ref{fprod mult and supp},
conditions (a), (b) of 
Corollary~\ref{isomorphism} are satisfied.
Thus
$\pi \isom_G \eta_I$.
\end{proof}

For an indecomposable cover $\eta$ of $G$
and a cardinality $\kappa$ 
let
$\eta^{(\kappa)}$
denote the fiber product of $\kappa$ copies of $\eta$.

\begin{Example}\label{multiple of one}
Let $\eta$ be an indecomposable cover of $G$
and $\kappa$ a cardinality.

(a)
If $\kappa = 0$,
then $\eta^{(\kappa)}$ is the identity of $G$,
hence
$\mult_\lambda(\eta^{(\kappa)}) = 0$
for all $\lambda \in \Lambda(G)$
and
$\supp_A(\eta^{(\kappa)}) = 0$
for all $A \in \Lambda_\ab(G)$.

Assume that
$\kappa > 0$.
Put $C = \Ker \eta$.

(b)
If $C$ is non-abelian,
then
for all $\lambda \in \Lambda(G)$
and all $A \in \Lambda_\ab(G)$
\begin{equation*}
\mult_\lambda(\eta^{(\kappa)}) =
\begin{cases}
\kappa & \textnormal{if $\lambda = \eta$}
\\
0 & \textnormal{if $\lambda \ne \eta$}
\end{cases},
\qquad
\supp_A(\eta^{(\kappa)}) = 0.
\end{equation*}
(This is true even if $\kappa = 0$, by (a).)

(c)
If $C$ is abelian,
then
$\mult_\lambda(\eta^{(\kappa)}) = 0$
for every $\lambda \in \Lambda_\na(G)$.
By Example~\ref{ex1},
$\dim B_C^* = \kappa$
and
$\Img S_C = \Span(\eta)$,
while
$\dim B_A^* = 0$
for all $C \ne A \in \Lambda_\ab(G)$.
We have
$\dim B_A^* = \dim \Img S_{A} + \dim \Ker S_{A}$,
for all $A \in \Lambda_\ab(G)$.
Thus for all $\lambda \in \Lambda(G)$
and all $A \in \Lambda_\ab(G)$
\begin{align*}
\mult_\lambda(\eta^{(\kappa)}) &=
\begin{cases}
0 & \textnormal{if $\lambda \ne C$}
\\
\kappa & \textnormal{if $\lambda = C$ and $\eta$ splits}
\\
\kappa - 1  & \textnormal{if $\lambda = C$, $\eta$ does not split,}
\end{cases}
\\
\supp_A(\eta^{(\kappa)}) &=
\begin{cases}
0 & \textnormal{if $A \ne C$}
\\
\Span(\eta) & \textnormal{if $A = C$.}
\end{cases}
\end{align*}
\end{Example}

\begin{Corollary}\label{reinterprete}
Let $\tau \colon H \onto G$
and
$\tau' \colon H' \onto G'$
be fundamental.

Let $\pi \colon G \onto G'$
be an epimorphism.
There is $\theta \colon H \onto H'$
such that
\begin{equation}\label{two taus}
\xymatrix@=20pt{
H 
\arr[rr]_{\theta} \arr[d]^{\tau}
&& H' \arr[d]^{\tau'}
\\
G \arr[rr]^{\pi} && G' \\
}
\end{equation}
is semi-cartesian
if and only if
\begin{itemize}
\item[(a)]
$\mult_{\zeta'}(\tau') \le \mult_\zeta(\tau)$,
for every
$\zeta' \in \Lambda_\na(G')$
and
$\zeta := \Inf_\pi(\zeta') \in \Lambda_\na(G)$;
\item[(b)]
$\Inf_{\pi,A'}(\supp_{A'}(\tau')) \subseteq \supp_A(\tau)$
for every
$A' \in \Lambda_\ab(G)$ and
\break
$A := \Inf_\pi(A') \in \Lambda_\ab(G)$;
and
\item[(c)]
$\nu_{\pi,A'} + \mult_{A'}(\tau') \le \mult_A(\tau)$,
for every
$A' \in \Lambda_\ab(G)$ and
\break
$A := \Inf_\pi(A') \in \Lambda_\ab(G)$;
here
$\nu_{\pi,A'}$ is the nullity of the restriction
$r \colon \supp_{A'}(\tau') \to \supp_A(\tau)$
of 
$\Inf_{\pi,A'}$ to $\supp_{A'}(\tau')$.
\end{itemize}
\end{Corollary}

\begin{proof}
This is reinterpretation of Theorem~\ref{epimorphism}:
By Proposition~\ref{fundamental is fiber product}
we may assume that
$\tau = \eta_I$
and
$\tau' = \eta'_{I'}$,
in the notation of Theorem~\ref{epimorphism}
(fixed at the beginning of Section~\ref{epi}),
and then
$\Img S_{A} = \supp_A(\tau)$,
and
$$
\Img(\Inf_{\pi,A'} \circ S'_{A'}) = 
\Inf_{\pi,A'}(\Img S'_{A'}) = 
\Inf_{\pi,A'}(\supp_{A'}(\tau')).
$$
Moreover,
if
$\Inf_{\pi,A'}(\supp_{A'}(\tau')) \subseteq \supp_A(\tau)$,
then 
$$
\Ker(\Inf_{\pi,A'} \circ S'_{A'}) = 
\Ker(r \circ S'_{A'}) = 
(S'_{A'})^{-1}( \Ker r),
$$
hence
$$
\dim \Ker(\Inf_{\pi,A'} \circ S'_{A'}) = \dim \Ker S'_{A'}
+ \dim \Ker r
= \mult_{A'}(\tau') + \nu_{\pi,A'}.
$$
\end{proof}

\begin{Corollary}\label{reinterprete G is G'}
Let $\tau \colon H \onto G$
and
$\tau' \colon H' \onto G$
be fundamental.
\begin{itemize}
\item[(1)]
$\tau' \mle \tau$
if and only if
\begin{itemize}
\item[(1a)]
$\mult_\lambda(\tau') \le \mult_\lambda(\tau)$,
for every $\lambda \in \Lambda(G)$;
and
\item[(1b)]
$\supp_A(\tau') \subseteq \supp_A(\tau)$,
for every $A \in \Lambda_\ab(G)$.
\end{itemize}

\item[(2)]
$\tau' \isom_G \tau$
if and only if
\begin{itemize}
\item[(2a)]
$\mult_\lambda(\tau') = \mult_\lambda(\tau)$,
for every $\lambda \in \Lambda(G)$;
and
\item[(2b)]
$\supp_A(\tau') = \supp_A(\tau)$,
for every $A \in \Lambda_\ab(G)$
\end{itemize}
\end{itemize}
\end{Corollary}

\begin{proof}
(1) is reinterpretation of Theorem~\ref{epimorphism G' is G}.

(2) is reinterpretation of Corollary~\ref{isomorphism}.
\end{proof}

Applying the criterion of Corollary~\ref{reinterprete G is G'}(1)
to Example~\ref{multiple of one} we get

\begin{Corollary}\label{dominates multiple of one}
Let $\pi \colon H \onto G$ be fundamental.
Let $\eta$ be an indecomposable cover of $G$
and $\kappa \ge 0$ a cardinality.
Put $C = \Ker \eta$.
\begin{itemize}
\item[(a)]
If $C$ is not abelian, then
$\eta^{(\kappa)} \mle \pi$ 
if and only if
$\kappa \le \mult_\eta(\pi)$.

\item[(b)]
If $C \in \Lambda_\ab(G)$, then
$\eta^{(\kappa)} \mle \pi$ 
if and only if
\begin{equation*}
\begin{cases}
\kappa \le 
\mult_C(\pi)&
\textnormal{if $\eta \in \supp_C(\pi)$ splits}
\\
\kappa \le 
\mult_C(\pi) + 1 &
\textnormal{if $\eta \in \supp_C(\pi)$ does not split}
\\
\kappa = 
0 &
\textnormal{if $\eta \notin \supp_C(\pi)$}
\end{cases}
\end{equation*}
\end{itemize}
\end{Corollary}

\begin{Corollary}\label{Max}
Let $\pi \colon H \onto G$
and
$\pi' \colon H' \onto G'$
be fundamental.
Then 
$\pi' \mle \pi$
if and only if
$\eta^{(\kappa)} \mle \pi'
\implies
\eta^{(\kappa)} \mle \pi$,
for
every indecomposable cover $\eta$ of $G$
and
every $\kappa$.
\end{Corollary}

\begin{proof}
If
$\pi' \mle \pi$,
then the condition holds
by the transitivity of $\mle$.

Conversely, assume that
the condition holds.

\subdemoinfo{Claim A}
{Let $C \in \Lambda_\ab(G)$.
Then
$\supp_C(\pi') \subseteq \supp_C(\pi)$.}

Indeed,
let $\eta \in H^2(G,C) \smallsetminus \supp_C(\pi)$;
in particular, $\eta \ne 0$, so $\eta$ does not split.
By Corollary~\ref{dominates multiple of one}(b),
$\eta = \eta^{(1)} \not\mle \pi$,
hence
$\eta^{(1)} \not\mle \pi'$.
As
$1 \le \mult_C(\pi') + 1$,
this implies,
again by Corollary~\ref{dominates multiple of one}(b),
that
$\eta \notin \supp_C(\pi')$.

\subdemoinfo{Claim B}
{Let $\eta \in \Lambda_\na(G)$.
Then
$\mult_\eta(\pi') \le \mult_\eta(\pi)$.}
Put
$\kappa = \mult_\eta(\pi')$.
By Corollary~\ref{dominates multiple of one}(a),
$\eta^{(\kappa)} \mle \pi'$,
hence
$\eta^{(\kappa)} \mle \pi$,
whence, again
by Corollary~\ref{dominates multiple of one}(a),
$\kappa \le \mult_\eta(\pi)$.

\subdemoinfo{Claim C}
{Let $C \in \Lambda_\ab(G)$.
Then
$\mult_C(\pi') \le \mult_C(\pi)$.}
Put
$\kappa = \mult_C(\pi')$ and
let $\eta$ be the split extension of $G$ with kernel $C$.
By Corollary~\ref{dominates multiple of one}(b),
$\eta^{(\kappa)} \mle \pi'$,
hence
$\eta^{(\kappa)} \mle \pi$,
whence, again
by Corollary~\ref{dominates multiple of one}(b),
$\kappa \le \mult_C(\pi)$.

Thus
$\pi' \mle \pi$
by Corollary~\ref{reinterprete G is G'}(1).
\end{proof}

\begin{Corollary}\label{fundament normal subgr}
Let
$\pi \colon H \onto G$
be a fundamental epimorphism
and let  $L$ be a normal subgroup of $H$,
contained in $\Ker \pi$.
Then there is an isomorphism
$\theta \colon H \onto \fprod{i \in I} H_i$
onto a fiber product of
indecomposable covers of $G$
such that
$\theta(L) = \prod_{i \in I'} K_i$,
for some
$I' \subseteq I$,
where
$K_i = \Ker \pr_{I, I \smallsetminus \{i\}}$,
for every $i \in I$.
\end{Corollary}

\begin{proof}
By Proposition~\ref{fundamental is fiber product}
we may assume that
$\pi \colon H \onto G$
is the fiber product of indecomposable covers of $G$.
So the assertion follows by Corollary~\ref{isom with subgroup gen}.
\end{proof}

\begin{Corollary}\label{fundament quotient}
In a semi-cartesian square \eqref{two taus},
if $\pi$ is fundamental, then so is $\pi'$.
\end{Corollary}

\begin{proof}
By Proposition~\ref{fundamental is fiber product},
$H$ is the fiber product of indecomposable extensions of $G$.
As \eqref{two taus} is semi-cartesian,
$\Ker \theta_0 = \pi(\Ker \theta)$.
Hence, by Proposition~\ref{gen epi},
$H'$ is the fiber product of indecomposable extensions of $G'$.
By Corollary~\ref{fiber product is fundamental},\
$\pi'$ is fundamental.
\end{proof}

\begin{Proposition}\label{fundament characterization}
Let
$\rho \colon H \onto G_1$,
$\pibar \colon G_1 \onto G$,
and $\pi = \pibar \circ \rho \colon H \onto G$
be epimorphisms.
Assume that
$\pibar$ is fundamental.
Then
$\pibar$ is the fundament of $\pi$ by $\rho$
if and only if
there is no
commutative diagram
\begin{equation}\label{no double}
\xymatrix{
H \arr[r]_{\rho} \arr[d]^{\gamma} \arr[rd]^{\pi} & G_1 \arr[d]_{\pibar}
\\
H_{0} \arr[r]^{\eta_{0}} & G
}
\end{equation}
with
a semi-cartesian square
and
indecomposable $\eta_{0}$.
\end{Proposition}

\begin{proof}
Suppose that
there is a semi-cartesian square
\eqref{no double}
with $\eta_{0}$ indecomposable.
Then
$(\Ker \gamma) (\Ker \rho) = \Ker \pi$.
As $\pi = \eta_0 \circ \gamma$
and $\eta_0$ is not an isomorphism,
$\Ker \pi$ strictly contains $\Ker \gamma$.
Hence
$\Ker \rho \not\le \Ker \gamma$.
But $\Ker \gamma \in \calN_{\Ker \pi}(H)$,
hence
$\Ker \rho \ne \calM(\pi)$.

Conversely,
suppose that
$\Ker \rho \ne \calM(\pi)$.

As $\pibar \colon G_1 \onto G$ is fundamental,
the intersection of the groups in
$$
\{\Nbar \normal G_1 \st \Nbar \le \Ker \pibar,\ G_1/\Nbar \onto G
\textnormal{ is indecomposable}\}
$$
is $1$.
Taking the inverse image of this set under $\rho$ we get that
the intersection of the groups in
$$
\calF =
\{N \normal H \st \Ker \rho \le N \le \Ker \pi,\ H/N \onto G
\textnormal{ is indecomposable}\}
$$
is $\Ker \rho$.
But $\calM(\pi)$ is the intersection of the groups in
$$
\calF' =
\{N \normal H \st N \le \Ker \pi,\ H/N \onto G
\textnormal{ is indecomposable}\}.
$$
So
$\calF \ne \calF'$.
Clearly 
$\calF \subseteq \calF'$.
Thus there is
$N \in \calF'$
such that
$\Ker \rho \not\le N$,
that is,
$N$ is properly contained in
$N \Ker \rho$.
As $N$ is a maximal normal subgroup of $H$
contained in $\Ker \pi$,
we have 
$N \Ker \rho = \Ker \pi$.

Put
$H_0 = H/N$,
let $\gamma \colon H \onto H_0$ be the quotient map,
and let
$\eta_0 \colon H_0 \onto G$
be the epimorphism induced from $\pi$.
Then \eqref{no double} is a semi-cartesian square.
\end{proof}

\section{Fundament series}\label{fundament series}

\begin{Definition}\label{def fundament series}
The \textbf{fundament kernel series} of 
an epimorphism $\pi \colon H \onto G$
of profinite groups
is the descending sequence
of normal subgroups of $H$
\begin{equation*}
M_0(\pi) \ge M_1(\pi) \ge M_2(\pi) \ge \cdots
\end{equation*}
inductively defined by
$M_0(\pi) = \Ker \pi$,
and $M_i(\pi)$
is the fundament kernel of the quotient map
$H \onto H/M_{i-1}(\pi)$,
for each $i \ge 1$,
that is,
$M_i(\pi)$ is the intersection of all
$N \in \calN_{M_{i-1}(\pi)}(H)$.
In particular,
$M_1(H) = \calM(\pi)$.

Put $G_i = H/M_i(\pi)$ for each $i \ge 0$.
The sequence of the induced quotient maps
\begin{equation*}
\xymatrix{
\cdots \arr[r]^{\pi_3}
& G_2 \arr[r]^{\pi_2}
& G_1 \arr[r]^{\pi_1}
& G_0\rlap{ $ = G$}
}
\end{equation*}
is the \textbf{fundament series} of $\pi$.
We call $\pi_k$
the \textbf{$k$-th fundament} of $\pi$.
Thus
$\pi_k$ is the fundament of the quotient map
$H \onto G_{k-1}$.
\end{Definition}

\begin{Proposition}\label{invlim}
We have $\bigcap_{i=0}^\infty M_i(\pi) = 1$
and hence $H = \varprojlim_{i} G_i$.
\end{Proposition}

\begin{proof}
Put $M_i = M_i(\pi)$ for each $i$.

The second assertion follows from the first one.
To prove the first one,
it suffices to show, for every open $U \normal H$,
that there is $i$ such that $M_i \le U$,
that is, $M_i U = U$.
Thus, as $H/U$ is finite,
it suffices to show that
if $M_{i-1}U \ne U$, 
then
$M_iU < M_{i-1}U$. 

So put $M = M_{i-1}$ and assume $MU \ne U$.
Let $\rho \colon H \onto \Hbar \defeq H/U$ be the quotient map.
By assumption,
$\Mbar \defeq \rho(M) \ne 1$,
hence, by Remark~\ref{rem: minimal normal subgroup}(d),
there is $\Nbar \in \calN_\Mbar(\Hbar)$.
By Lemma~\ref{image and preimage}(c)
there is $N \in \calN_M(H)$
such that $\rho(N) = \Nbar$.

By the inductive definition of $M_i$
we have $M_i \le N$.
It follows that
$M_iU /U = \rho(M_i) \le \Nbar < \Mbar = M_{i-1}U/U$,
that is,
$M_i U < M_{i-1}U$.
\end{proof}

\begin{Lemma}\label{fundaments image}
Let
\begin{equation}\label{ex2 square}
\xymatrix@=20pt{
H 
\arr[rr]_{\theta} \arr[d]^{\pi}
&& H' \arr[d]^{\pi'}
\\
G \arr[rr]^{\theta_0} && G' \\
}
\end{equation}
be a commutative diagram.
Then 
$\theta(M_k(\pi)) \le M_k(\pi')$,
for every $k \ge 0$.
If \eqref{ex2 square} is semi-cartesian,
then
$\theta(M_k(\pi)) = M_k(\pi')$,
for every $k \ge 0$.
\end{Lemma}

\begin{proof}
By induction on $k$.
For $k = 0$ 
the assertions follow from definitions.
Let $k \ge 1$.
Put
$M = M_{k-1}(\pi)$
and
$M' = M_{k-1}(\pi')$,
and
assume that
$\theta(M) \le M'$.
Then
$M \le \theta^{-1}(M')$.

By definition,
$M_k(\pi) = \bigcap_{N \in \calN_{M}(H)} N$
and
$M_k(\pi') = \bigcap_{N' \in \calN_{M'}(H')} N'$.

For
$N' \in \calN_{M'}(H')$
let
$N = \theta^{-1}(N')$.
Then
$N \in \calN_{\theta^{-1}(M')}(H)$.
If $M \le N$, then
$M_k(\pi) \le M \le N$.
If
$M \not\le N$,
then,
as $N$ is a maximal normal subgroup of $H$
contained in $\theta^{-1}(M')$,
we have
$N M = \theta^{-1}(M')$.
Thus 
$N \in \calN_{N M}(H)$,
whence
$M \cap N \in \calN_M(H)$.
Therefore
$M_k(\pi) \le M \cap N \le N$.

Thus,
$M_k(\pi) \le 
\bigcap_{N' \in \calN_{M'}(H')}
\theta^{-1}(N')
 =\theta^{-1}(M_k(\pi'))$.
It follows that
$\theta(M_k(\pi)) \le M_k(\pi')$.

Now assume that
$\theta(M) = M'$.
Let $L' = \theta(M_k(\pi))$.
Then, by the above,
$L' \le M_k(\pi') \le M'$
and $L' \normal H'$,
and the square
\begin{equation*}
\xymatrix@=12pt{
H/M_k(\pi) \arr[r] \arr[d] & H'/L' \arr[d]
\\
H/M \arr[r] & H'/M'
}
\end{equation*}
with
vertical quotient maps
and horizontal maps induced from $\theta$,
is semi-cartesian.
Hence, by Corollary~\ref{fundament quotient},
$H'/L' \onto H'/M'$
is fundamental.
Therefore
$L'$ is the intersection of groups in 
$\calN_{M'}(H')$,
and hence 
$L' \ge M_k(\pi')$.
Thus $L' = M_k(\pi')$.
\end{proof}

When, instead of normal subgroups,
we consider the quotient maps modulo these subgroups,
we get the following equivalent formulation:

\begin{Theorem}\label{two fundament series}
Let \eqref{ex2 square}
be a commutative diagram.
Let
\begin{equation*}
\xymatrix{
\cdots \arr[r]^{\pi_3}
& G_2 \arr[r]^{\pi_2}
& G_1 \arr[r]^{\pi_1}
& G_0 = G 
}
\textnormal{ and }
\xymatrix{
\cdots \arr[r]^{\pi'_3}
& G'_2 \arr[r]^{\pi'_2}
& G'_1 \arr[r]^{\pi'_1}
& G'_0 = G'
}
\end{equation*}
be the fundament series
of $\pi \colon H \onto G$ and $\pi' \colon H' \onto G'$, respectively.
Then
\begin{itemize}
\item[(a)]
there is a commutative diagram
\begin{equation}\label{two series diagram}
\xymatrix{
H \arr@/^1.1pc/[rr]_{\rho_k}
\arr@/^1.7pc/[rrr]^(.7){\rho_{k-1}}
\arr[d]_{\theta}
& \cdots \arr[r]_{\pi_{k+1}}
& G_{k} \arr[r]^{\pi_{k}}
\arr[d]_{\theta_{k}}
& G_{k-1} \arr[r]_{\pi_{k-1}}
\arr[d]_{\theta_{k-1}}
& \cdots \arr[r]_{\pi_{3}}
& G_{2} \arr[r]_{\pi_{2}}
\arr[d]_{\theta_{2}}
& G_{1} \arr[r]_{\pi_{1}}
\arr[d]_{\theta_{1}}
& G_{0}
\arr[d]_{\theta_{0}}
\\
H' \arr@/_1.1pc/[rr]^{\rho'_k}
\arr@/_1.7pc/[rrr]_(.7){\rho'_{k-1}}
& \cdots \arr[r]^{\pi'_{k+1}}
& G'_{k} \arr[r]^{\pi'_{k}}
& G'_{k-1} \arr[r]^{\pi'_{k-1}}
& \cdots \arr[r]^{\pi'_{3}}
& G'_{2} \arr[r]^{\pi'_{2}}
& G'_{1} \arr[r]^{\pi'_{1}}
& G'_{0}
}
\end{equation}
\item[(b)]
If 
\eqref{ex2 square}
is a semi-cartesian square,
then so are the squares
$(\rho_k, \theta_k, \theta, \rho'_k)$
and
$(\pi_{k+1}, \theta_k, \theta_{k+1}, \pi'_{k+1})$
for every $k \ge 0$.
\item[(c)]
If $\theta_0$ and $\theta$
are isomorphisms,
then so is $\theta_k$,
for every $k \ge 0$.
\end{itemize}
\end{Theorem}

\begin{proof}
(a)
For every $k \ge 0$ we have
$G_k = H/M_k(\pi)$
and
$G'_k = H'/M_k(\pi')$,
and $\pi_k$, $\pi'_k$, $\rho_k$ and $\rho'_k$ 
are the quotient maps.
As $\theta(M_k(\pi)) \le M_k(\pi')$,
\break
$\theta$ induces the $\theta_k$ such that the diagram commutes.

(b)
If
$(\pi, \theta_0, \theta, \pi')$
is semi-cartesian,
then so are
$(\rho_k, \theta_k, \theta, \rho'_k)$,
by the second assertion of Lemmma~\ref{fundaments image}.
By Lemmma~\ref{two semi squares},
so are
$(\pi_{k+1}, \theta_k, \theta_{k+1}, \pi'_{k+1})$.

(c)
By (b),
$(\rho_k, \theta_k, \theta, \rho'_k)$
is semi-cartesian.
As
$\theta$ is an isomorphism,
so is $\theta_k$.
\end{proof}

For the next proposition we use the following setup.
Let
\begin{equation}\label{seq fund epi}
\xymatrix{
\cdots \arr[r]^{\pi_3}
& G_2 \arr[r]^{\pi_2}
& G_1 \arr[r]^{\pi_1}
& G_0\rlap{ $ = G$}
}
\end{equation}
be a sequence of fundamental epimorphisms.
Let $H = \varprojlim_{i} G_i$,
and let $\rho_i \colon H \onto G_i$ be the maps 
of the inverse limit,
so that
$\rho_{i-1} = \pi_i \circ \rho_i$,
for every $i \ge 1$.
Also, for all $\ell > k$, let
$\pi_{\ell,k} = \pi_{k+1} \circ \cdots \circ \pi_{\ell}
\colon G_\ell \onto G_k$.
Put $\pi = \rho_0 \colon H \onto G$.
We then ask,
when is \eqref{seq fund epi}
the fundament series of $\pi$?

\begin{Proposition}\label{fundament series characterization}
In the above setup
the following are equivalent:
\begin{itemize}
\item[(a)]
for no $k \ge 1$ there is a semi-cartesian square
\eqref{large}
with indecomposable $\eta_{k-1}$.
\item[(b)]
for no $\ell \ge k$ there is a semi-cartesian square
\eqref{middle}
with indecomposable $\eta_{k-1}$.
\item[(c)]
for no $k \ge 1$ there is a semi-cartesian square
\eqref{small}
with indecomposable $\eta_{k-1}$.
\item[(d)]
\eqref{seq fund epi}
is the fundament series of $\pi$.
\end{itemize}
\noindent\begin{minipage}{.315\linewidth}
\begin{equation}\label{large}
\xymatrix{
H \arr[r]^{\rho_k} \arr[d]_{\sigma} \dotarr[rd]^{\rho_{k-1}} 
& G_{k} \arr[d]^{\pi_k}
\\
H_{k-1} \arr[r]_{\eta_{k-1}} & G_{k-1}
}
\end{equation}
\end{minipage}%
\begin{minipage}{.355\linewidth}
\begin{equation}\label{middle}
\xymatrix{
H \arr[r]^{\rho_\ell} \arr[d]_{\sigma} \dotarr[rd]^{\rho_{k-1}} 
& G_{\ell} \arr[d]^{\pi_{\ell,k-1}}
\\
H_{k-1} \arr[r]_{\eta_{k-1}} & G_{k-1}
}
\end{equation}
\end{minipage}
\begin{minipage}{.315\linewidth}
\begin{equation}\label{small}
\xymatrix{
G_{k+1} \arr[r]^{\pi_{k+1}} \arr[d]_{\tau} & G_k \arr[d]^{\pi_k}
\\
H_{k-1} \arr[r]_{\eta_{k-1}} & G_{k-1}
}
\end{equation}
\end{minipage}
\end{Proposition}

\begin{proof} 
It will be simpler to work with the negations
($\neg$a), ($\neg$b), ($\neg$c), ($\neg$d)
of 
(a), (b), (c), (d),
respectively.

($\neg$a) $\implies$ ($\neg$b):
Put $\ell = k$.

($\neg$b) $\implies$ ($\neg$a):
As $\pi_{\ell,k-1} = \pi_k \circ \pi_{\ell,k}$,
this follows by
Lemma~\ref{two semi squares}(a).

($\neg$a) $\implies$ ($\neg$c):
Suppose 
there is a semi-cartesian square
\eqref{large}
with indecomposable $\eta_{k-1}$.

As $\Ker \eta_{k-1}$ is finite,
that is,
$1$ is open in $\Ker \eta_{k-1}$,
we get that
$\Ker \sigma = \sigma^{-1}(1)$
is open in
$\Ker \rho_{k-1} = \sigma^{-1}(\Ker \eta_{k-1})$.
But
$\bigcap_m \Ker \rho_m = 1$,
hence
there is $m > k$ such that
$\Ker \rho_m \le \Ker \sigma$.
Thus there is an epimorphism
$\delta_{k-1} \colon G_m \onto H_{k-1}$
such that
$\sigma = \delta_{k-1} \circ \rho_m$,
whence
$\eta_{k-1} \circ \delta_{k-1} = \pi_{m,k-1}$.

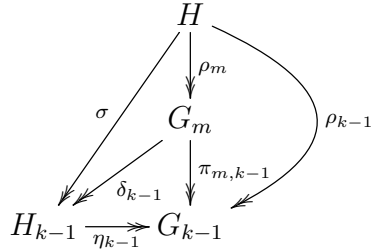
\begin{equation*}
\xymatrix{
& H \arr[ldd]_{\sigma} \arr[d]^{\rho_m}
\arr@/^4pc/[dd]^{\rho_{k-1}} 
\\
& G_m \arr[d]
^{\pi_{m,k-1}}
\arr[ld]^{\delta_{k-1}}
\\
H_{k-1} \arr[r]_{\eta_{k-1}} & G_{k-1} 
}
\end{equation*}

Form the following diagram
(without the dotted arrows)
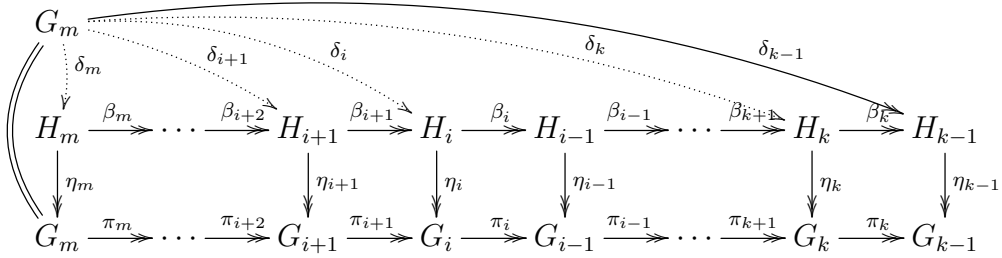
\begin{equation*}
\xymatrix{
G_m \arr@/_1.5pc/@{=}[dd]
\arr@/^2.0pc/[drrrrrrr]^(.8){\delta_{k-1}} 
\ar@/^1.4pc/@{.>}[drrrrrr]^(.7){\delta_{k}} 
\ar@/^1.2pc/@{.>}[drrr]^(.7){\delta_{i}} 
\ar@/^.5pc/@{.>}[drr]^(.6){\delta_{i+1}} 
\ar@/^.3pc/@{.>}[d]^(.4){\delta_{m}} 
\\
H_m \arr[r]^{\beta_m}
\arr[d]^{\eta_{m}}
& \cdots \arr[r]^{\beta_{i+2}}
& H_{i+1} \arr[r]^{\beta_{i+1}}
\arr[d]^{\eta_{i+1}}
& H_{i} \arr[r]^{\beta_{i}}
\arr[d]^{\eta_{i}}
& H_{i-1} \arr[r]^{\beta_{i-1}}
\arr[d]^{\eta_{i-1}}
& \cdots \arr[r]^{\beta_{k+1}}
& H_{k} \arr[r]^{\beta_{k}}
\arr[d]^{\eta_{k}}
& H_{k-1}
\arr[d]^{\eta_{k-1}}
\\
G_m \arr[r]^{\pi_m}
& \cdots \arr[r]^{\pi_{i+2}}
& G_{i+1} \arr[r]^{\pi_{i+1}}
& G_{i} \arr[r]^{\pi_{i}}
& G_{i-1} \arr[r]^{\pi_{i-1}}
& \cdots \arr[r]^{\pi_{k+1}}
& G_{k} \arr[r]^{\pi_{k}}
& G_{k-1}
}
\end{equation*}
in which,
recursively,
$H_{i}$ is the fiber product
$G_{i} \times_{G_{i-1}} H_{i-1}$
with respect to
$\pi_{i}$ and $\eta_{i-1}$,
for every $k \le i \le m$.
The universal property of a fiber product
gives a homomorphism $\delta_{i}$,
recursively for every $k-1 \le i \le m$,
such that the above diagram commutes.

The cartesian squares in the diagram
with sides
$\pi_{k},\ldots, \pi_{m}$
cannot be all compact.
Indeed, if all of them were compact,
then, recursively,
$\delta_{k-1},\ldots, \delta_{m}$
would be epimorphisms.
In particular,
$\delta_{m}$ would be an epimorphism,
whence $\eta_m$
would be an isomorphism,
a contradiction, since, recursively,
the $\eta_i$ are indecomposable
and hence not isomorphisms.

Let $k \le i < m$
such that
the square with side $\pi_{i+1}$ is not compact.
By Proposition~\ref{indecomposable cartesian square}
there is
$\gamma_i \colon G_{i+1} \onto H_i$
such that
$\pi_{i+1} = \eta_i \circ \gamma_i$.
Put
$\tau = \beta_i \circ \gamma_i \colon G_{i+1} \onto H_{i-1}$.
This gives diagram \eqref{small},
with $i$ instead of $k$.
As $\gamma_i$ is surjective,
by Lemma~\ref{expand}(a)
this square is semi-cartesian.

($\neg$c) $\implies$ ($\neg$a):
Diagram
\eqref{small} gives diagram
\eqref{large},
with $\sigma \defeq \tau \circ \rho_{k+1}$.
It is semi-cartesian by Lemma~\ref{expand}(a).

(d) $\Leftrightarrow$ (a):
By definition,
(d) means that
$\pi_k$ is the fundament of $\rho_{k-1}$ by $\rho_k$,
for every $k \ge 1$.
By Lemma~\ref{fundament characterization}
this is equivalent to (a).
\end{proof}


\begin{thebibliography}{BS99}


\bibitem[BHH]{BHH}
L. Bary-Soroker, D. Haran, D. Harbater,
\emph{Permanence criteria for semi-free profinite groups},
Mathematische Annalen \textbf{348}, 539--563 (2010).

\bibitem[Br]{Br}
K.S.~Brown.
\emph{Cohomology of groups},
GTM \textbf{87},
Springer-Verlag 1982.

\bibitem[Ch]{Ch}
Z.~Chatzidakis.
\emph{Model Theory of profinite groups having the Iwasawa property},
Illinois Journal of Mathematics \textbf{42}, 70--96 (1998).

\bibitem[FH]{FH}
S.~Fried and D.~Haran.
\emph{Quasi-formations},
Israel Journal of Mathematics \textbf{229}, 193--217 (2019).



\bibitem[H2]{embed}
D.~Haran.
\emph{On the uniqueness of the smallest embedding cover},
\texttt{arXiv}, 2025.

\bibitem[H3]{proj-free}
D.~Haran.
\emph{Fundaments of projective and free profinite groups},
in preparation.

\bibitem[DM]{DM}
J.D.~Dixon and B.~Mortimer.
\emph{Permutation Groups},
GTM \textbf{163},
Springer-Verlag, 1996.

\bibitem[FJ]{FJ}
M.D.~Fried and M.~Jarden.
\emph{Field Arithmetic},
Ergebnisse der Mathematik III \textbf{11},
3rd edition, revised by M. Jarden.
Springer, 2008.


\bibitem[Ri]{Ribes}
L. Ribes.
\emph{Introduction to Profinite Groups and Galois Cohomology},
Queen's papers in Pure and Applied Mathematics \textbf{24},
Queen's University, Kingston, 1970.

\bibitem[RZ]{RZ}
L. Ribes and P. Zalesskii.
\emph{Profinite Groups},
(2nd edition)
Ergebnisse der Mathematik III \textbf{40}. Springer, 2010.

\bibitem[Sy]{Symonds}
P.~Symonds.
\emph{Permutation complexes for profinite groups},
Comment. Math. Helv. \textbf{82}, 1--37 (2007).


\bibitem[E1]{Efrat1}
I. Efrat.
\emph{The Zassenhaus filtration, Massey products,
and representations of profinite groups},
Adv. Math. 263 (2014), 389--411.

\bibitem[E2]{Efrat2}
I. Efrat.
\emph{The $p$-Zassenhaus filtration of a free profinite group
and shuffle relations},
J. Inst. Math. Jussieu 22 (2023), no. 2, 961--983.

\bibitem[CE]{Efrat3}
M. Chapman, I. Efrat.
\emph{Filtrations of free groups arising from the lower central series},
J. Group Theory 19 (2016), no. 3, 405--433.

\bibitem[EM]{Efrat-Minac}
I. Efrat, J. Min\'a\v c.
\emph{On the descending central sequence of absolute Galois groups}
Amer. J. Math. 133 (2011), no. 6, 1503--1532.

\bibitem[MS]{Minac-Spira}
J. Min\'a\v c, M. Spira.
\emph{Witt rings and Galois groups},
Ann. of Math. (2) 144 (1996), no. 1, 35--60.


\bibitem[MRT]{Minac et al}
J. Min\'a\v c, M. Rogelstad, N.~D. T\^{a}n.
\emph{Dimensions of Zassenhaus filtration subquotients
of some pro-$p$-groups},
Israel J. Math. 212 (2016), no. 2, 825--855.

\bibitem[MRT]{Minac-Tan}
J. Min\'a\v c, N.~D. T\^{a}n.
\emph{Triple Massey products and Galois theory},
J. Eur. Math. Soc. (JEMS) 19 (2017), no. 1, 255--284.
\end{thebibliography}
\end{document}